\documentclass[11pt]{extarticle}
\usepackage[margin=2cm]{geometry}

\usepackage[font=small]{caption}

\usepackage[dvipsnames]{xcolor}
\usepackage[hyperfootnotes=false]{hyperref}
\usepackage{graphicx}
\usepackage{amssymb}
\usepackage{amsmath}
\usepackage{amsthm}
\usepackage{wasysym}
\usepackage{xcolor}
\usepackage{cancel}
\usepackage{textcomp}
\usepackage{yhmath}

\usepackage{hyperref}

\newcommand{\R}{\mathbb{R}}
\newcommand{\intlim}{\int\displaylimits}

\newcommand{\vbf}{\mathbf{v}}

\newcommand{\Vrm}{\mathrm{V}}

\newcommand{\wbf}{\mathbf{w}}

\newcommand{\x}{\mathbf{x}}
\newcommand{\X}{\mathbf{X}}

\newcommand{\urm}{\mathrm{u}}
\newcommand{\Urm}{\mathrm{U}}

\newcommand{\y}{\mathbf{y}}
\newcommand{\Y}{\mathbf{Y}}

\newcommand{\E}{\mathbb{E}}
\newcommand{\PR}{\mathbb{P}}
\newcommand{\Vol}{\mathrm{Vol}}
\newcommand{\Conv}{\mathrm{Conv}}
\newcommand{\Center}{\mathrm{Center}}
\newcommand{\CO}{\mathcal{C}_0}
\newcommand{\C}{\mathcal{C}}

\newcommand{\VO}{\mathcal{V}}

\newcommand{\VOM}{\mathcal{V}_{\blacktriangle}}
\newcommand{\VOMR}{\mathcal{V}_{\blacktriangle}^{\geq t}}
\newcommand{\Sp}{\mathcal{S}}
\newcommand{\Ball}{\mathcal{B}}
\newcommand{\D}{\mathcal{D}}
\newcommand{\sign}{\text{sign}}
\newcommand{\Card}{\#}
\newcommand{\fall}{\text{ for all }}
\newcommand{\Cst}{C}
\newcommand{\Rcirc}{\mathcal{D}}
\newcommand{\Proj}{P}

\newcommand{\m}{\mathrm}

\newcommand{\1}{1}

\newtheorem{theorem}{Theorem}

\newtheorem{proposition}{Proposition}
\newtheorem{lemma}{Lemma}
\newtheorem{remark}{Remark}

\title{Sharp asymptotics for the maximal distance from the boundary to the nucleus of a typical Poisson-Voronoi cell}

\author{P. Calka, C. D'Errico and N. Enriquez}

\date{}

\begin{document}

    \maketitle
\begin{abstract}
    We consider the typical Poisson-Voronoi cell in the Euclidean space ${\mathbb R}^d$ and in particular the maximal distance ${\mathcal D}$ from a vertex of that cell to its nucleus. We provide a sharp asymptotics for the tail distribution of ${\mathcal D}$. As a byproduct, we prove that the extremal index related to the sequence of such distances for all Voronoi cells included in a large box is equal to $(2d)^{-1}$. This confirms a conjecture formulated by Chenavier and Robert in \cite{ChenavierRobert}. The explicit constant appearing in the estimate of the tail probability of ${\mathcal D}$ is proved to be the mean volume of a random simplex formed by uniformly distributed points on the unit sphere conditioned on satisfying some spatial condition.
  \end{abstract}
  ~\\
\footnote{\textit{MSC2020 subject classifications.} Primary 60D05; Secondary 52A22, 60G55}
\footnote{\textit{Key words and phrases.} \hspace*{-.15cm} Poisson-Voronoi tessellation, Poisson point process, random simplex, extremal index, typical cell}
	\tableofcontents
	
	\newpage
	
	\newpage
	\section{Introduction and statement of the result}
	
	The homogeneous Poisson-Voronoi tessellation is certainly the most popular model of random tessellation in stochastic geometry. Indeed, random Voronoi tessellations have been commonly used for representing a large variety of natural phenomena (cell biology, astrophysics, crystal growth, ...) and for this reason, they have been studied intensively starting from the second half of the 20th century, see e.g. \cite[Chapter 9.2]{Stoyan}, \cite[Chapter 1.2]{Okabe} and \cite{Moller}.  Despite this general interest, the explicit laws of most of the geometric functionals of the Poisson-Voronoi tessellation are still unknown, except for the case of the minimal distance from the boundary to the nucleus of the typical cell. In particular, the study of the regularity of the Poisson-Voronoi tessellation raises the question of the frequency of elongated cells, that is, the law of the maximal distance from a vertex of the typical cell to its nucleus. This random variable, denoted by $\D$ in the case of a Poisson-Voronoi tessellation with intensity $1$, will be the central object of this work.
	
	In the two-dimensional case, Calka \cite{Calka2002} obtains a double bound on the tail probability of $\D$ by an indirect method based on the random covering of the circle. The two bounds match, up to a multiplicative constant. With a slightly different point of view, Calka and Chenavier \cite{CalkaChen14} investigate the asymptotics of extremes of geometric characteristics along all $d$-dimensional Poisson-Voronoi cells intersecting a given convex body. In this respect, they obtain in particular the convergence in distribution when the underlying intensity $\gamma$ goes to infinity for the maximum, denoted by $R_{\mbox{\tiny{max}}}(\gamma)$ therein, among all cells in a fixed window of the maximal distance from their boundary to their nucleus. The method used in that paper relies on a reinterpretation of the problem in terms of the covering of the window by Poisson balls in the spirit of Janson \cite{Janson}. This trick makes it possible to completely avoid the classical way for deriving extremal asymptotics. This consists in starting with sharp asymptotics for the tail probability of $\D$, then combining them with mixing properties and especially with the computation of the so-called extremal index, often reinterpreted as the inverse of the mean cardinality of a cluster of exceedances \cite{Leadbetter}. However, the tail probability of $\D$ was largely unknown at the time. As a consequence, the value of the extremal index has stayed an open question, although conjectured to be \((2d)^{-1}\) in the work due to Chenavier and Robert \cite[Section 4.3]{ChenavierRobert}.
	
	In this paper, we aim at deriving sharp asymptotics for the tail probability of $\D$. In turn, this result combined with the extremal convergence from \cite{CalkaChen14} induces the calculation of the extremal index mentioned above and whose existence is guaranteed by \cite[Theorem 4]{ChenavierRobert}. Usually, the calculation of the extremal index together with the tail probability is a prerequisite to the identification of the limit distribution for the maximum of a stationary sequence of real variables. In our case, we benefit from the knowledge of the limit distribution in \cite{CalkaChen14} combined with our Theorem \ref{thm_main_result} to deduce the extremal index, see Theorem \ref{thm_extremal_index}. The determination of the extremal index in stochastic geometry is in general a delicate matter, see e.g. \cite{ChenavierRobert} and \cite{Basrak}. In \cite[section H]{Aldous}, Aldous provides heuristic arguments which give a good intuition of the configuration of the tessellation in the neighborhood of a vertex of a Voronoi cell with large distance to its corresponding nucleus. This should in principle provide information on the mean size of a cluster of exceedances and subsequently on the required extremal index. Nevertheless, it does not lead to the sharp asymptotics for $\D$ that we are seeking.
	
	
	In order to state our main results, we need to introduce some notation. We denote by \(\Phi\) a Poisson point process on the space \(\R^d\) with the Lebesgue measure as the intensity measure. To each point \(\m X\in\Phi\), we associate its Voronoi cell \(\C(\m X,\Phi)\) which is the set of all points in space that are closer to \(\m X\) than to any other point of \(\Phi\). Explicitly,
	\[
	\C(\m X,\Phi) = \lbrace \m y\in\R^d : \|\m X-\m y\|\leq\|\m X'-\m y\| \fall X'\in\Phi \rbrace
	\]
	where \(\|\cdot\|\) is the Euclidean norm. The point \(\m X\) is called \emph{nucleus} of the cell \(\C(\m X,\Phi)\) and the collection of all cells constitutes the \emph{Voronoi tessellation} associated to \(\Phi\). Informally, the \emph{typical cell} of the Voronoi tessellation is the cell chosen uniformly at random within the set of cells intersecting a large window. Equivalently, it is the cell associated to the nucleus \(0\) of the Voronoi tessellation generated by a Poisson point process to which we add the origin, that is \(\Phi\cup\lbrace0\rbrace\), see e.g. \cite[Proposition 4.4.1]{Moller}. Henceforth, when referring to the Poisson-Voronoi typical cell, we mean the cell
	\begin{equation*}
		\CO = \C(0,\Phi\cup\lbrace0\rbrace).
	\end{equation*}
	The maximal distance from a vertex of $\CO$ to its nucleus at $0$ can be rewritten as
	\begin{equation*}
		\D = \min \left\{ \m r \geq 0 \;:\; \CO \subseteq \Ball_{\m r}(0) \right\}
	\end{equation*}
	where $\Ball_{\m r}(\m x)$ denotes the ball centered at $\m x\in\R^d$ and of radius $\m r>0$.

	Our main result, stated in Theorem \ref{thm_main_result}, is the tail of the distribution of \(\Rcirc\). For two positive functions $f$ and $g$, we note $f={\mathcal O}(g)$ when $t\to\infty$ if the ratio $f(t)/g(t)$ is bounded for large $t$. We also denote by $\kappa_d$ the Lebesgue measure of the $d$-dimensional unit ball, given by the formula
	\begin{equation}\label{eq:defkappad}
	    \kappa_d = \frac{\pi^{\frac{d}{2}}}{\Gamma(\frac{d}{2}+1)}.
	\end{equation}
	\begin{theorem}\label{thm_main_result}
		When \(t \rightarrow \infty\),
		\begin{equation}\label{eq_thm_main_result}
			\PR(\D\geq t) = C_d (d\kappa_d)^d t^{d(d-1)} e^{-\kappa_d t^d} + \mathcal{O}( t^{d(d-2)} e^{-\kappa_d t^d})
		\end{equation}
		where 
		\begin{equation}\label{eq:value of Cd}
			C_d = \frac{1}{2^{d-1}\sqrt{\pi}(d-1)!} \frac{\Gamma\left(\frac{d}{2}\right)^d}{\Gamma\left(\frac{d+1}{2}\right)^{d-1}}.
		\end{equation}
	\end{theorem}
    It is in general a delicate task to make explicit the distribution of a size functional of $\CO$. To the best of our knowledge, the most precise asymptotics are given in 
    \cite[Theorem 2]{HugSchneiderAsymptotic} and are logarithmic asymptotics for the tail probability of a series of size functionals of $\CO$ which are continuous, homogeneous and increasing. This includes $\Rcirc$ as well as the volume, the surface area, the diameter and so on.
    Surprisingly, we derive in Theorem \ref{thm_main_result} a more precise estimate in the case of \(\D\). The exponential factor of the right-hand side in \eqref{eq_thm_main_result} is consistent with both the two-dimensional bounds given in \cite{Calka2002} and  \cite[Theorem 2]{HugSchneiderAsymptotic}. This factor corresponds to the probability that a given point at distance $t$ from the origin lies in $\CO$, i.e. that the Poisson point process avoids a given large spherical region of radius $t$ adjacent to the nucleus of the typical cell. However, having $\Rcirc$ larger than $t$ means in fact that there exists such an empty region which can be positioned in any direction around the nucleus and this gives rise to a polynomial factor, see the integral calculation in Section \ref{sec_expectation_of_nr_of_some_distant_vertices}, of which we make explicit the power in \eqref{eq_thm_main_result}. This factor $t^{d(d-1)}$ reflects the fact that $d$ Poisson points together with the origin must lie on a $(d-1)$-dimensional sphere of radius $t$ simultaneously to give birth to a vertex distant from the origin by $t$. Most importantly, we make fully explicit the constant in front of the polynomial and exponential factors. Its computation leads to the resolution of an interesting and independent problem regarding random simplices in the unit ball: 
	\begin{theorem}\label{thm_value_of_C_d}
	 Let us denote by \(\m U_0,...,\m U_d\) $(d+1)$ i.i.d. uniform random variables on the \((d-1)\)-dimensional unit sphere \(\Sp_{\R^d}\) and by \(\Delta_{d}(\m U_0,...,\m U_d)\) the $d$-dimensional volume of the 
	 \(d\)-simplex generated by these random points. Then
		\begin{equation}\label{eq_C_d_as_expectation_intro}
			\E\left[ \Delta_{d}(\m U_0,...,\m U_d) 
			 \1_{0\in\Conv(P_{\m U_0^\perp}(\m U_1),...,P_{\m U_0^\perp}(\m U_d))} \right]=C_d
		\end{equation}
		where \(P_{\m U_0^\perp}(\cdot)\) denotes the orthogonal projection onto the hyperplane \(\m U_0^\perp\) and $C_d$ is given at \eqref{eq:value of Cd}.
	\end{theorem}
	This computation is part of a long series of results concerning average volumes of simplices with random vertices. The first result that would not be limited to specific \(2\) or \(3\)-dimensional cases is due to Kingman (1969) \cite{Kingman}, who found the average volume \(\E\left[\Delta_d\left(\m P_0,...,\m P_d\right)\right]\) when \(\m P_0,...,\m P_d\) are i.i.d. uniform random points in the \(d\)-dimensional unit ball \(\Ball_{\R^d}\). This result was greatly improved by Miles (1971) \cite{Miles}, who found all moments of the volume \(\Delta_d\left(\m P_0,...,\m P_d\right)\) in the case where the points are independent and distributed uniformly either in the unit ball or on the unit sphere, see also \cite{Kabluchko} which collects new results extending \cite{Miles}.
    In our setting, the additional indicator function in the considered expectation in the left-hand side of \eqref{eq_C_d_as_expectation_intro} completely modifies the nature of the problem and consequently, leads us to finding a new method for deriving \eqref{eq_C_d_as_expectation_intro}.
	
	Finally, going back to the previous discussion related to the maximum among all cells in a large window of the maximal distance from their boundary to their nucleus, we use Theorem \ref{thm_main_result} to identify the underlying extremal index in Theorem \ref{thm_extremal_index} below.
	\begin{theorem}\label{thm_extremal_index}
	Consider the set of Poisson-Voronoi cells with nucleus included in $[0,n]^d$. When $n\to\infty$, the extremal index of the sequence of the maximal distances from their boundary to their nucleus is  equal to $\frac1{2d}$.
	\end{theorem}
	
	The key idea leading to Theorem \ref{thm_main_result} and consequently Theorem \ref{thm_extremal_index} consists in reducing the asymptotic estimate of the tail probability of \(\Rcirc\) to the expectation of the number of so-called \emph{pointy} vertices that are far away from the origin, where a \emph{pointy} vertex is a vertex of $\CO$ such that its distance to the origin is a local maximum among the points in the cell. Such a vertex can be proved to be unique with high probability. Our plan then consists in estimating that expectation and prove that the standard deviation is negligible. The method is robust enough to extend to a parametric model where the intensity measure of the underlying Poisson point process has a density proportional to a power of the distance from the origin.
	
	The paper is structured as follows. In Section \ref{sec_expectation_of_nr_of_some_distant_vertices}, we introduce the notion of \emph{pointy} vertex and calculate the expected number of such vertices up to a non-explicit constant which is interpreted as the mean volume of a random simplex, the same which appears in the left-hand side of \eqref{eq_C_d_as_expectation_intro}. In Section \ref{sec_expected_nr_pairs} we investigate the second moment of the number of pointy vertices notably through the use of a generalized Blaschke-Petkantschin-type change of variables formula. In Section \ref{sec_proof_of_theorem} we use these results to prove Theorem \ref{thm_main_result} up to the multiplicative constant \(C_d\) appearing on the right-hand side of \eqref{eq_C_d_as_expectation_intro}. Section \ref{sec_volume_kingman} is self-contained and is devoted to the proof of Theorem \ref{thm_value_of_C_d}, i.e. the calculation of the constant \(C_d\) which completes the proof of Theorem \ref{thm_main_result}. In Section \ref{sec_extrremal_index}, we prove Theorem \ref{thm_extremal_index} related to the calculation of the extremal index. Finally, in Section \ref{sec_alpha_case}, we extend our results to a parametric model which includes both the case of the typical Poisson-Voronoi and the case of the zero-cell of an isotropic and stationary Poisson hyperplane process.
	
	\section{Proof of Theorem \ref{thm_main_result} up to the multiplicative constant \(C_d\)}\label{sec_from_expectation_to_probability}
	\subsection{The expected number of far-off pointy  vertices}\label{sec_expectation_of_nr_of_some_distant_vertices}
	As hinted previously, the determination of the tail probability of $\Rcirc$ is strongly related to the expectation of the number of so-called \textit{pointy} vertices. A \textit{pointy} vertex is a vertex of $\CO$ which locally maximizes, inside $\CO$, the distance to the nucleus.
	Clearly, the variable $\Rcirc$ is the distance from the nucleus to one of the pointy vertices.
	We denote by  \(\VOM\) the set of all pointy vertices of \(\CO\) and \(\VOMR\) the set of far-off pointy vertices of \(\CO\) that are at distance greater than \(t\) from the origin,
	\begin{equation*}
		\VOMR = \lbrace c \in \VOM : \|c\| \geq t \rbrace.
	\end{equation*}
	Moreover, we denote by
	\begin{equation*}
		(\VOMR)^2_{\neq} = \lbrace (c,c') \in (\VOMR)^2 : c \neq c' \rbrace
	\end{equation*}
	the set of all pairs of distinct vertices of \(\CO\) that are pointy and at distance greater than \( t\) from the nucleus \(0\). We then have the equality of events 
	\begin{equation*}
		\lbrace\Rcirc\geq t\rbrace = \lbrace\Card \left(\VOMR\right) \geq 1\rbrace
	\end{equation*}
	where \(\Card(\cdot)\) denotes the cardinality of the specified set. We obtain the asymptotic behavior of \(\PR(\Rcirc\geq t)=\PR(\Card \left(\VOMR\right) \geq 1)\) from that of  \(\E[\Card(\VOMR)]\), which we compute in Sections \ref{sec_expectation_of_nr_of_some_distant_vertices} and \ref{sec_volume_kingman}.  Indeed, in Section \ref{sec_expected_nr_pairs}, we show that the expected number of distinct pairs of such vertices is negligible with respect to \(\E[\Card\VOMR]\) as \(t\rightarrow\infty\) and this implies that, asymptotically for \(t\rightarrow\infty\), if a pointy vertex at distance \(\geq t\) from \(0\) exists, it is unique. For this reason, the asymptotic behavior of \(\Rcirc\) corresponds to that of \(\E[\Card(\VOMR)]\): the details are presented in Section \ref{sec_proof_of_theorem}.
	
	Local maximality is a property describing the local shape of the cell around a vertex. Lemma \ref{prop_equivalent_characterization_of_VOM} translates it into a condition on the points of the Poisson point process that determine the said vertex.
	
	\begin{lemma}\label{prop_equivalent_characterization_of_VOM}
		Let \(c\in\VO\) be a vertex of \(\CO\) determined by the points 0 and \(\m X_1,...,\m X_d\) of \(\Phi\), and let \(\Urm_0,...,\Urm_d \in\Sp_{\R^d}\), where \(\Sp_{\R^d}\) denotes the \((d-1)\)-dimensional unit sphere, be the unit vectors
		\begin{equation*}
			\Urm_0 = \frac{c}{\|c\|} \; , \; \Urm_i = \frac{\m X_i-c}{\|\m X_i-c\|}.
		\end{equation*}
		Then the vertex \(c\) is pointy if and only if
		\begin{equation*}
			0 \in \Conv\left(P_{\Urm_0^\perp}(\Urm_1),...,P_{\Urm_0^\perp}(\Urm_d)\right)
		\end{equation*}
		where by \(\Proj_{\m u^\perp}\) we denote the orthogonal projection onto \(\m u^\perp\), the \((d-1)\)-dimensional linear subspace orthogonal to \(u\in\Sp_{\R^d}\), see Figure \ref{fig_proj_refl}, and where \(\Conv\) denotes the relative interior of the convex hull of the specified points.
	\end{lemma}
	\begin{figure}[h!]
		\centering
			\includegraphics[scale=1]{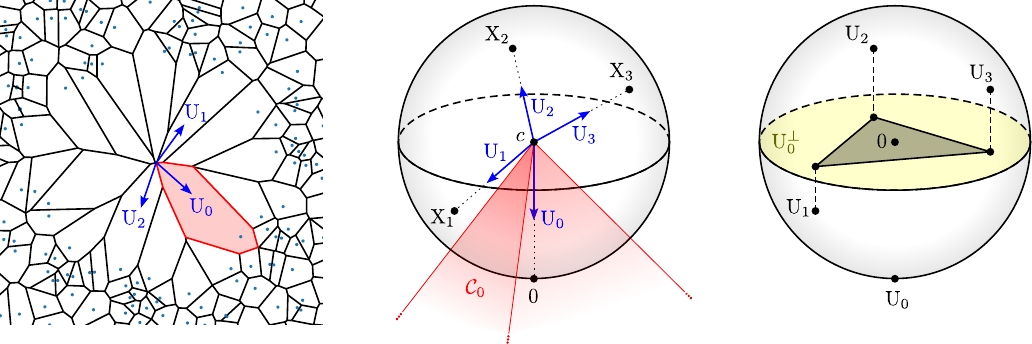}
		\caption{The left-hand side picture is a cluster of big cells in the \(2\)-dimensional case containing at least one with maximal nucleus-vertex distance \(\geq 5\) obtained by emptying a ball of radius \(5\) in the spirit of \cite[Chapter H1]{Aldous}. In the center, a pointy vertex \(c\) of \(\CO\) in the \(3\)-dimensional case and the points \(\m X_1, \m X_2, \m X_3\) that together with \(0\) determine it. On the right-hand side, the shaded triangle is the set \(\Conv(P_{\m U_0^\perp}(\m U_1), P_{\m U_0^\perp}(\m U_2), P_{\m U_0^\perp}(\m U_3))\) obtained from the same set of points: it contains the origin, showing that condition of Lemma \ref{prop_equivalent_characterization_of_VOM} is satisfied.}
		\label{fig_proj_refl}
	\end{figure}
	
	\begin{proof}
		We show the equivalence between the three following statements:
		\begin{enumerate}
			\item the vertex \(c\) is pointy,
			\item no (closed) half sphere whose base is a great circle passing by \(\Urm_0\) contains all \(\Urm_1,...,\Urm_d\),
			\item \(0 \in \Conv\left(P_{\Urm_0^\perp}(\Urm_1),...,P_{\Urm_0^\perp}(\Urm_d)\right)\)
		\end{enumerate}
		Let us denote the cone delimited by the bisecting hyperplanes around a vertex \(c\) as
		\begin{equation*}
			\mathcal{K}(c) = \lbrace e\in\R^d : \langle e, \Urm_i \rangle \leq \langle e, \Urm_0 \rangle \, \forall \, i\in\lbrace1,...,d\rbrace \rbrace.
		\end{equation*}
		A vertex \(c\) is pointy whenever \(\CO\) does not cross the affine hyperplane passing by \(c\) and orthogonal to \(\Urm_0\), or in terms of the cone, if
		\begin{equation}\label{eq_condition_loc_max_wrt_cone}
			\forall \, e \in \mathcal{K}(c) : \langle e, \Urm_0 \rangle > 0 .
		\end{equation}
		We can now show the equivalence between the two assertions 1. and 2. A vertex \(c\in\VO\) is not pointy when there exists a non-zero element \(e\in\mathcal{K}(c)\setminus\lbrace0\rbrace\) satisfying \(\langle e, \Urm_0 \rangle=0\). In particular this element satisfies
		\begin{equation*}
			\forall i\in\lbrace1,..., d \rbrace : \langle e,\Urm_i \rangle \leq \langle e,\Urm_0 \rangle = 0.
		\end{equation*}
		By the inequality above, the hyperplane \(e^\perp\) supports the base of a closed half-sphere containing \(\Urm_1,...,\Urm_d\) and having \(\Urm_0\) on its boundary. Conversely, to a closed half-sphere containing \(\Urm_1,...,\Urm_d\) and having \(\Urm_0\) on its boundary, we associate a vector \(e\) normal to the hyperplane supporting its base and pointing outwards with respect to the half-sphere. This vector satisfies \(\langle e, \Urm_0\rangle=0\) and \(\langle e,\Urm_i \rangle \leq 0\): it is an element for which condition (\ref{eq_condition_loc_max_wrt_cone}) fails.
		
		We finally show the equivalence between assertions 2. and 3.
		Indeed, the origin \(0\) of \(\Urm_0^{\perp}\) is not in \(\Conv(P_{\Urm_0^{\perp}}(\Urm_1),...,P_{\Urm_0^{\perp}}(\Urm_d))\) if and only if there exists a  \((d-2)\)-dimensional linear subspace \(F\subset \Urm_0^{\perp}\) which supports the basis of a half-ball \(B\) of \(\Urm_0^{\perp}\) containing \(P_{\Urm_0^\perp}(\Urm_1),...,P_{\Urm_0^{\perp}}(\Urm_d)\). This means that there exists a linear hyperplane generated by  \(F\cup\{\Urm_0\}\) which supports the basis of a half-sphere \(H\) containing \(\Urm_1,...,\Urm_d\), see Figure \ref{fig_B_H_calka_s_charact_of_condition}.
	\end{proof}
	
	\begin{figure}[h!]
		\centering
		\includegraphics[scale=0.9]{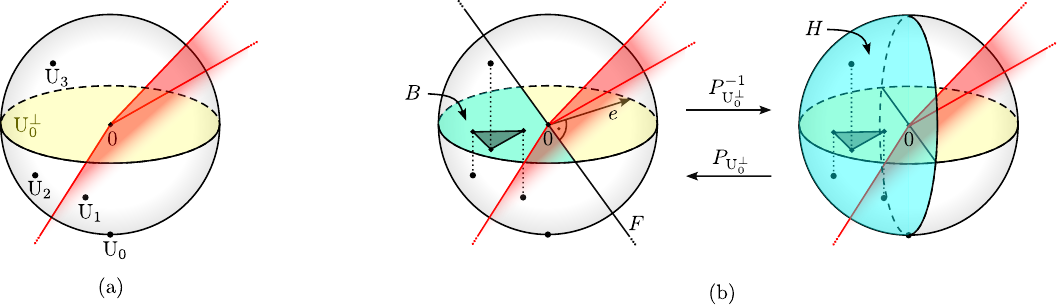}
		\caption{(a) Shaded in red, the local shape of the Voronoi cell associated to \(\Urm_0\) around a non pointy vertex generated by points \(\m U_0, \m U_1 , \m U_2, \m U_3 \). (b) In blue, the sets \(B\) and \(H\)	 together with the subspace \(F\) and the unit vector \(e\) in the context of the proof of Lemma \ref{prop_equivalent_characterization_of_VOM}.}
		\label{fig_B_H_calka_s_charact_of_condition}
	\end{figure}
	
	We are now ready to compute \(\E[\Card(\VOMR)]\), however up to an unknown multiplicative constant \(C_d\). This constant is computed in Section \ref{sec_volume_kingman}.
	
	\begin{proposition}\label{thm_E_nr_locmaxvert}
		For every $t>0$, we have
		\begin{equation}\label{eq_prop_expectation_nr_locmax}
			\E[\Card(\VOMR)] = (d\kappa_d)^d C_d \sum_{i=0}^{d-1} t^{di} \kappa_d^{i-d+1} \frac{(d-1)!}{i!} e^{-\kappa_d t^d} 
		\end{equation}
		where \(\kappa_d\) is given at \eqref{eq:defkappad}. Therefore,
		\begin{equation}\label{eq_prop_expectation_nr_locmax_asymptotic}
			\E[\Card(\VOMR)] \underset{t\rightarrow\infty}{\sim} (d\kappa_d)^d C_d t^{d(d-1)} e^{-\kappa_d t^d}
		\end{equation}
		where, denoting by \(\m U_0,...,\m U_d\) $(d+1)$ i.i.d. uniform random variables on the \((d-1)\)-dimensional unit sphere \(\Sp_{\R^d}\) and by \(\Delta_{d}(\m U_0,...,\m U_d) = \Vol_d(\Conv(\m U_0,...,\m U_d))\) the \(d\)-dimensional Lebesgue measure of the \(d\)-dimensional simplex having these random points as vertices,
		\begin{equation}\label{eq_C_d_as_expectation}
			C_d = \E\left[ \Delta_{d}(\m U_0,...,\m U_d)  \1_{0\in\Conv(P_{\m U_0^\perp}(\m U_1),... ,P_{\m U_0^\perp}(\m U_d))} \right].
		\end{equation}
	\end{proposition}
	
	\begin{proof}
		We start by expressing the quantity \(\Card(\VOMR)\) in relation to the set of points of \(\Phi\) that determine the vertices, preparing the ground for the multivariate Mecke formula \cite[Theorem 4.4]{LastPen}. Explicitly,
		\begin{equation}\label{eq_Card_VOMR_as_sum}
			\Card(\VOMR) = \frac{1}{d!} \sum\limits_{(\m X_1,...,\m X_d)\in\Phi^d_{\neq}} \1_{\Center(0,\m X_1,...,\m X_d)\in\VOMR}
		\end{equation}
		where \(\Center(\cdot)\) denotes the center of the sphere passing by the specified points, \(\Phi^d_{\neq}\) denotes the set of all \(d\)-tuples of pairwise distinct points of \(\Phi\) and the division by \(d!\) compensates for the over representation of vertices given by the permutation of the points \(\m X_1,...,\m X_d\) within themselves. By taking the expectation in (\ref{eq_Card_VOMR_as_sum}) and applying the Mecke formula, we have
		\begin{align}\label{eq_expectation_nr_maxvert_up_to_MS}
			\E\left[\Card(\VOMR)\right] & = \E\left[\frac{1}{d!}\sum_{(\m X_1,...,\m X_d)\in\Phi^d_{\neq}}\1_{\Center(0,\m X_1,...,\m X_d)\in\VOMR}\right] \nonumber
			\\
			& = \frac{1}{d!} \intlim_{(\R^d)^d} d\m x_1... \m x_d \, \E\left[\1_{\text{Center}(0,\m x_1,...,\m x_d)\in\VOMR(\Phi\cup\lbrace0,\m x_1,...,\m x_d\rbrace)}\right].
		\end{align}
		To compute the integral in the last line of (\ref{eq_expectation_nr_maxvert_up_to_MS}) we do a spherical Blaschke-Petkantschin change of variables, see e.g. [\cite{Moller}, Proposition 2.2.3]. As represented in Figure \ref{fig_BP_spherical_classical}, there is a unique \((d-1)\)-dimensional sphere passing by \(0\) and \(\m x_1,...,\m x_d\). We denote by \(\m r\) its radius and by \(\m u_0,\m u_1,...,\m u_d\) the unit vectors pointing from the center of the said sphere towards the points \(0,\m x_1,...,\m x_d\) respectively. Then for \(i\in\lbrace1,...,d\rbrace\),
		\begin{equation}\label{eq_coord_chg_BP}
			\m x_i = \m r(\m u_i-\m u_0)
		\end{equation}
		and for all non-negative functions,
		\begin{equation}\label{eq_coord_chg_BP_Jacobian}
			\int\limits_{(\R^d)^d} d\m x_1... d\m x_d\, f(\m x_1,...,\m x_d) = \int\limits_{0}^{\infty} d\m r \int\limits_{(\Sp_{\R^d})^{d+1}} d\m u_0... d\m u_d \, d! \Delta_{d}(\m u_0,...,\m u_d) \m r^{d^2-1} f(\m r(\m u_1-\m u_0),...,\m r(\m u_d-\m u_0))
		\end{equation}
		where \(d\m u_0,...,d\m u_d\) denotes the standard surface measure on the \((d-1)\)-dimensional unit sphere \(\Sp_{\R^d}\).
		\begin{figure}[h!]
			\centering
			\includegraphics[scale=1]{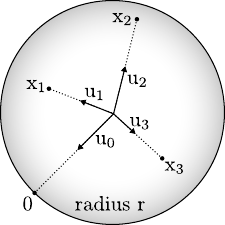}
			\caption{Representation of the classical Blaschke-Petkantschin spherical variable change \((\m x_1,...,\m x_d) \leftrightarrow (\m r,\m u_0,...,\m u_d)\).}
			\label{fig_BP_spherical_classical}
		\end{figure}
		
		We now apply (\ref{eq_coord_chg_BP_Jacobian}) to (\ref{eq_expectation_nr_maxvert_up_to_MS}). Using the identity \(\Center(0,\m x_1,...,\m x_d) = -\m r\m u_0\) and simplifying the factor \(\frac{1}{d!}\) with the factor \(d!\) arising from the Blaschke-Petkantschin coordinate change, we get
		\begin{equation}\label{eq_expectation_after_MS_and_BP}
			\E[\Card(\VOMR)] = \int\limits_0^\infty d\m r \intlim_{(\Sp_{\R^d})^{d+1}} d\m u_0... d\m u_d \, \m r^{d^2-1} \Delta_{d}(\m u_0,...,\m u_d) \E\left[\1_{-\m r\m u_0\in\VOMR(\Phi\cup\lbrace0,\m r(\m u_1-\m u_0),...,\m r(\m u_d-\m u_0)\rbrace)}\right].
		\end{equation}
		We focus now on the calculation of the expectation appearing in the right-hand side of \eqref{eq_expectation_after_MS_and_BP}. The event \( \{-\m r\m u_0\in\VOMR(\Phi\cup\lbrace0,\m r(\m u_1-\m u_0),...,\m r(\m u_d-\m u_0)\rbrace)\} \) occurs if and only if three conditions are satisfied simultaneously. The first two conditions depend deterministically on the variables of integration \(\m r\) and \(\m u_0,...,\m u_d\), i.e. the radius \(\m r\) must be greater than \(t\) and the unit vectors \(\m u_0,...,\m u_d\) must satisfy the geometric condition of Lemma \ref{prop_equivalent_characterization_of_VOM}.
		Once \(\m r\) and \(\m u_0,...,\m u_d\) are fixed, the third condition involves the random point process \(\Phi\): for \(-\m r\m u_0\) to be a vertex of the Voronoi tessellation generated by \(\Phi\cup\{0,\m r(\m u_1-\m u_0),...,\m r(\m u_d-\m u_0)\}\), no point of \(\Phi\) must fall inside of the ball  \(\Ball(0,\m x_1,...,\m x_d)\), where \(\Ball(\cdot)\) denotes the ball whose boundary contains the specified points. With these considerations,
		\begin{equation}\label{eq_01}
		    \1_{-\m r\m u_0\in\VOMR(\Phi\cup\lbrace0,\m r(\m u_1-\m u_0),...,\m r(\m u_d-\m u_0)\rbrace)} = \1_{\m r\geq t}\, \1_{0\in\Conv(P_{\m u_0^\perp}(\m u_1),...,P_{\m u_0^\perp}(\m u_d))}\, \1_{\Phi\cap \Ball(0,\m x_1,...,\m x_d)=\emptyset}.
		\end{equation}
		Taking the expectation of equation \eqref{eq_01}, we obtain
		\begin{align}\label{eq_set_decomposition_of_VOMR_after_MS}
			\E \left[ \1_{-\m r\m u_0\in\VOMR(\Phi\cup\lbrace0,\m r(\m u_1-\m u_0),...,\m r(\m u_d-\m u_0)\rbrace)} \right] 
			& =  \1_{\m r\geq t}\, \1_{0\in\Conv(P_{\m u_0^\perp}(\m u_1),...,P_{\m u_0^\perp}(\m u_d))} e^{-\kappa_d\m r^d}.
		\end{align}
		We insert this last identity \eqref{eq_set_decomposition_of_VOMR_after_MS} into \eqref{eq_expectation_nr_maxvert_up_to_MS}. By Fubini, the integral splits and we obtain
		\begin{align}\label{eq_label}
			& \E[\Card(\VOMR)] \nonumber
			\\
			& = \int\limits_0^\infty d\m r \int\limits_{(\Sp_{\R^d})^{d+1}} d\m u_0... d\m u_d \, \m r^{d^2-1} \Delta_{d}(\m u_0,...,\m u_d) e^{-\kappa_d\m r^d} \1_{\m r\geq t} \, \1_{0\in\Conv(P_{\m u_0^\perp}(\m u_1),...,P_{\m u_0^\perp}(\m u_d))} \nonumber
			\\
			& = \left( \int\limits_{(\Sp_{\R^d})^{d+1}} d\m u_0... d\m u_d \, \Delta_{d}(\m u_0,...,\m u_d) \1_{\{0\in\Conv(P_{\m u_0^\perp}(\m u_1),...,P_{\m u_0^\perp}(\m u_d))\}} \right) \left(\int\limits_0^\infty d\m r \, \m r^{d^2-1} e^{-\kappa_d\m r^d} \1_{\{\m r\geq t\}}\right) \nonumber 
			\\
			& = \left( \int\limits_{(\Sp_{\R^d})^{d+1}} d\m u_0... d\m u_d \, \Delta_{d}(\m u_0,...,\m u_d) \1_{\{0\in\Conv(P_{\m u_0^\perp}(\m u_1),...,P_{\m u_0^\perp}(\m u_d))\}} \right) \left(\frac{1}{d\kappa_d} \sum_{i=0}^{d-1} t^{di} \kappa_d^{i-d+1} \frac{(d-1)!}{i!} e^{-\kappa_d t^d} \right) 
		\end{align}
		where the integral with respect to \(\m r\) has been computed by multiple integration by parts. Up to renormalization by \(\frac{1}{(d\kappa_d)^{d+1}}\) (\(d\kappa_d\) is the surface volume of the \((d-1)\)-dimensional sphere \(\Sp_{\R^d}\)), the integral with respect to \(\m u_0,...,\m u_d\) in (\ref{eq_label}) is \(C_d\):
		\begin{equation}\label{eq_Cecilia_label2}	\int\limits_{(\Sp_{\R^d})^{d+1}} d\m u_0... d\m u_d \, \Delta_{d}(\m u_0,...,\m u_d) \1_{0\in\Conv(P_{\m u_0^\perp}(\m u_1),...,P_{\m u_0^\perp}(\m u_d))} = (d\kappa_d)^{d+1} C_d.
		\end{equation}
		We combine (\ref{eq_Cecilia_label2}) and (\ref{eq_label}) to find the first statement of Proposition \ref{thm_E_nr_locmaxvert}, i.e. equation (\ref{eq_prop_expectation_nr_locmax}). Then, equation (\ref{eq_prop_expectation_nr_locmax_asymptotic}) follows directly.
	\end{proof}
	
	\subsection{The expectation of the number of pairs of far-off pointy vertices}\label{sec_expected_nr_pairs}
	
	In this section we study \((\VOMR)^2_{\neq}\), the set of all pairs of distinct pointy vertices of \(\CO\) at distance \(\geq t\) from \(0\). In Proposition \ref{thm_E_pairs_negligible}, we show that its expectation \(\E\left[\Card\left((\VOMR)^2_{\neq}\right)\right]\) is negligible with respect to \(\E[\Card(\VOMR)]\) as \(t\rightarrow\infty\). This key result allows us, in Section \ref{sec_proof_of_theorem}, to prove Theorem \ref{thm_main_result} up to the multiplicative constant \(C_d\), i.e. that the tail probability \(\PR(\Rcirc\geq t)\) asymptotically behaves like \(\E[\Card(\VOMR)]\).
	
	We explain below why we expect the number of couples of distinct far-off pointy vertices to be negligible in mean in front of the number of far-off pointy vertices. For a vertex \(c\) of \(\CO\) at distance \(\geq t\) from \(0\) to exist, \(\Phi\) must avoid the interior of the ball \(\Ball=\Ball(0,\m X_1,...,\m X_d)\) of radius \(\geq t\), where \(\m X_1,...,\m X_d\in\Phi\) are the nuclei that together with \(0\) determine \(c\). Moreover, for \(c\) to be pointy, the nuclei \(\m X_1,...,\m X_d\) must be scattered on \(\partial\Ball\) according to the criterion discussed in the proof of Lemma \ref{prop_equivalent_characterization_of_VOM}: that is, no hemisphere of \(\Ball\) whose boundary passes by \(0\) contains all of them. Conditional on the existence of \(c\in\VOMR\), if there is a second pointy vertex \(c'\in\VOMR\), it will most likely be close to \(c\). That is because an additional ball \(\Ball'\), centered in \(c'\) and passing by \(0\), must be avoided by \(\Phi\). When \(t\) is large, this is very unlikely, unless \(\Ball\) and \(\Ball'\) overlap extensively. This forced proximity adds a significant constraint on the distribution of the nuclei \(\m X_1,...,\m X_d\) on the surface \(\partial\Ball\) and a similar constraint on the additional nuclei along $\partial\Ball'$ which determine $c'$. As the interior of \(\Ball'\) must be avoided by \(\Phi\), \(\m X_1,...,\m X_d\) must be distributed on \(\partial\Ball\setminus\Ball'\), which is a region just slightly larger than an hemisphere. At the same time and most importantly, for \(c\) to be pointy, enough points within \(\m X_1,...,\m X_d\) must be located close to the interface \(\partial\Ball\cap\partial\Ball'\) so to avoid the existence of an hemisphere of \(\partial\Ball\) containing them all. Figure \ref{fig_scenarios_close} represents an (unlikely) realization of a pair of pointy vertices in the \(2\)-dimensional case. This is formalized in Proposition \ref{thm_E_pairs_negligible} below. 
	
	\begin{figure}[h!]
		\centering
		\includegraphics[scale=1]{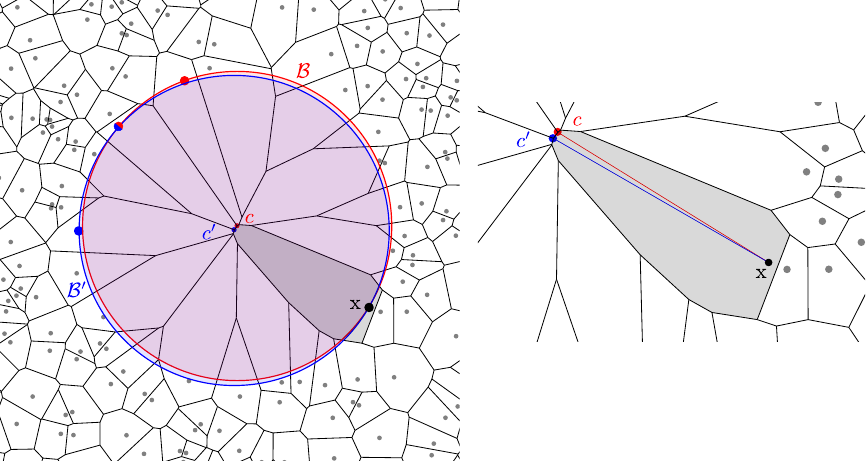}
		\caption{Simulation of a cell with nucleus \(\m x\), shaded in gray, featuring a large maximal nucleus-vertex distance and two pointy vertices \(c\) and \(c'\) located at the center of respectively \(\Ball\) and \(\Ball'\). In this realization, \(c\) and \(c'\) share an edge, and therefore are determined by two common points of the Poisson point process: \(\m x\) and the point colored both in red and in blue.}
		\label{fig_scenarios_close}
	\end{figure}
	
	\begin{proposition}\label{thm_E_pairs_negligible}
		As \(t\rightarrow\infty\), 
		\begin{equation*}
			\E\left[\Card \left( \left(\VOMR\right)^2_{\neq} \right) \right] 
			= \mathcal{O}\left( t^{d(d-2)} e^{-\kappa_dt^d} \right)
			= o\left( \E\left[\Card (\VOMR) \right] \right).
		\end{equation*}
	\end{proposition}
	
	\begin{proof}
	This proof is structured into six steps. In Step 1, we use the Mecke formula to obtain an integral expression for \(\E \left[ \Card \left((\VOMR)^2_{\neq}\right) \right]\), to which we apply  in Step 2 a Blaschke-Petkantschin change of variables stated in Lemma \ref{prop_Nikitenko}. In Step 3, we use the intermediary geometric result given in Lemma \ref{prop_bound_for_couple_of_vertices} to obtain a first upper bound for this integral. Finally, Steps 4-6 are devoted to a purely analytical treatment of the integrals involved.
	\\~\\
	\noindent\textbf{Step 1. Rewriting of \(\E \left[ \Card \left((\VOMR)^2_{\neq}\right) \right]\) using Mecke's formula}\\~\\
	Let us consider a pair \((c,c')\in(\VOMR)^2_{\neq}\).
	We associate to each of the two vertices a family of \(d\) distinct points of \(\Phi\) which, together with the origin, are the nuclei of the Voronoi cells containing respectively \(c\) and \(c'\). We then distinguish \(d\) different cases in function of the cardinality of the intersection of these two families. When the two families share exactly \(k\) points, for \(k\in\lbrace0,...,d-1\rbrace\), we use the short notation \(\X=(\m X_1,...,\m X_k)\) for the common points and \(\Y = (\m Y_1,...,\m Y_{d-k})\)	and \(\Y' = (\m Y'_1,...,\m Y'_{d-k})\) for the remaining points of the two families respectively. 
	The number \((d-k)\) is the minimal dimension of a facet containing both \(c\) and \(c'\). We remark that the limit case \(k=d\) corresponds to \(c=c'\), which is excluded. Consequently, we describe the cardinality of \((\VOMR)^2_{\neq}\) as follows:
	\begin{equation}\label{eq_card_pair_vertices_as_sum_over_pts_of_PPP}
		\Card \left((\VOMR)^2_{\neq}\right) = \sum\limits_{k=0}^{d-1} \frac{1}{k!((d-k)!)^2} \sum_{\Phi^{2d-k}_{\neq}}
		\1_{(\Center(0,\m X_1,...,\m X_k,\m Y_1,...,\m Y_{d-k}),\Center(0,\m X_1,...,\m X_k,\m Y_1',...,\m Y_{d-k}'))\in(\VOMR)^2_{\neq}}
	\end{equation}
	where \(\Center(\cdot)\) denotes the center of the sphere passing through the specified points and 
	where, explicitly, the sum runs over the set of all \((2d-k)\)-tuples of pairwise distinct points of \(\Phi\),
	\begin{equation*}
	    \Phi^{2d-k}_{\neq} = \lbrace (\m X_1,...,\m X_k,\m Y_1,...,\m Y_{d-k},\m Y_1',...,\m Y_{d-k}') = (z_1,...,z_{2d-k}) \in \Phi^{2d-k} | z_i\neq z_j \forall i\neq j\rbrace.
	\end{equation*}
	Moreover, the division by \(k!((d-k)!)^2\) compensates for the over-representation of each couple of vertices given by the permutations of the entries of \(\X\), \(\Y\) and \(\Y'\) within their own set. We take the expectation of the identity (\ref{eq_card_pair_vertices_as_sum_over_pts_of_PPP}) and then apply Mecke formula: we obtain
	\begin{align}\label{eq_expectation_pairs_just_after_Mecke}
		\E  \left[ \Card \left((\VOMR)^2_{\neq}\right) \right] 
		& = \sum\limits_{k=0}^{d-1} \frac{1}{k!((d-k)!)^2} \E\left[ \sum_{\Phi^{2d-k}_{\neq}}
		\1_{(\Center(0,\m X_1,...,\m X_k,\m Y_1,..., \m Y_{d-k}),\Center(0,\m X_1,...,\m X_{k},\m Y_1',...,\m Y_{d-k}'))\in(\VOMR)^2_{\neq}} \right]
		\nonumber
		\\
		& = \sum\limits_{k=0}^{d-1} \frac{1}{k!((d-k)!)^2} \int\limits_{(\R^d)^{k}} d\x\; \int\limits_{(\R^d)^{d-k}} d\y \int\limits_{(\R^d)^{d-k}} d\y' \E \left[ \1_{(c,c')\in(\VOMR)^2_{\neq}(\Phi\cup\lbrace\x,\y,\y'\rbrace)} \right]
	\end{align}
	where we use the analogous short notation \(\x=(\m x_1,...,\m x_k)\), \(\y=(\m y_1,...,\m y_{d-k})\), \(\y'=(\m y'_1,...,\m y'_{d-k})\) for the points and \(d\x=d\m x_1... d\m x_k\), \(d\y=d\m y_1... d\m y_{d-k}\), \(d\y'=d\m y'_1... d\m y'_{d-k}\) for the corresponding Lebesgue measures, as well as the short notations
	\begin{equation*}
		c = \Center(0,\m x_1,...,\m x_k,\m y_1,...,\m y_{d-k}) \text{ and } c' = \Center(0,\m x_1,...,\m x_k,\m y_1',...,\m y_{d-k}')
	\end{equation*}
	and
	\begin{equation*}
	    \lbrace\x,\y,\y'\rbrace = \lbrace \m x_1,...,\m x_k,\m y_1,...,\m y_{d-k},\m y_1',...,\m y_{d-k}'\rbrace.
	\end{equation*}
	Within the expectation on the right-hand side of (\ref{eq_expectation_pairs_just_after_Mecke}), \(\x,\y\) and \(\y'\) are deterministic. Therefore, the only random feature of the indicator function consists on whether the points \(c,c'\) indeed are vertices of the Voronoi tessellation generated by \(\Phi\cup\lbrace0\rbrace\cup\lbrace\x,\y,\y'\rbrace\). That is the case exactly when no other point of \(\Phi\) falls in the interior of the balls \(\Ball = \Ball(0,\m x_1,\ldots,\m x_k,\m y_1,\ldots,\m y_{d-k})\) and \(\Ball'=\Ball(0,\m x_1,\ldots,\m x_k,\m y_1',\ldots,\m y_{d-k}')\). We therefore have
	\begin{align}\label{eq_expectation_pairs_equals_exp_x_indicator}
		\E \left[ \1_{c,c'\in\VOMR(\Phi\cup\lbrace\x,\y,\y'\rbrace)} \right] 
		& = e^{-\Vol_d(\Ball\cup\Ball')} \1_{(c,c')\in\left(\VOMR\right)^2_{\neq}(\lbrace0\rbrace\cup\lbrace\x,\y,\y'\rbrace)}.
	\end{align}
	Inserting (\ref{eq_expectation_pairs_equals_exp_x_indicator}) into (\ref{eq_expectation_pairs_just_after_Mecke}), we obtain
	\begin{equation}\label{eq_expectation_pairs_for_k_fixed_after}
		\E \left[ \Card \left((\VOMR)^2_{\neq}\right) \right] =  \sum\limits_{k=0}^{d-1} \frac{1}{k!((d-k)!)^2} \int\limits_{(\R^d)^{k}} d\x\; \int\limits_{(\R^d)^{d-k}} d\y \int\limits_{(\R^d)^{d-k}} d\y'\; e^{-\Vol_d(\Ball\cup\Ball')} \1_{(c,c')\in\left(\VOMR\right)^2_{\neq}(\lbrace0\rbrace\cup\lbrace\x,\y,\y'\rbrace)}.
	\end{equation}
	\noindent\textbf{Step 2. Spherical change of variables}\\~\\	
	The computation of the integrals in (\ref{eq_expectation_pairs_equals_exp_x_indicator}) requires a spherical Blaschke-Petkantschin-type change of variables. We introduce it in Lemma \ref{prop_Nikitenko} below which combines two integral geometric formulas given in \cite{Nikitenko}, see Figure \ref{fig_BP_Nikitenko}. 
	\begin{figure}[h!]
		\centering
		\includegraphics[scale=1]{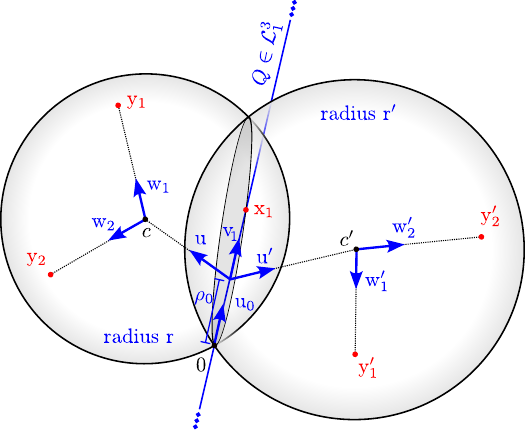}
		\caption{Sketch of the Blaschke-Petkantschin-type variable change of Lemma \ref{prop_Nikitenko} for the case \(d=3,k=1\).}
		\label{fig_BP_Nikitenko}
	\end{figure}
	\begin{lemma}\label{prop_Nikitenko}
		For \(k\in\lbrace0,...,d-1\rbrace\), let \(\m x_1,...,\m x_k,\m y_{1},...,\m y_{d-k},\m y'_{1},...,\m y'_{d-k}\in\R^d\) be \(2d-k\) points in general position. We denote by \(\text{Vect}(\cdot)\) the linear subspace of real linear combination of the specified points and by \(\Sp(\cdot)\) the sphere of least dimension that passes by the specified points.
		\begin{itemize}
			\item The \(k\) points \(\m x_1,...,\m x_k\) determine a \(k\)-dimensional linear subspace \(Q=\text{Vect}(\m x_1,...,\m x_k)\in\mathcal{L}^d_k\). Within the subspace \(Q\), the points \(0\) and \(\m x_1,...,\m x_k\) lay on a unique \((k-1)\)-dimensional sphere \(\Sp(0,\m x_1,...,\m x_k)\) of radius \(\rho_0\in\R_{+}\). We denote by \(\m u_0\in\Sp_{Q}\), where \(\Sp_Q\) is the unit sphere on the space \(Q\), the unit vector pointing from the origin towards the center of \(\Sp(0,\m x_1,...,\m x_k)\). Moreover, for \(i\in\lbrace1,...,k\rbrace\), we denote by \(\m v_i\in\Sp_Q\) the unit vector pointing from the center of \(\Sp(0,\m x_1,...,\m x_k)\) towards \(\m x_i\).
			\item The \((d+1)\) points \(0,\m x_1,...,\m x_k,\m y_1,...,\m y_{d-k}\) determine a \(d\)-dimensional sphere \(\Sp(0,\m x_1,...,\m x_k,\m y_1,...,\m y_{d-k})\) of center \(c\) and radius \(\m r\in\R_{+}\). For \(j\in\lbrace1,...,d-k\rbrace\) we denote by \(\m w_j\in\Sp_{\R^d}\) the unit vector pointing from \(c\) towards \(\m y_j\), where \(\Sp_{\R^d}\) denotes the unit sphere in the space \(\R^d\).
			\item The \((d+1)\) points \(0,\m x_1,...,\m x_k,\m y_1',...,\m y_{d-k}'\) determine a \(d\)-dimensional sphere \(\Sp(0,\m x_1,...,\m x_k,\m y_1',...,\m y_{d-k}')\) of center \(c'\) and radius \(\m r'\in\R_{+}\). For \(j\in\lbrace1,...,d-k\rbrace\) we denote by \(\m w'_j\in\Sp_{\R^d}\) the unit vector pointing from \(c'\) towards \(\m y'_j\).
			\item Finally, we denote by \(\m u,\m u'\in\Sp_{Q^\perp}\) the unit vectors in the \((d-k)\)-dimensional space \(Q^\perp\) pointing respectively towards \(c\) and \(c'\) from the center of the \(k\)-dimensional ball \(\Sp(0,\m x_1,...,\m x_k)\).
		\end{itemize}
		We have for \(i\in\lbrace1,...,k\rbrace\),
		\begin{equation*}
			\m x_i = \rho_0(\urm_0+\m v_i)
		\end{equation*}
		and for \(j\in\lbrace1,...,d-k\rbrace\),
		\begin{align*}
			& \m y_j = \rho_0\urm_0+\sqrt{\m r^2-\rho_0^2}\m u+\m r\m w_j, \\
			& \m y'_j = \rho_0\urm_0+\sqrt{\m (\m r')^2-\rho_0^2}\m u+\m r'\m w'_j.
		\end{align*}
		Then for all non-negative functions we have
		\begin{equation*}
			\begin{split}
				\int\limits_{(\R^d)^{2d-k}} d\x \, d\y \, d\y' f(\x,\y,\y')
				= & \int\limits_{\mathcal{L}_k^d}d\m Q\int\limits_{\Sp_{Q}^{k+1}}d\urm_0d\vbf\int\limits_{\R_{+}}d\rho_0 \int\limits_{\Sp_{Q^\perp}}d\m u\int\limits_{\R_{+}}d\m r\int\limits_{\Sp_{\R^d}^{d-k}}d\wbf \int\limits_{\Sp_{Q^\perp}}d\m u'\int\limits_{\R_{+}}d\m r'\int\limits_{\Sp_{\R^d}^{d-k}}d\wbf' \\
				& \rho_0^{kd-1} (k!\Delta_{k}(-\urm_0,\vbf))^{d-k+1} \1_{0\leq\rho_0\leq \m r,\m r'} \\
				& \m r^{(d-k)(d-1)+1} \sqrt{\m r^2-\rho_0^2}^{d-k-2}(d-k)!\Delta_{d-k}(-\frac{\sqrt{\m r^2-\rho_0^2}}{\m r}\m u,P_{\m Q^\perp}(\wbf)) \\
				& \m (\m r')^{(d-k)(d-1)+1} \sqrt{\m (\m r')^2-\rho_0^2}^{d-k-2}(d-k)!\Delta_{d-k}(-\frac{\sqrt{\m (\m r')^2-\rho_0^2}}{\m r'}\m u',P_{\m Q^\perp}(\wbf')) \\
				& f(\rho_0(\urm_0+\vbf,\rho_0\urm_0+\sqrt{\m r^2-\rho_0^2}\m u+\m r\wbf,\rho_0\urm_0+\sqrt{\m (\m r')^2-\rho_0^2}\m u+\m r'\wbf')
			\end{split}
		\end{equation*}
		where we used the short notation \(P_{\m Q^\perp}(\wbf) = (P_{\m Q^\perp}(\m w_1),...,P_{\m Q^\perp}(\m w_{d-k}))\) (and similarly for \(P_{\m Q^\perp}(\wbf'))\) and where we denote by \(d\m u_0\), \(d\vbf=d\m v_1... d\m v_k\), \(d\m u\), \(d\m u'\), \(d\wbf=d\m w_1... d\m w_{d-k}\) and \(d\wbf'=d\m w'_1... d\m w'_{d-k}\) the respective surface area measures on the corresponding spheres and where we associate to the Grassmannian \(\mathcal{L}^d_k\) its standard measure \(Q\), normalized such that
		\begin{equation*}
			\int\limits_{\mathcal{L}^d_k} d\m Q = \binom{d}{k} \frac{\kappa_d \kappa_{d-1} ... \kappa_{d-k+1}}{\kappa_k \kappa_{k-1} ... \kappa_1},
		\end{equation*}
		with $\kappa_d$ given at \eqref{eq:defkappad}.
	\end{lemma}
	
	\begin{remark}\label{remark_k=0}
		When \(k=0\), the Grassmannian \(\mathcal{L}^d_0\) considered in Lemma \ref{prop_Nikitenko} above is trivial, equal to  \(\lbrace0\rbrace\). For this reason, the integrals with respect to \(Q,\urm_0,\m v_1,...,\m v_k,\rho_0\) disappear and this change of variables corresponds to a double classical spherical Blaschke-Petkantschin change of variables, see \eqref{eq_coord_chg_BP_Jacobian}.
	\end{remark}
	
	\begin{proof}[Proof of Lemma \ref{prop_Nikitenko}]
		This particular change of variables is a combination of Blaschke-Petkantschin type formulas  presented in \cite{Nikitenko}. In particular, \cite[Theorem 4]{Nikitenko} presents a spherical change of variables for \(m\) points laying on a \((m-1)\)-dimensional sphere passing through the origin: we apply this result with \(m=k\) to represent the points \(\m x_1,...\m x_k\) as \(\rho_0(\urm_0+\m v_1),... ,\rho_0(\urm_0+\m v_k)\). \cite[Theorem 5]{Nikitenko} produces a spherical change of variables describing \(m\) points belonging to a \((m+k-1)\)-dimensional sphere which is determined by those points together with a given \((k-1)\)-dimensional sphere centered at the origin. We apply this result twice with \(m=d-k\), once for \((\m y_1-\rho_0\m u_0,...,\m y_{d-k}-\rho_0\m u_0)\) and once for \((\m y'_1-\rho_0\m u_0,...,\m y'_{d-k}-\rho_0\m u_0)\) to obtain our formula.
	\end{proof}
	
	With the change of variables of Lemma \ref{prop_Nikitenko}, expression (\ref{eq_expectation_pairs_for_k_fixed_after}) takes the form
	\begin{align}\label{eq_expectation_pairs_after_Nikitenko}
		\E \left[ \Card \left((\VOMR)^2_{\neq}\right) \right] =  & \sum\limits_{k=0}^{d-1} \frac{1}{k!((d-k)!)^2} \int\limits_{\mathcal{L}_k^d}d\m Q\int\limits_{\Sp_{\m Q}^{k+1}}d\urm_0d\vbf\int\limits_{\R_{+}}d\rho_0 \int\limits_{\Sp_{\m Q^\perp}}d\m u\int\limits_{\R_{+}}d\m r\int\limits_{\Sp_{\R^d}^{d-k}}d\wbf \int\limits_{\Sp_{Q^\perp}}d\m u'\int\limits_{\R_{+}}d\m r'\int\limits_{\Sp_{\R^d}^{d-k}}d\wbf 
		\nonumber\\
		& \rho_0^{kd-1} \left(k!\Delta_{k}\left(-\urm_0,\vbf\right)\right)^{d-k+1} \1_{0\leq\rho_0\leq \m r,\m r'} 
		\nonumber\\
		& \m r^{(d-k)(d-1)+1} \sqrt{\m r^2-\rho_0^2}^{d-k-2}(d-k)!\Delta_{d-k}\left(-\frac{\sqrt{\m r^2-\rho_0^2}}{\m r}\m u,P_{\m Q^\perp}(\wbf)\right) \nonumber
		\\
		& \m (\m r')^{(d-k)(d-1)+1} \sqrt{\m (\m r')^2-\rho_0^2}^{d-k-2}(d-k)!\Delta_{d-k}\left(-\frac{\sqrt{\m (\m r')^2-\rho_0^2}}{\m r'}\m u',P_{\m Q^\perp}(\wbf')\right) 
		\nonumber\\
		& e^{-\Vol_d(\Ball\cup\Ball')} \1_{(c,c')\in(\VOMR)^2_{\neq}(\lbrace0\rbrace\cup\lbrace\x,\y,\y'\rbrace)}
	\end{align}
	where we used the short notation introduced in the Lemma.
	\\~\\
	\noindent\textbf{Step 3. Upper bound for the integrand}
	\\~\\	
	We can now start to determine an upper bound for each factor of the integrand of (\ref{eq_expectation_pairs_after_Nikitenko}). By symmetry of this integrand, we can henceforth assume without loss of generality that 
	\begin{equation*}
		\m r'\geq \m r.
	\end{equation*}
	Next, we bound the quantities in the last line of the right hand side of (\ref{eq_expectation_pairs_after_Nikitenko}) in Lemma \ref{prop_bound_for_couple_of_vertices} below.
	\begin{lemma}\label{prop_bound_for_couple_of_vertices}
		Let \(\Ball\) and \(\Ball'\) be two balls of respective centers $c$ and $c'$ and respective radii \(\m r\) and \(\m r'\) where \(\m r\leq \m r'\) and such that $0\in \partial \Ball \cap \partial \Ball'$. Let $\x=(\m x_1,...,\m x_k)\in (\partial \Ball\cap \partial \Ball')^k$, $\y=(\m y_1,...,\m y_{d-k})\in (\partial\Ball\setminus \Ball')^{d-k}$, $\y'=(\m y'_1,...,\m y'_{d-k})\in(\partial\Ball'\setminus\Ball)^{d-k}$.
		We then obtain
		\begin{align*}\label{eq_bound_on_{d-1}ol_for_couple_of_vertices}
			e^{-\Vol_d(\Ball\cup\Ball')} \1_{(c,c')\in(\VOM)^2_{\neq}(\lbrace0\rbrace\cup\lbrace\x,\y,\y'\rbrace)} 
			& \leq e^{-\frac{\kappa_{d}}{2} \m r^d -\frac{\kappa_{d}}{2} (\m r')^d - \frac{\kappa_{d-1}}{2d} \m r^{d-1} \delta} \1_{\delta^2\geq \m (\m r')^2 - \m r^2}
		\end{align*}
		where we recall that \(\kappa_d\) is the volume of the d-dimensional unit ball and where \(\delta = \|c-c'\|\).
	\end{lemma}
	
	\begin{proof}[Proof]
		By Lemma \ref{prop_equivalent_characterization_of_VOM} and because the elements of $\y$ (resp. $\y'$) lie outside of $\Ball'$ (resp. $\Ball$), the condition \((c,c')\in(\VOM)^2_{\neq}(\lbrace0\rbrace\cup\lbrace\x,\y,\y'\rbrace)\) implies that there must be at least a hemisphere of \(\Sp=\partial\Ball\) (resp. \(\Sp'=\partial\Ball'\)) outside of \(\Ball'\) (resp.  \(\Ball\)). 
		Therefore, the two outermost half balls \(\Ball_{(\text{\textonehalf})}\subset\Ball\) and \(\Ball'_{(\text{\textonehalf})}\subset\Ball'\), with base perpendicular to the line joining \(c\) and \(c'\), are disjoint, see Figure \ref{fig_half_balls_for_vol_bound} (a), hence
		\begin{equation}\label{eq:vol2halfballs}
			\Vol_d\left( \Ball_{(\text{\textonehalf})} \cup \Ball_{(\text{\textonehalf})}' \right) = \frac{\kappa_d}{2} \m r^d + \frac{\kappa_d}{2} \m (\m r')^d .
		\end{equation}
		\begin{figure}[h!]
			\centering
			\includegraphics[scale=1]{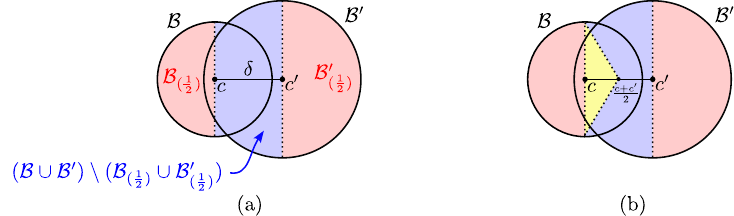}
			\caption{(a) Partition of the set \(\Ball\cup\Ball'\) into two half-balls \(\Ball_{(\text{\textonehalf})}\) and \(\Ball_{(\text{\textonehalf})}'\) and their complement.
			(b) The pyramid, shaded in green, with same base as \(\Ball_{(\text{\textonehalf})}\) contained in \((\Ball\cup\Ball')\setminus(\Ball_{(\text{\textonehalf})} \cup \Ball_{(\text{\textonehalf})}')\).}
			\label{fig_half_balls_for_vol_bound}
		\end{figure}
		\\
		\noindent
		As for the remaining set \(\Ball\cup\Ball'\setminus(\Ball_{(\text{\textonehalf})} \cup \Ball_{(\text{\textonehalf})}')\), we observe that it contains a pyramid with same base as $\Ball_{(\text{\textonehalf})}$ and with apex at the midpoint of the segment joining \(c\) and \(c'\), see Figure \ref{fig_half_balls_for_vol_bound} (b). Consequently,
		\begin{equation}\label{eq:volremaining}
			\Vol_d(\Ball\cup\Ball'\setminus(\Ball_{(\text{\textonehalf})} \cup \Ball_{(\text{\textonehalf})}'))\ge \frac{\kappa_{d-1}}{2d}\m r^{d-1}\delta.
		\end{equation}
		\noindent Combining \eqref{eq:vol2halfballs} and \eqref{eq:volremaining}, we get
		\begin{equation*}
			e^{-\Vol_d(\Ball\cup\Ball')} = e^{-\Vol_d\left(\text{\textonehalf}\Ball\cup\text{\textonehalf}\Ball'\right)-\Vol_d(\Ball\cup\Ball'\setminus(\text{\textonehalf}\Ball \cup \text{\textonehalf}\Ball')} \leq e^{-\frac{\kappa_d}{2}\m r^d-\frac{\kappa_d}{2}(\m r')^d-\frac{\kappa_{d-1}}{2d}\m r^{d-1}\delta}.
		\end{equation*}
		Finally, we find an upper bound for the indicator function of \((c,c')\in(\VOM)^2_{\neq}(\lbrace0\rbrace\cup\lbrace\x,\y,\y'\rbrace)\). As mentioned previously, for it to be satisfied, there must be at least a hemisphere of \(\partial\Ball\) outside of \(\Ball'\). This is the case whenever the distance \(\delta\) between the centers of the two balls is greater than a minimal value \(\sqrt{\m (\m r')^2-\m r^2}\), corresponding to the limit case of Figure \ref{fig_vol_bounds_condition_on_delta}, where \(\partial\Ball\) exceeds \(\Ball'\) by exactly one hemisphere. This justifies the bound
		\begin{equation*}
			\1_{(c,c')\in(\VOM)^2_{\neq}(\lbrace0\rbrace\cup\lbrace\x,\y,\y'\rbrace)} \leq \1_{\delta^2\geq \m (\m r')^2-\m r^2}.
		\end{equation*}
		\begin{figure}[h!]
			\centering
			\includegraphics[scale=1]{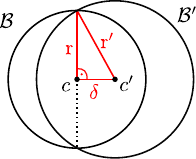}
			\caption{Lower limit case for the distance \(\delta\) between the centers \(c,c'\) of the balls \(\Ball,\Ball'\).}
			\label{fig_vol_bounds_condition_on_delta}
		\end{figure}
	\end{proof}
	We now insert the upper bound provided by Lemma \ref{prop_bound_for_couple_of_vertices} into the right hand side of (\ref{eq_expectation_pairs_after_Nikitenko}). We also bound the functionals
	$ \Delta_{k}(\cdot)$ and $\Delta_{d-k}(\cdot)$ by a constant since they represent the volumes of simplices which are contained in the unit ball. We then obtain the existence of a positive constant $C$ such that
	\begin{align}\label{eq_bound_on_couple_vert_before_var_change}
		\E \left[ \Card \left((\VOMR)^2_{\neq}\right) \right] \leq \Cst \sum\limits_{k=0}^{d-1} & \int\limits_{\mathcal{L}_k^d}d\m Q\int\limits_{\Sp_{\m Q}^{k+1}}d\urm_0d\vbf \int\limits_{t}^\infty d\m r \int\limits_0^{\m r} d\rho_0 \int\limits_{\Sp_{\m Q^\perp}}d\m u\int\limits_{\Sp_{\R^d}^{d-k}}d\wbf \int\limits_{\Sp_{\m Q^\perp}}d\m u'\int\limits_{\m r}^\infty d\m r'\int\limits_{\Sp_{\R^d}^{d-k}}d\wbf'\rho_0^{kd-1}  
		\nonumber\\
		&
		\hspace{-0.5cm}(\m r\m r')^{(d-k)(d-1)+1} \left((\m r^2-\rho_0^2)((\m r')^2-\rho_0^2)\right)^{\frac{d-k-2}{2}}
		e^{-\frac{\kappa_d}{2}(\m r^d+(\m r')^d) - \frac{\kappa_{d-1}}{2d}\m r^{d-1}\delta} \1_{\delta^2\geq \m (\m r')^2-\m r^2} 
	\end{align}
	where the distance between centers \(\delta=\|c-c'\|\) can be expressed in function of the variables of integration as
	\begin{equation}\label{eq_delta_generic}
		\delta = \sqrt{\m r^2-\rho_0^2+\m (\m r')^2-\rho_0^2-2\sqrt{\m r^2-\rho_0^2}\sqrt{\m (\m r')^2-\rho_0^2}\langle \m u,\m u'\rangle}.
	\end{equation}
	Henceforth, we denote by $C$ a generic positive constant depending only on dimension $d$ and whose value may change from line to line (not to be confused with the constant $C_d$ whose explicit value is determined in Section \ref{sec_volume_kingman}). 

	The integrand of (\ref{eq_bound_on_couple_vert_before_var_change}) does not depend on the variables \(\urm_0,\vbf,\wbf\) and \(\wbf'\), which we immediately integrate to a constant. Moreover, since it only depends on \(\m Q,\m u\) and \(\m u'\)  through the scalar product \(\langle u, u'\rangle\) hidden within \(\delta\), we can without loss of generality make the choice
	\begin{equation*}
		\m Q = \m Q_0 = \text{Vect}_\R (e_1,...,e_k) \text{ and } \m u' = e_d
	\end{equation*}
	where \((e_1,...,e_d)\) denotes the standard orthogonal basis of \(\R^d\).
	After these operations, we obtain the bound
	\begin{equation}\label{eq_bound_on_couple_vert_with_fixed_Q}
		\E \left[ \Card \left((\VOMR)^2_{\neq}\right) \right] \leq \Cst \sum\limits_{k=0}^{d-1} I_k(t) 
	\end{equation}
	where 
	for \(k\in\lbrace1,...,d-1\rbrace\),
	\begin{align}\label{eq_definition_I_k}
		I_k(t)=&\int\limits_{t}^{\infty}d\m r  \int\limits_{\m r}^{\infty}d\m r' \int\limits_{0}^{\m r} d\rho_0 \int\limits_{\Sp_{\m Q_0^\perp}}d\m u \, \rho_0^{kd-1} (\m r\m r')^{(d-k)(d-1)+1} \left((\m r^2-\rho_0^2)((\m r')^2-\rho_0^2)\right)^{\frac{d-k-2}2} \nonumber
		\\
		& \hspace{5cm}e^{-\frac{\kappa_d}{2}\m r^d-\frac{\kappa_d}{2}(\m r')^d- \frac{\kappa_{d-1}}{2d}\m r^{d-1}\delta} \1_{\delta^2\geq (\m r')^2-\m r^2} 
	\end{align}
	where 
	\begin{equation}\label{eq_definition_delta_before_step_4}
		\delta = \sqrt{\m r^2-\rho_0^2+(\m r')^2-\rho_0^2-2\sqrt{\m r^2-\rho_0^2}\sqrt{(\m r')^2-\rho_0^2}\langle\m  u,e_d\rangle}.
	\end{equation}
	Concerning the case \(k=0\), as discussed in Remark \ref{remark_k=0}, the integral with respect to \(\rho_0\) disappears. Moreover we have \(\delta=\sqrt{(\m r')^2+\m r^2-2\m r'\m r\langle\m u, e_d\rangle}\) (which corresponds to setting \(\rho_0=0\) in \eqref{eq_delta_generic}) and consequently \(\delta^2\geq(\m r')^2-\m r^2 \iff \langle\m u, e_d\rangle\leq \frac{\m r}{\m r'}\), hence
	\begin{equation}\label{eq_definition_I_0}
		I_0(t) = \int\limits_{t}^{\infty}d\m r  \int\limits_{\m r}^{\infty}d\m r' \int\limits_{\Sp_{\R^d}}d\m u \,(\m r\m r')^{d^2-1} e^{-\frac{\kappa_d}{2}\m r^d-\frac{\kappa_d}{2}(\m r)^d- \frac{\kappa_{d-1}}{2d}\m r^{d-1}\sqrt{(\m r')^2+\m r^2-2\m r'\m r\langle\m u, e_d\rangle}} \1_{\langle\m u, e_d\rangle\leq \frac{\m r}{\m r'}} .
	\end{equation}
	In the next three steps, we bound the integrals \(I_k(t)\) according to three different cases: \(k=d-1\), \(k\in\lbrace1,...,d-1\rbrace\) and \(k=0\).
	\\~\\
	\noindent\textbf{Step 4. Analytical treatment of \(I_{d-1}(t)\), see \eqref{eq_definition_I_k}}
	\\~\\
	In the case \(k=d-1\), \(\Sp_{Q_0^\perp}\) is a \(0\)-dimensional sphere, so the fact that \(\Ball\) and \(\Ball'\) are distinct fixes \(\langle u, e_d\rangle=-1\), hence \(\delta\) takes the form
	\begin{equation*}
	    \delta=\sqrt{(\m r')^2-\rho_0^2+\m r^2-\rho_0^2+2\sqrt{(\m r')^2-\rho_0^2} \sqrt{\m r^2-\rho_0^2} }=\sqrt{(\m r')^2-\rho_0^2}+\sqrt{\m r^2-\rho_0^2}.
	\end{equation*} 
	Moreover, the integral with respect to \(\m u\) disappears,
	\begin{equation}\label{eq_I_d-1}
		I_{d-1}(t) = \int\limits_{t}^{\infty} d\m r \int\limits_{\m r}^{\infty} d\m r' \int\limits_{0}^{\m r} d\rho_0 \frac{\rho_0^{d^2-d-1} (\m r \m r')^{d}}{\left((\m r^2-\rho_0^2)((\m r')^2-\rho_0^2)\right)^{\frac{1}{2}}}  e^{-\frac{\kappa_d}{2}\m r^d-\frac{\kappa_d}{2}(\m r')^d} e^{-\frac{\kappa_{d-1}}{2d}\m r^{d-1}(\sqrt{\m r^2-\rho_0^2}+\sqrt{(\m r')^2-\rho_0^2})}.
	\end{equation}
	We now bound the factor \(e^{-\frac{\kappa_d}{2}\m r^d-\frac{\kappa_d}{2}(\m r')^d}\leq e^{-\kappa_d\m r^d}\) in \eqref{eq_I_d-1} and treat the integrals with respect to \(\m r'\) and \(\rho_0\) with the change of variables \(\m b = \m r^{d-1}\sqrt{\m r^2-\rho_0^2}\in[0,\m r^d]\) and \(\m a=\m r^{d-1}\sqrt{(\m r')^2-\rho_0^2}\in[\m b,\infty[\) whose volume form satisfies $$d\m a d\m b \m a \m b \m r^{-4(d-1)}=d\m r'd\m \rho_0 \m r' \rho_0.$$ We then obtain
	\begin{align}\label{eq_bound_on_I_d-1}
		I_{d-1}(t) &\leq \int\limits_{t}^{\infty} d\m r \m r^{d^2-d-1} e^{-\kappa_d\m r^d} \int\limits_{0}^{\m r^d} d\m b \int\limits_{\m b}^{\infty} d \m a \left(1-\left(\frac{\m b}{\m r^d}\right)^2\right)^{\frac{d^2-d-2}{2}} \left(1+\left(\frac{\m a}{\m r ^d}\right)^2-\left(\frac{\m b}{\m r^d}\right)^2\right)^{\frac{d-1}2} e^{-\frac{\kappa_{d-1}}{2d}\m a -\frac{\kappa_{d-1}}{2d}\m b} \nonumber
		\\
		&\leq C \int\limits_{t}^{\infty} d\m r \m r^{d^2-d-1} e^{-\kappa_d\m r^d} \sim C t^{d^2-2d} e^{-\kappa_d t^d} = o\left(t^{d^2-d} e^{-\kappa_d t^d}\right) \text{ as } t\rightarrow\infty.
	\end{align}
	\\~\\
	\noindent\textbf{Step 5. Analytical treatment of \(I_{k}(t)\) for \(1\leq k\leq d-2\), see \eqref{eq_definition_I_k}}
	\\~\\
	For \(k\in\lbrace1,...,d-2\rbrace\), \(\Sp_{Q_0^\perp}\) is a \((d-k-1)\)-dimensional unit sphere. To estimate the inner integral with respect to \(\m u\), we substitute the variable \(\m u\) with
	\begin{equation}\label{eq_definition_z}
		\m z = \langle \m u,e_d \rangle \text{ and } \m s = \frac{\m u- \langle \m u,e_d \rangle e_d}{\|\m u- \langle \m u,e_d \rangle e_d\|},
	\end{equation}
	whose volume form transforms as
	\begin{equation*}
		d\m u = d\m s d\m z \; (1-\m z^2)^{\frac{d-k-3}{2}}
	\end{equation*}
	where \(d\m s\) denotes the surface measure on the \((d-k-2)\)-dimensional unit sphere. Since \(\m s\) does not appear in the integrand, we can obviously integrate with respect to \(\m s\), which only modifies the constant in front of the integral. Moreover, the condition \(\delta^2\geq(\m r')^2-\m r^2\) rewrites as \(\m z\leq\frac{\sqrt{\m r^2-\rho_0^2}}{\sqrt{(\m r')^2-\rho_0^2}}\), hence
	\begin{align}\label{eq_I_k_as_sum}
		I_k(t) & \leq C \int\limits_{t}^{\infty} d\m r \int\limits_{\m r}^{\infty} d\m r' \int\limits_{0}^{\m r} d\rho_0 \rho_0^{kd-1} (\m r\m r')^{(d-k)(d-1)+1} \left((\m r^2-\rho_0^2)((\m r')^2-\rho_0^2)\right)^{\frac{d-k-2}{2}} e^{-\frac{\kappa_d}{2}\m r^d-\frac{\kappa_d}{2}(\m r')^d}\nonumber
		\\
		& \hspace{2cm} \int\limits_{-1}^{\frac{\sqrt{\m r^2-\rho_0^2}}{\sqrt{(\m r')^2-\rho_0^2}}} d\m z (1-\m z)^{\frac{d-k-3}{2}} e^{-\frac{\kappa_{d-1}}{2d}\m r^{d-1}\sqrt{(\m r')^2-\rho_0^2+\m r^2-\rho_0^2-2\sqrt{(\m r')^2-\rho_0^2}\sqrt{\m r^2-\rho_0^2}\m z}}\nonumber
		\\
		& \leq C\left(I_k^{(A)}(t)+I_k^{(B)}(t)+I_k^{(C)}(t)\right)
	\end{align}
	where 
	\begin{align}
		I_k^{(A)}(t) & = \int\limits_{t}^{\infty} d\m r \int\limits_{\m r}^{\infty} d\m r' \int\limits_{0}^{\m r} d\rho_0 \rho_0^{kd-1} (\m r\m r')^{(d-k)(d-1)+1} \left((\m r^2-\rho_0^2)((\m r')^2-\rho_0^2)\right)^{\frac{d-k-2}{2}} e^{-\frac{\kappa_d}{2}\m r^d-\frac{\kappa_d}{2}(\m r')^d}\notag
		\\
		& \hspace{0.5cm} \int\limits_{-1}^{M(\m r,\m r',\rho_0)} d\m z (1-\m z^2)^{\frac{d-k-3}{2}} e^{-\frac{\kappa_{d-1}}{2d}\m r^{d-1}\sqrt{(\m r')^2-\rho_0^2+\m r^2-\rho_0^2-2\sqrt{(\m r')^2-\rho_0^2}\sqrt{\m r^2-\rho_0^2}\m z}},\label{eq:IkA}
		\\
		I_k^{(B)}(t) & = \int\limits_{t}^{\infty} d\m r \int\limits_{\m r}^{\infty} d\m r' \int\limits_{0}^{\m r} d\rho_0 \rho_0^{kd-1} (\m r\m r')^{(d-k)(d-1)+1} \left((\m r^2-\rho_0^2)((\m r')^2-\rho_0^2)\right)^{\frac{d-k-2}{2}} e^{-\kappa_d\m r^d}\nonumber
		\\
		& \hspace{0.5cm} \int\limits_{M(\m r,\m r',\rho_0)}^{\frac{\sqrt{\m r^2-\rho_0^2}}{\sqrt{(\m r')^2-\rho_0^2}}} d\m z 1_{\m r^2-\rho_0^2>\frac{1}{\m r^{2d-2-\varepsilon}}} (1-\m z^2)^{\frac{d-k-3}{2}} e^{-\frac{\kappa_{d-1}}{2d}\m r^{d-1}\sqrt{(\m r')^2-\rho_0^2+\m r^2-\rho_0^2-2\sqrt{(\m r')^2-\rho_0^2}\sqrt{\m r^2-\rho_0^2}\m z}},\label{eq:IkB}
		\\
		I_k^{(C)}(t) & = \int\limits_{t}^{\infty} d\m r \int\limits_{\m r}^{\infty} d\m r' \int\limits_{0}^{\m r} d\rho_0 \rho_0^{kd-1} (\m r\m r')^{(d-k)(d-1)+1} \left((\m r^2-\rho_0^2)((\m r')^2-\rho_0^2)\right)^{\frac{d-k-2}{2}} e^{-\kappa_d\m r^d}\notag
		\\
		& \hspace{0.5cm} \int\limits_{M(\m r,\m r',\rho_0)}^{\frac{\sqrt{\m r^2-\rho_0^2}}{\sqrt{(\m r')^2-\rho_0^2}}} d\m z 1_{\m r^2-\rho_0^2\leq\frac{1}{\m r^{2d-2-\varepsilon}}} (1-\m z^2)^{\frac{d-k-3}{2}} e^{-\frac{\kappa_{d-1}}{2d}\m r^{d-1}\sqrt{(\m r')^2-\rho_0^2+\m r^2-\rho_0^2-2\sqrt{(\m r')^2-\rho_0^2}\sqrt{\m r^2-\rho_0^2}\m z}},\label{eq:IkC}
		\end{align}
		where \(\varepsilon>0\) is a small positive constant and
		\begin{equation}\label{eq:defM}
		    M(\m r,\m r',\rho_0) = \max\left(\frac{\sqrt{\m r^2-\rho_0^2}}{\sqrt{(\m r')^2-\rho_0^2}}-\frac{1}{\m r^{2d-2-\varepsilon}\sqrt{(\m r')^2-\rho_0^2}\sqrt{\m r^2-\rho_0^2}},-1\right).
	    \end{equation}
		Let us notice that we use the inequality	 \(e^{-\frac{\kappa_{d}}{2}\m r^d-\frac{\kappa_d}{2}(\m r')^d}\leq e^{-\kappa_{d}\m r^d}\) in order to derive the integrals \(I_k^{(B)}(t)\) and \(I_k^{(C)}(t)\) given above. 
		
		In the lines below, we explain how to estimate each of the three integrals $I_k^{(A)}$, $I_k^{(B)}$ and $I_k^{(C)}$ defined at \eqref{eq:IkA}, \eqref{eq:IkB} and \eqref{eq:IkC} respectively.\\
		
		\noindent{\bf Upper bound for $I_k^{(A)}(t)$, see \eqref{eq:IkA}}
		
		The integral $I_k^{(A)}(t)$ is expected to be the smallest of the three. Indeed, the estimation below will show that the integral corresponds to the geometric situation when the two balls ${\mathcal B}$ and ${\mathcal B}'$ have an intersection small enough so that the volume of the union ${\mathcal B}\cup{\mathcal B}'$ slightly exceeds by some power of $t$ the volume of one ball of radius $t$. 
		
		We start by noticing that it is enough to estimate the inner integral with respect to \(\m z\) in $I_k^{(A)}(t)$ when $M(r,r',\rho_0)=\frac{\sqrt{\m r^2-\rho_0^2}}{\sqrt{(\m r')^2-\rho_0^2}}-\frac{1}{\m r^{2d-2-\varepsilon}\sqrt{(\m r')^2-\rho_0^2}\sqrt{\m r^2-\rho_0^2}}$. In this case, we bound the exponential function by its maximal value \(e^{-\frac{\kappa_{d-1}}{2d}\m r^{d-1}\sqrt{(\m r')^2-\m r^2+2\frac{1}{\m r^{2d-2-\varepsilon}}}}\), which is attained when \(\m z=\frac{\sqrt{\m r^2-\rho_0^2}}{\sqrt{(\m r')^2-\rho_0^2}}-\frac{1}{\m r^{2d-2-\varepsilon}\sqrt{(\m r')^2-\rho_0^2}\sqrt{\m r^2-\rho_0^2}}\). We then bound the integral of the remaining factor \((1-\m z^2)^{\frac{d-k-3}{2}}\) by a constant. With this strategy, we obtain
			\begin{align}\label{eq_int_z_I_k_A}
				&\int\limits_{-1}^{M(\m r,\m r',\rho_0)} d\m z (1-\m z^2)^{\frac{d-k-3}{2}} e^{-\frac{\kappa_{d-1}}{2d}\m r^{d-1}\sqrt{(\m r')^2-\rho_0^2+\m r^2-\rho_0^2-2\sqrt{(\m r')^2-\rho_0^2}\sqrt{\m r^2-\rho_0^2}\m z}} \nonumber
				\\
				&\leq C e^{-\frac{\kappa_{d-1}}{2d}\m r^{d-1}\sqrt{(\m r')^2-\m r^2+2\frac{1}{\m r^{2d-2-\varepsilon}}}} \leq C e^{-\frac{\kappa_{d-1}}{2d}\m r^{d-1}\sqrt{\frac{2}{\m r^{2d-2-\varepsilon}}}} 
				\leq C e^{-\frac{\kappa_{d-1}}{\sqrt{2}d}t^{\frac{\varepsilon}{2}}}
			\end{align}
			where we used \((\m r')^2-\m r^2\geq 0\) and \(\m r\geq t\) in the last two inequalities.
			We now insert the bound \eqref{eq_int_z_I_k_A} into the expression of \(I_k^{(A)}(t)\) at \eqref{eq:IkA} and evaluate the resulting expression when using the change of variables \(\m h = \frac{\rho_0}{\m r}\) whose volume form transforms as \(d\rho_0 = d \m h \m r\).
			\begin{align}\label{eq_bound_on_I_k_A}
				& I_k^{(A)}(t)\notag
				\\
				& \leq C e^{-\frac{\kappa_{d-1}}{\sqrt{2}d}t^{\frac{\varepsilon}{2}}} \int\limits_{t}^{\infty} d\m r \int\limits_{\m r}^{\infty} d\m r' \int\limits_{0}^{\m r} d\rho_0 \rho_0^{kd-1} (\m r\m r')^{(d-k)(d-1)+1} \left((\m r^2-\rho_0^2)((\m r')^2-\rho_0^2)\right)^{\frac{d-k-2}{2}} e^{-\frac{\kappa_d}{2}\m r^d-\frac{\kappa_d}{2}(\m r')^d} \nonumber
				\\
				& = C e^{-\frac{\kappa_{d-1}}{\sqrt{2}d}t^{\frac{\varepsilon}{2}}}
				\int\limits_{t}^{\infty} d\m r 
				\int\limits_{\m r}^{\infty} d\m r' 
				\int\limits_{0}^{1} d\m h \m r^{kd} 
				\m h^{kd-1} (\m r\m r')^{(d-k)(d-1)+1} \left(\m r^2(1-\m h^2)((\m r')^2-(\m r\m h)^2)\right)^{\frac{d-k-2}{2}} e^{-\frac{\kappa_d}{2}(\m r^d+(\m r')^d)}\nonumber
				\\
				& \leq C e^{-\frac{\kappa_{d-1}}{\sqrt{2}d}t^{\frac{\varepsilon}{2}}}
				\int\limits_{t}^{\infty} d\m r \m r^{d^2-1} e^{-\frac{\kappa_d}{2}\m r^d}
				\int\limits_{\m r}^{\infty} d\m r' (\m r')^{d^2-dk-1} e^{-\frac{\kappa_d}{2}(\m r')^d}\nonumber
				\\
				& \sim C e^{-\frac{\kappa_{d-1}}{\sqrt{2}d}t^{\frac{\varepsilon}{2}}}
				\int\limits_{t}^{\infty} d\m r \m r^{2d^2-dk-d-1} e^{-\kappa_d\m r^d} \sim C t^{2d^2-dk-2d}
				e^{-\kappa_d t^d-\frac{\kappa_{d-1}}{\sqrt{2}d}t^{\frac{\varepsilon}{2}}}
				 =o\left(t^{d^2-d}e^{-\kappa_{d}t^d}\right) \text{ as } t\rightarrow\infty,
			\end{align}
			where we use twice the estimate $\displaystyle\int_t^{\infty}{x}^{\alpha}e^{-{x}^d}d{x}\sim \frac1{d}t^{\alpha+1-d}e^{-t^d}$ when $t\to\infty$, for fixed $\alpha$.\\
			
			\noindent{\bf Upper bound for $I_k^{(B)}(t)$, see \eqref{eq:IkB}}
			
            We start by bounding the inner integral with respect to \(\m z\) in $I_k^{(B)}(t)$. We remark that under the condition \(\lbrace\m r^2-\rho_0^2>\frac{1}{\m r^{2d-2-\varepsilon}}\rbrace\) guaranteed by the indicator function, the quantity \(M(\m r, \m r',\rho_0)\) defined at \eqref{eq:defM} satisfies
			\begin{equation}\label{eq_90}
			    M(\m r, \m r',\rho_0) = \frac{\sqrt{\m r^2-\rho_0^2}}{\sqrt{(\m r')^2-\rho_0^2}} - \frac{1}{\m r^{2d-2-\varepsilon}\sqrt{\m r^2-\rho_0^2}\sqrt{(\m r')^2-\rho_0^2}}.
			\end{equation}
			Moreover, we bound the exponential function by its maximal value \(e^{-\frac{\kappa_{d-1}}{2d}\m r^{d-1}\sqrt{(\m r')^2-\m r^2}}\), attained when \(\m z=\frac{\sqrt{\m r^2-\rho_0^2}}{\sqrt{(\m r')^2-\rho_0^2}}\). 
			
			In order to estimate the factor \((1-\m z^2)^{\frac{d-k-3}{2}}\), we need to consider two cases.\\~\\
		$\bullet$ For \(k\in\lbrace1,...,d-3\rbrace\), the function \(\m z\mapsto(1-\m z^2)^{\frac{d-k-3}{2}}\) is decreasing, so we bound it by its maximal value, which is attained at the minimal value of \(\m z\), that is \(\m z = \frac{\sqrt{\m r^2-\rho_0^2}}{\sqrt{(\m r')^2-\rho_0^2}}-\frac{1}{\m r^{2d-2-\varepsilon}\sqrt{\m r^2-\rho_0^2}\sqrt{(\m r')^2-\rho_0^2}}\). Explicitly, this gives
				\begin{align}\label{eq_91}
				(1-\m z^2)^{\frac{d-k-3}{2}} & \leq(1-\frac{\m r^2-\rho_0^2}{(\m r')^2-\rho_0^2}+\frac{2}{\m r^{2d-2-\varepsilon} \sqrt{\m r^2-\rho_0^2} \sqrt{(\m r')^2-\rho_0^2} })^{\frac{d-k-3}{2}} \nonumber
				\\
				&= \frac{1}{\left(r^{2d-2-\varepsilon} \sqrt{\m r^2-\rho_0^2} \sqrt{(\m r')^2-\rho_0^2}\right)^{\frac{d-k-3}{2}}}
				\left(\frac{\m r^{2d-2}((\m r')^2-\m r^2)\sqrt{\m r^2-\rho_0^2}}{\m r^\varepsilon\sqrt{(\m r')^2-\rho_0^2}}+2\right)^{\frac{d-k-3}{2}}.
				\end{align}
				Subsequently, the integral with respect to \(\m z\) is bounded by the product of \eqref{eq_91} with the length of the interval of integration, which, thanks to \eqref{eq_90}, is \(\left(\m r^{2d-2-\varepsilon}\sqrt{\m r^2-\rho_0^2}\sqrt{(\m r')^2-\rho_0^2}\right)^{-1}\). We deduce that
				\begin{align*}
					& \int\limits_{M(\m r, \m r',\rho_0)}^{\frac{\sqrt{\m r^2-\rho_0^2}}{\sqrt{(\m r')^2-\rho_0^2}}} d\m z 1_{\m r^2-\rho_0^2>\frac{1}{\m r^{2d-2-\varepsilon}}} (1-\m z^2)^{\frac{d-k-3}{2}} e^{-\frac{\kappa_{d-1}}{2d}\m r^{d-1}\sqrt{(\m r')^2-\rho_0^2+\m r^2-\rho_0^2-2\sqrt{(\m r')^2-\rho_0^2}\sqrt{\m r^2-\rho_0^2}\m z}}
					\\
					& \leq \frac{1}{\left(\m r^{2d-2-\varepsilon} \sqrt{\m r^2-\rho_0^2} \sqrt{(\m r')^2-\rho_0^2}\right)^{\frac{d-k-3}{2}+1}}
					\left(\frac{\m r^{2d-2}((\m r')^2-\m r^2)\sqrt{\m r^2-\rho_0^2}}{\m r^\varepsilon \sqrt{(\m r')^2-\rho_0^2}}+2\right)^{\frac{d-k-3}{2}} \hspace{-5mm} e^{-\frac{\kappa_{d-1}}{2d}\m r^{d-1}\sqrt{(\m r')^2-\m r^2}}.
				\end{align*}
				We now insert the bound above into the expression for \(I_k^{(B)}(t)\) at \eqref{eq:IkB} and treat the resulting expression with the change of variable \(\m a = \m r^{d-1}\sqrt{(\m r')^2-\m r^2}\) and \(\m h = \frac{\rho_0}{\m r}\), whose volume form satisfies \(d\m a d\m h \m a \m r^{-2(d-1)+1} = d\m r' d\rho_0 \m r'\). We then obtain
				\begin{align}\label{eq_bound_on_I_k_B_1}
					I_k^{(B)}(t) &\leq C \int\limits_{t}^{\infty} d\m r \int\limits_{\m r}^{\infty} d\m r' \int\limits_{0}^{\m r} d\rho_0 \rho_0^{kd-1} \m r^{d+\frac{\varepsilon}{2}(d-k-1)}(\m r')^{(d-k)(d-1)+1} \left((\m r^2-\rho_0^2)((\m r')^2-\rho_0^2)\right)^{\frac{d-k-3}{4}} e^{-\kappa_d\m r^d} \nonumber
					\\
					&\hspace{2cm}\left(\frac{\m r^{2d-2}((\m r')^2-\m r^2)\sqrt{\m r^2-\rho_0^2}}{\m r^\varepsilon\sqrt{(\m r')^2-\rho_0^2}}+2\right)^{\frac{d-k-3}{2}} e^{-\frac{\kappa_{d-1}}{2d}\m r^{d-1}\sqrt{(\m r')^2-\m r^2}} \nonumber
					\\
					& = C \int\limits_{t}^{\infty} d\m r \m r^{d^2-d+\frac{\varepsilon}{2}(d-k-1)-1 } e^{-\kappa_{d}\m r^d} \int\limits_{0}^{\infty} d\m a \int\limits_{0}^{1} d\m h \m a \m h^{kd-1} \left(1+\left(\frac{\m a}{\m r^d}\right)^2\right)^{\frac{(d-k)(d-1)}{2}} \nonumber
					\\
					& \hspace{2cm}\left((1-\m h^2)\left(1+\left(\frac{\m a}{\m r^d}\right)^2-\m h^2\right)\right)^{\frac{d-k-3}{4}} \left(\frac{\m a^2\sqrt{1-\m h^2}}{\m r^\varepsilon\sqrt{1+\left(\frac{\m a}{\m r^d}\right)^2-\m h^2}}+2\right)^{\frac{d-k-3}{2}} e^{-\frac{\kappa_{d-1}}{2d}\m a}\nonumber
					\\
					& \leq C \int\limits_{t}^{\infty} d\m r \m r^{d^2-d+\frac{\varepsilon}{2}(d-k-1)-1 } e^{-\kappa_{d}\m r^d} \sim C t^{d^2-2d+\frac{\varepsilon}{2}(d-k-1)} e^{-\kappa_{d} t^d} = o\left(t^{d^2-d}e^{-\kappa_{d} t^d}\right) \text{ as } t\rightarrow\infty,
				\end{align}
				where we use in the last line the fact that both the integrals with respect to $\m a$ and $\m h$ are finite and uniformly bounded with respect to $t$.\\~\\
				$\bullet$ For \(k=d-2\), the function \(\m z\mapsto(1-\m z^2)^{\frac{d-k-3}{2}}=(1-\m z^2)^{-\frac{1}{2}}\) is increasing, so we bound it by its maximal value \(\frac{\sqrt{(\m r')^2-\rho_0^2}}{\sqrt{(\m r')^2-\m r^2}}\) attained when \(\m z = \frac{\sqrt{\m r^2-\rho_0^2}}{\sqrt{(\m r')^2-\rho_0^2}}\). As in the previous case, \eqref{eq_90} guarantees that the length of the interval of integration is equal to \(\left(\m r^{2d-2-\varepsilon}\sqrt{\m r^2-\rho_0^2}\sqrt{(\m r')^2-\rho_0^2}\right)^{-1}\). Consequently, we obtain
				\begin{align*}
					&\int\limits_{M(\m r,\m r',\rho_0)}^{\frac{\sqrt{\m r^2-\rho_0^2}}{\sqrt{(\m r')^2-\rho_0^2}}} d\m z 1_{\m r^2-\rho_0^2>\frac{1}{\m r^{2d-2-\varepsilon}}} (1-\m z^2)^{-\frac{1}{2}} e^{-\frac{\kappa_{d-1}}{2d}\m r^{d-1}\sqrt{(\m r')^2-\rho_0^2+\m r^2-\rho_0^2-2\sqrt{(\m r')^2-\rho_0^2}\sqrt{\m r^2-\rho_0^2}\m z}}
					\\
					&\hspace{3cm}\leq \frac{1}{\m r^{2d-2-\varepsilon}\sqrt{(\m r')^2-\m r^2}\sqrt{\m r^2-\rho_0^2}} e^{-\frac{\kappa_{d-1}}{2d}\m r^{d-1}\sqrt{(\m r')^2-\m r^2}},
				\end{align*}
				where we use again in the last inequality the fact that the exponential factor is bounded by its maximal value along the interval of integration, see the discussion following \eqref{eq_90}. 
				
				As we did in the case \(k\in\lbrace1,...,d-3\rbrace\), we insert this bound in the expression for \(I_k^{(B)}(t)\) and treat it with the same change of variables \(\m a = \m r^{d-1}\sqrt{(\m r')^2-\m r^2}\) and \(\m h = \frac{\rho_0}{\m r}\). Thus, we obtain 
				\begin{align}\label{eq_bound_on_I_k_2}
					I_{d-2}^{(B)}(t)&\leq C \int\limits_{t}^{\infty} d\m r \int\limits_{\m r}^{\infty} d\m r' \int\limits_{0}^{\m r} d\rho_0 \rho_0^{d^2-2d-1} (\m r \m r')^{2d-1} e^{-\kappa_{d}\m r^d} \frac{1}{\m r^{2d-2-\varepsilon}\sqrt{(\m r')^2-\m r^2}\sqrt{\m r^2-\rho_0^2}} e^{-\frac{\kappa_{d-1}}{2d}\m r^{d-1}\sqrt{(\m r')^2-\m r^2}} \nonumber
					\\
					& = C \int\limits_{t}^{\infty} d\m r \m r^{d^2-d-1+\varepsilon} e^{-\kappa_{d}\m r^d} \int\limits_{0}^{\infty} d\m a \int\limits_{0}^{1} d\m h \frac{ \m h^{d^2-2d-1} }{\sqrt{1-\m h^2}} \left(1+\left(\frac{\m a}{\m r^d}\right)^2\right)^{d-1} e^{-\frac{\kappa_{d-1}}{2d}\m a} \nonumber
					\\
					& \leq C \int\limits_{t}^{\infty} d\m r \m r^{d^2-d-1+\varepsilon} e^{-\kappa_{d}\m r^d} \sim t^{d^2-2d+\varepsilon} e^{-\kappa_{d} t^d} =o\left(t^{d^2-d} e^{-\kappa_{d} t^d}\right) \text{ as } t\rightarrow\infty.
				\end{align}
			
			\noindent{\bf Upper bound for $I_k^{(C)}(t)$, see \eqref{eq:IkC}}
			
			As for $I_k^{(A)}(t)$ and $I_k^{(B)}(t)$, we start by bounding the inner integral with respect to \(\m z\) in $I_k^{(C)}(t)$. Again, we bound the exponential factor by its maximal value which is the same as in the case of $I_k^{(B)}(t)$. Moreover, we bound the integral of the remaining factor \((1-\m z^2)^{\frac{d-k-3}{2}}\) by a constant. This gives
			\begin{align}
				& \int\limits_{M(\m r, \m r',\rho_0^2)}^{\frac{\sqrt{\m r^2-\rho_0^2}}{\sqrt{(\m r')^2-\rho_0^2}}} d\m z 1_{\m r^2-\rho_0^2\leq\frac{1}{\m r^{2d-2-\varepsilon}}} (1-\m z^2)^{\frac{d-k-3}{2}} e^{-\frac{\kappa_{d-1}}{2d}\m r^{d-1}\sqrt{(\m r')^2-\rho_0^2+\m r^2-\rho_0^2-2\sqrt{(\m r')^2-\rho_0^2}\sqrt{\m r^2-\rho_0^2}\m z}} \nonumber
				\\
				& \leq C 1_{\m r^2-\rho_0^2\leq\frac{1}{\m r^{2d-2-\varepsilon}}} e^{-\frac{\kappa_{d-1}}{2d}\m r^{d-1}\sqrt{(\m r')^2-\m r^2}}.
			\end{align}
			We now insert this bound into the expression of \(I_k^{(C)}(t)\) given at \eqref{eq:IkC} and apply the condition \(\lbrace\m r^2-\rho_0^2\leq\frac{1}{\m r^{2d-2-\varepsilon}}\rbrace\) of the indicator function on both the factor \((\m r^2-\rho_0^2)^{\frac{d-k-2}{2}}\) and the domain of integration of \(\rho_0\). This produces in particular an integral over \(\rho_0\in\left[r\sqrt{1-\frac{1}{\m r^{2d-\varepsilon}}},\m r\right]\). Moreover, we apply the same change of variables as in the case of $I_k^{(B)}$, i.e. \(\m a = \m r^{d-1}\sqrt{(\m r')^2-\m r^2}\) and \(\m h = \frac{\rho_0}{\m r}\). We then derive
			\begin{align}\label{eq_bound_on_I_k_C}
			&	I_k^{(C)}(t)\notag\\& \leq C \int\limits_{t}^{\infty} d\m r \int\limits_{\m r}^{\infty} d\m r' \hspace{-0.4cm} \int\limits_{r\sqrt{1-\frac{1}{\m r^{2d-\varepsilon}}}}^{\m r} \hspace{-0.4cm} d\rho_0 \rho_0^{kd-1} (\m r\m r')^{(d-k)(d-1)+1} \left(\frac{1}{\m r^{2d-2-\varepsilon}}((\m r')^2-\rho_0^2)\right)^{\frac{d-k-2}{2}} \hspace{-0.2cm} e^{-\kappa_d\m r^d} e^{-\frac{\kappa_{d-1}}{2d}\m r^{d-1}\sqrt{(\m r')^2-\m r^2}}\nonumber
				\\
				& \leq C \int\limits_{t}^{\infty} d\m r  \m r^{d^2+\frac{\varepsilon}{2}(d-k-2)-1} e^{-\kappa_d\m r^d} \nonumber
				\\
				& \hspace{3cm} \int\limits_{0}^{\infty} d\m a \hspace{-0.4cm} \int\limits_{\sqrt{1-\frac{1}{\m r^{2d-\varepsilon}}}}^{1} \hspace{-0.4cm} d\m h \m a \m h^{kd-1} \left(1+\left(\frac{\m a}{\m r^d}\right)^2\right)^{\frac{(d-k)(d-1)}{2}} \left(1+\left(\frac{\m a}{\m r^d}\right)^2-\m h^2\right)^{\frac{d-k-2}{2}} e^{-\frac{\kappa_{d-1}}{2d}\m a}\nonumber
				\\
				& \leq C \int\limits_{t}^{\infty} d\m r  \m r^{d^2+\frac{\varepsilon}{2}(d-k-2)-1} \left(1-\sqrt{1-\frac{1}{\m r^{2d-\varepsilon}}}\right) e^{-\kappa_d\m r^d} \int\limits_{0}^{\infty} d\m a \left(1+\left(\frac{\m a}{\m r^d}\right)^2\right)^{\frac{d(d-k)}{2}-1} e^{-\frac{\kappa_{d-1}}{2d}\m a} \nonumber
				\\
				&\leq \int\limits_{t}^{\infty} d\m r  \m r^{d^2-2d+\frac{\varepsilon}{2}(d-k)-1} e^{-\kappa_d\m r^d} \sim t^{d^2-3d+\frac{\varepsilon}{2}(d-k)} e^{-\kappa_d t^d} = o\left(t^{d^2-d}e^{-\kappa_d t^d}\right) \text{ as } t\rightarrow\infty
			\end{align}
			where, in the third inequality, we use \(\left(1-\sqrt{1-\frac{1}{\m r^{2d-\varepsilon}}}\right) \leq C \m r^{-2d+\varepsilon}\).\\ 
			
			\noindent{\bf Conclusion of Step 5}
			
			Finally, combining \eqref{eq_bound_on_I_k_A} \eqref{eq_bound_on_I_k_B_1}, \eqref{eq_bound_on_I_k_2} and \eqref{eq_bound_on_I_k_C} with \eqref{eq_I_k_as_sum}, we obtain for \(k\in\lbrace1,...,d-2\rbrace\),
			\begin{equation}\label{eq_bound_on_I_k}
				I_k(t) = o\left(t^{d^2-d}e^{-\kappa_{d}t^d}\right) \text{ as } t\rightarrow\infty.
			\end{equation}

	\noindent\textbf{Step 6. Analytical treatment of \(I_{0}(t)\), see \eqref{eq_definition_I_0}}
	\\~\\
	To estimate the inner integral with respect to \(\m u\in\Sp_{\R^{d}}\), we use the change of variables analogue to \eqref{eq_definition_z}, that is \(\m z = \langle \m u,e_d \rangle\) and \(\m s = \frac{\m u- \langle \m u,e_d \rangle e_d}{\|\m u- \langle \m u,e_d \rangle e_d\|}\in\Sp_{\R^{d-1}}\) whose volume form transforms as \(d\m u = d\m s d\m z \; (1-\m z^2)^{\frac{d-3}{2}}\) where \(d\m s\) denotes the surface measure on the \((d-2)\)-dimensional unit sphere. Again, since the variable \(\m s\) does not appear in the integrand, we integrate with respect to \(\m s\) and we obtain
	\begin{align}\label{eq_I_0_as_sum}
		I_0(t) & \leq C \int\limits_{t}^{\infty}d\m r  \int\limits_{\m r}^{\infty}d\m r' (\m r\m r')^{d^2-1} e^{-\frac{\kappa_d}{2}\m r^d-\frac{\kappa_d}{2}(\m r')^d } \int\limits_{-1}^{\frac{\m r}{\m r'}}d\m z \, (1-\m z^2)^{\frac{d-3}{2}} e^{- \frac{\kappa_{d-1}}{2d}\m r^{d-1}\sqrt{(\m r')^2+\m r^2-2\m r'\m r\m z}} \nonumber
		\\
		&\leq C\left(I_0^{(A)}(t)+I_0^{(B)}(t)\right)
	\end{align}
	where
	\begin{align*}
		& I_0^{(A)}(t) = \int\limits_{t}^{\infty}d\m r  \int\limits_{\m r}^{\infty}d\m r' (\m r\m r')^{d^2-1} e^{-\frac{\kappa_d}{2}\m r^d-\frac{\kappa_d}{2}(\m r')^d} \int\limits_{-1}^{\frac{\m r}{\m r'}-\frac{1}{\m r^{2d-\varepsilon}}}d\m z \, (1-\m z^2)^{\frac{d-3}{2}} e^{- \frac{\kappa_{d-1}}{2d}\m r^{d-1}\sqrt{(\m r')^2+\m r^2-2\m r'\m r\m z}}
		\\
		& I_0^{(B)}(t) = \int\limits_{t}^{\infty}d\m r  \int\limits_{\m r}^{\infty}d\m r' (\m r\m r')^{d^2-1} e^{-\kappa_d\m r^d } \int\limits_{\frac{\m r}{\m r'}-\frac{1}{\m r^{2d-\varepsilon}}}^{\frac{\m r}{\m r'}}d\m z \, (1-\m z^2)^{\frac{d-3}{2}} e^{- \frac{\kappa_{d-1}}{2d}\m r^{d-1}\sqrt{(\m r')^2+\m r^2-2\m r'\m r\m z}}
	\end{align*}
	and where \(\varepsilon>0\) is a small positive constant and to obtain \(I_0^{(B)}(t)\)  from \(I_0(t)\) we used \(e^{-\frac{\kappa_d}{2}\m r^d-\frac{\kappa_d}{2}(\m r')^d}\leq e^{-\kappa_d\m r^d}\).
	We now bound the inner integral with respect to \(\m z\) of \(I_0^{(A)}(t)\). We bound the exponential function by its greatest value \(e^{-\frac{\kappa_{d-1}}{2d}\m r^{d-1}\sqrt{(\m r')^2-\m r^2+2\frac{\m r\m r'}{\m r^{2d-\varepsilon}}}}\), which is attained for \(\m z=\frac{\m r}{\m r'}-\frac{1}{\m r^{2d-\varepsilon}}\) and the integral of the remaining factor \((1-\m z^2)^{\frac{d-3}{2}}\) by a constant. This way we obtain
	\begin{align}\label{eq_int_z_0_first}
		\int\limits_{-1}^{\frac{\m r}{\m r'}-\frac{1}{\m r^{2d-1}}}d\m z \, (1-\m z^2)^{\frac{d-3}{2}} e^{- 2\frac{\kappa_{d-1}}{2d}\m r^{d-1}\sqrt{(\m r')^2+\m r^2-2\m r'\m r\m z}} 
		& \leq C e^{-\frac{\kappa_{d-1}}{2d}\m r^{d-1}\sqrt{(\m r')^2-\m r^2+2\frac{\m r\m r'}{\m r^{2d-\varepsilon}}}} \nonumber
		\\
		& \leq C e^{-\frac{\kappa_{d-1}}{\sqrt{2}d}\m r^{d-1}\sqrt{\frac{\m r^2}{\m r^{2d-\varepsilon}}}} 
		\leq  C e^{-\frac{\kappa_{d-1}}{\sqrt{2}d}t^{\frac{\varepsilon}{2}}}
	\end{align}
	where in the last two inequalities we successively used \((\m r')^2-\m r^2\geq0\) and \(\m r^{d-1}\sqrt{\frac{\m r'\m r}{\m r^{2d-\varepsilon}}}\geq\m r^{d-1}\sqrt{ \frac{\m r^2}{\m r^{2d-\varepsilon}} }=\m r^{\frac{\varepsilon}{2}}\geq t^{\frac{\varepsilon}{2}}\). We now use the bound \eqref{eq_int_z_0_first} in \(I_0^{(A)}(t)\) and estimate the integrals with respect to \(\m r\) and \(\m r'\),
	\begin{align}\label{eq_bound_I_0_A}
		I_0^{(A)}(t) & \leq C e^{-\frac{\kappa_{d-1}}{\sqrt{2}d}t^{\frac{\varepsilon}{2}}} \int\limits_{t}^{\infty}d\m r  \int\limits_{\m r}^{\infty}d\m r' (\m r\m r')^{d^2-1} e^{-\frac{\kappa_d}{2}\m r^d -\frac{\kappa_d}{2}(\m r')^d } \nonumber
		\\
		& \sim C t^{2d^2-2d} e^{-\frac{\kappa_{d-1}}{\sqrt{2}d}t^{\frac{\varepsilon}{2}}} e^{-\kappa_d t^d} 
		= o\left(t^{d^2-d}e^{-\kappa_d t^d}\right) \text{ as } t\rightarrow\infty.
	\end{align}
	Concerning \(I_0^{(B)}(t)\), we start by estimating the inner integral with respect to \(\m z\). We bound the exponential function by its greatest value, attained for \(\m z=\frac{\m r}{\m r'}\),
	\begin{equation}\label{eq_92}
	    e^{- \frac{\kappa_{d-1}}{2d}\m r^{d-1}\sqrt{(\m r')^2+\m r^2-2\m r'\m r\m z}} \leq e^{-\frac{\kappa_{d-1}}{2d}\m r^{d-1}\sqrt{(\m r')^2-\m r^2}}.
	\end{equation}
	To complete the bound, we need to distinguish two cases.
	\\~\\
	$\bullet$ If \(d\geq 3\), the function \(\m z\mapsto(1-\m z^2)^{\frac{d-3}{2}}\) is decreasing, so we bound it with the value attained for the lower bound of \(\m z = \frac{\m r}{\m r'} - \frac{1}{\m r^{2d-\varepsilon}}\),
		\begin{align}\label{eq_93}
			(1-\m z^2)^{\frac{d-3}{2}}
			&
			\leq \left(1-\left(\frac{\m r}{\m r'}-\frac{1}{\m r^{2d-\varepsilon}}\right)^2\right)^{\frac{d-3}{2}} 
			= \left(\frac{(\m r')^2-\m r^2}{(\m r')^2} + \frac{2\m r}{\m r'\m r^{2d-\varepsilon}}-\frac{1}{\m r^{2(2d-\varepsilon)}}\right)^{\frac{d-3}{2}} \nonumber
			\\
			& \leq \left(\frac{(\m r')^2-\m r^2}{(\m r')^2}+\frac{2}{\m r^{2d-\varepsilon}}\right)^{\frac{d-3}{2}}
			= \m r^{-d^2+3d+\frac{\varepsilon}{2}(d-3)}
			\left(\frac{\m r^{2d-\varepsilon}((\m r')^2-\m r^2)}{(\m r')^2}+2\right)^{\frac{d-3}{2}}
		\end{align}
		where, in the second inequality, we bound \(-\frac{1}{\m r^{2(2d-\varepsilon)}}\leq0\) and \(\frac{\m r}{\m r'}\leq 1\). To bound the inner integral with respect to \(\m z\) of \(I^{(B)}_0(t)\), we apply the bounds \eqref{eq_93} and \eqref{eq_92}: as this fixes the integrand, integrating with respect to \(\m z\) corresponds to multiplying by the length of the interval of integration \(\frac{1}{\m r^{2d-\varepsilon}}\). Consequently, we obtain
		\begin{align*}
			&\int\limits_{\frac{\m r}{\m r'}-\frac{1}{\m r^{2d-\varepsilon}}}^{\frac{\m r}{\m r'}}d\m z \, (1-\m z^2)^{\frac{d-3}{2}} e^{- \frac{\kappa_{d-1}}{2d}\m r^{d-1}\sqrt{(\m r')^2+\m r^2-2\m r'\m r\m z}} 
			\\
			& \leq 
			\m r^{-d^2+d+\frac{\varepsilon}{2}(d-1)}
			\left(\frac{\m r^{2d-\varepsilon}((\m r')^2-\m r^2)}{(\m r')^2}+2\right)^{\frac{d-3}{2}} e^{-\frac{\kappa_{d-1}}{2d}\m r^{d-1}\sqrt{(\m r')^2-\m r^2}}.
		\end{align*}
		Now we use this bound in the definition of \(I_0^{(B)}(t)\) and treat the resulting expression with the change of variables \(\m a = \m r^{d-1}\sqrt{(\m r')^2-\m r^2}\) whose volume form satisfies \(d\m a \m a \m r^{-2(d-1)}=d\m r'\m r'\) to conclude
		\begin{align}\label{eq_bound_I_0_B_1}
			I_0^{(B)}(t) & \leq \int\limits_{t}^{\infty}d\m r  \int\limits_{\m r}^{\infty}d\m r' (\m r\m r')^{d^2-1} e^{-\kappa_d\m r^d } \m r^{-d^2+d+\frac{\varepsilon}{2}(d-1)} \left(\frac{\m r^{2d-\varepsilon}((\m r')^2-\m r^2)}{(\m r')^2}+2\right)^{\frac{d-3}{2}} e^{-\frac{\kappa_{d-1}}{2d}\m r^{d-1}\sqrt{(\m r')^2-\m r^2}} \nonumber
			\\
			& = \int\limits_{t}^{\infty}d\m r \m r^{d^2-d-1+\frac{\varepsilon}{2}(d-1)}e^{-\kappa_d\m r^d } \int\limits_{0}^{\infty}d\m a \m a \left(1+\left(\frac{\m a}{\m r^d}\right)^2\right)^{\frac{d^2-2}{2}} \left(\frac{\m r^{-\varepsilon} \m a^2}{ \left(1+\left(\frac{\m a}{\m r^d}\right)^2\right)}+2\right) e^{-\frac{\kappa_{d-1}}{2d}\m a} \nonumber
			\\
			& \leq C \int\limits_{t}^{\infty}d\m r \m r^{d^2-d-1+\frac{\varepsilon}{2}(d-1)}e^{-\kappa_d\m r^d } \sim t^{d^2-2d+\frac{\varepsilon}{2}(d-1)}e^{-\kappa_dt^d } = o\left( t^{d^2-d} e^{-\kappa_dt^d } \right) \text{ as } t\rightarrow\infty.
		\end{align}
		$\bullet$ If \(d=2\), we have
		\begin{equation*}
		    I_k^{(B)}(t) = \int\limits_t^\infty d\m r \int\limits_{\m r}^\infty d\m r' (\m r \m r')^3 e^{-\kappa_2\m r^2} \int\limits_{\frac{\m r}{\m r'}-\frac{1}{\m r^{4-\varepsilon}}}^{\frac{\m r}{\m r'}} \left(1-\m z^2\right)^{-\frac{1}{2}} e^{-\frac{\kappa_1}{4}\m r^2\sqrt{(\m r')^2+\m r^2-2\m r\m r' \m z}}.
		\end{equation*}
		The function \(\m z\mapsto(1-\m z^2)^{-\frac{1}{2}}\) is increasing, so we bound it with the value attained for the maximal value of \(\m z = \frac{\m r}{\m r'} \) obtaining 
		\(	(1-\m z^2)^{-\frac{1}{2}} \leq \frac{\m r'}{\sqrt{(\m r')^2-\m r^2}}\). To bound the inner integral with respect to \(\m z\) of \(I^{(B)}_0(t)\) we apply this bound combined with \eqref{eq_92}, which fixes its integrand. Then, integrating with respect to \(\m z\) corresponds to multiplying by the length of the domain of integration of \(\m z\) which is \(\m r^{-4+\varepsilon}\). This way we obtain
		\begin{equation*}
			\int\limits_{\frac{\m r}{\m r'}-\frac{1}{\m r^{2d-\varepsilon}}}^{\frac{\m r}{\m r'}}d\m z \,  (1-\m z^2)^{-\frac{1}{2}} e^{- \frac{\kappa_1}{4}\m r^{d-1}\sqrt{(\m r')^2+\m r^2-2\m r'\m r\m z}} \leq \m r^{-4+\varepsilon} \frac{\m r'}{\sqrt{(\m r')^2-\m r^2}} e^{-\frac{\kappa_1}{4}\m r \sqrt{(\m r')^2-\m r^2}}.
		\end{equation*}
		We insert this bound in the expression for \(I_0^{(B)}(t)\) and use the change of variables \(\m a = \m r \sqrt{(\m r')^2-\m r^2}\) with \(d\m a \m a \m r^{-2}=d\m r' \m r'\) to obtain,
		\begin{align}\label{eq_bound_I_0_B_2}
			I_0^{(B)}(t) & \leq \int\limits_{t}^{\infty}d\m r  \int\limits_{\m r}^{\infty}d\m r' (\m r\m r')^{3} e^{-\kappa_2\m r^2 } \frac{1}{\m r^{4-\varepsilon}} \frac{\m r'}{\sqrt{(\m r')^2-\m r^2}} e^{-\frac{\kappa_1}{4}\m r \sqrt{(\m r')^2-\m r^2}} \nonumber
			\\
			& = \int\limits_{t}^{\infty}d\m r \m r^{\varepsilon+1} e^{-\kappa_2\m r^2 } \int\limits_{0}^{\infty} d\m a \left(1+\left(\frac{\m a}{\m r^2}\right)^2\right)^{\frac{3}{2}} e^{-\frac{\kappa_1}{4} \m a} \nonumber
			\\
			& \leq C \int\limits_{t}^{\infty}d\m r \m r^{\varepsilon+1} e^{-\kappa_2\m r^2 } \sim t^\varepsilon e^{-\kappa_2 t^2 } = o\left(t^{2} e^{-\kappa_2 t^2 }\right) \text{ as } t\rightarrow\infty.
		\end{align}
		Using \eqref{eq_bound_I_0_A}, \eqref{eq_bound_I_0_B_1} and \eqref{eq_bound_I_0_B_2} in \eqref{eq_I_0_as_sum} we prove that
	\begin{equation}\label{eq_bound_on_I_0}
		I_0(t) = o\left(t^{d^2-d}e^{-\kappa_{d}t^d}\right) \text{ as } t\rightarrow\infty.
	\end{equation}
	\noindent{\bf Conclusion of the proof of Proposition \ref{thm_E_pairs_negligible}}\\~\\
	Finally, we use the bounds \eqref{eq_bound_on_I_d-1}, \eqref{eq_bound_on_I_k} and \eqref{eq_bound_on_I_0} in \eqref{eq_bound_on_couple_vert_with_fixed_Q} to conclude that
	\begin{equation*}
		\E\left[\Card\left(\left(\VOMR\right)^2_{\neq}\right)\right] = o\left(t^{d^2-d}e^{-\kappa_{d} t^d}\right)  = o\left(\E\left[\Card\left(\VOMR\right)\right]\right)  \text{ as } t\rightarrow\infty.
	\end{equation*}
	\end{proof}
	
	\subsection{From expectation to probability: proof of Theorem \ref{thm_main_result}}\label{sec_proof_of_theorem}
	By definition, \(\mathcal{D}\) is the distance from \(0\) to the furthest away point of \(\CO\) and this point is a pointy vertex. Therefore, we have \(\Rcirc \geq t\) precisely when there exists at least one pointy vertex at distance larger than \(t\) from \(0\). That is, we have the equality of events
	\begin{equation*}
		\lbrace\Rcirc \geq t \rbrace = \lbrace \Card (\VOMR) \geq 1 \rbrace.
	\end{equation*}
	This characterization gives us access to the asymptotics of the distribution tail of \(\Rcirc\). Indeed,
	\begin{equation}\label{eq_proof_of_main_result}
		\E\left[\Card(\VOMR)\right] - \E\left[\Card\left((\VOMR)^2_{\neq}\right)\right] \leq \PR\left(\Rcirc \geq t\right) = \PR\left(\Card(\VOMR)\geq 1\right) \leq \E\left[\Card(\VOMR)\right].
	\end{equation}
	The inequality in the right-hand side of (\ref{eq_proof_of_main_result}) is Markov's inequality. Concerning the left-hand side inequality, we have
	\begin{equation*}
		\Card\left((\VOMR)^2_{\neq}\right) = \binom{\Card(\VOMR)}{2}
	\end{equation*}
	and therefore,
	\begin{align*}
		\E[\Card(\VOMR)] - \E\left[\Card\left((\VOMR)^2_{\neq}\right)\right] &= \sum\limits_{k=1}^\infty \PR\left(\Card(\VOMR)\geq k\right) - \sum\limits_{l=1}^\infty \PR\left(\Card\left((\VOMR)^2_{\neq}\right)\geq l\right)
		\\
		& = \PR\left(\Card(\VOMR)\geq 1\right) + \sum\limits_{k=2}^\infty \PR\left(\Card(\VOMR)\geq k\right) - \sum\limits_{l=1}^\infty \PR\left(\Card\left((\VOMR)^2_{\neq}\right)\geq l\right)
		\\
		& = \PR\left(\Card(\VOMR)\geq 1\right) - \sum\limits_{l\geq1, l\notin\lbrace\binom{k}{2}|k\in\mathbb{N}\rbrace} \PR\left(\Card\left((\VOMR)^2_{\neq}\right)\geq l\right)
		\\
		& \leq \PR\left(\Card(\VOMR)\geq 1\right).
	\end{align*}
	Finally, Proposition \ref{thm_E_pairs_negligible} shows that
	\begin{equation*}
		\E[\Card\left((\VOMR)^2_{\neq}\right)] = o(\E[\Card(\VOMR)]) \text{ as } t\rightarrow\infty
	\end{equation*}
	and therefore, (\ref{eq_proof_of_main_result}) implies that the tail probability of \(\mathcal{D}\) has the same asymptotic behavior of \(\E[\Card(\VOMR)]\). Proposition \ref{thm_E_nr_locmaxvert} provides the expectation \(\E[\Card(\VOMR)]\) up to the multiplicative constant \(C_d\). The value of \(C_d\) is given by Theorem \ref{thm_value_of_C_d}, which is proven in Section \ref{sec_volume_kingman}.
	
	\section{Mean volume of a random simplex: computation of \(C_d\) and proof of Theorem \ref{thm_value_of_C_d}}\label{sec_volume_kingman}
	In this section we compute the constant \(C_d\) of Theorem \ref{thm_value_of_C_d}, that is, the average volume of the random simplex \(\Conv\left(\m U_0,...,\m U_d\right)\) determined by \((d+1)\) i.i.d. uniform random points \(\m U_0,...,\m U_d\in\Sp_{\R^d}\) under the condition that \(\m U_1,...,\m U_d\) satisfy the geometrical pointy condition of Lemma \ref{prop_equivalent_characterization_of_VOM} with respect to \(\m U_0\):
	\begin{equation*}
		C_d = \E\left[ \Delta_{d}(\m U_0,...,\m U_d) \1_{0\in\Conv(P_{\m U_0^\perp}(\m U_1),...,P_{\m U_0^\perp}(\m U_d))} \right].
	\end{equation*}
	By means of a series of geometrical considerations in which the pointy condition plays a central role, we show in Proposition \ref{prop_recursive_relation_on_vol_full_simplex} that \(C_d\) satisfies a recurrence relation on \(d\). This relation, together with the explicit computation of \(C_2\) allows us to give an explicit value to \(C_d\).
	\begin{proposition}\label{prop_recursive_relation_on_vol_full_simplex}
		The expectation \(C_d\) given in identity (\ref{eq_C_d_as_expectation}) satisfies the recursive relation:
		\begin{equation}\label{eq_prop_recursive_formula_on_C_d}
			C_d = \frac{1}{2(d-1)} \left( \frac{B\left(\frac{d}{2},\frac{1}{2}\right)}{B\left(\frac{d-1}{2},\frac{1}{2}\right)} \right)^{d-1} C_{d-1}.
		\end{equation}
		As a consequence,
		\begin{equation}\label{eq_prop_avg_volume_simplex_explicit_value}
			C_d = 		\frac{1}{2^{d-1}\sqrt{\pi}(d-1)!}\frac{\Gamma\left(\frac{d}{2}\right)^{d}}{\Gamma\left(\frac{d+1}{2}\right)^{d-1}}.
		\end{equation}
	\end{proposition}
	\noindent{\bf Remark.} This result follows Miles' calculation of the moments of the volume of a random simplex in the unit ball. Indeed, in \cite{Miles}, he obtains in particular
	\begin{equation}\label{eq_Miles}
		\E\left[\Delta_{d}(\m U_0,...,\m U_d)\right] = \frac{1}{d!} \frac{\Gamma\left(\frac{d^2+1}{2}\right)}{\Gamma\left(\frac{d^2}{2}\right)} \frac{\Gamma\left(\frac{d}{2}\right)}{\Gamma\left(\frac{1}{2}\right)}
		\left(\frac{\Gamma\left(\frac{d}{2}\right)}{\Gamma\left(\frac{d+1}{2}\right)}\right)^d.
	\end{equation}
	His method provides additionally all the moments, which is not possible in our situation with the constraint given by the indicator function. We assert that the mean of the volume of the random simplex conditional on the event \(\{0\in\Conv(P_{\m U_0^\perp}(\m U_1),...,P_{\m U_0^\perp}(\m U_d))\}\) is  larger than the right-hand side of \eqref{eq_Miles}. This follows from the calculation of \(C_d\) in Proposition \ref{prop_recursive_relation_on_vol_full_simplex} and the calculation of the probability for the pointy condition given by Wendel's formula \cite{Wendel}, i.e.
	\(\PR(0\in\Conv(P_{\m U_0^\perp}(\m U_1),...,P_{\m U_0^\perp}(\m U_d)))=2^{-(d-1)}\).
	We then obtain for any \(d\geq 1\)
	\begin{align*}
	& \frac{\E\left[\Delta_{d}(\m U_0,...,\m U_d)\,|\,0\hspace{-1mm}\in\hspace{-1mm}\Conv\left(P_{\m U_0^\perp}(\m U_1),...,P_{\m U_0^\perp}(\m U_d)\right)\right]}{\E\left[\Delta_{d}(\m U_0,...,\m U_d)\right]}
	\\
	& =
	\frac{C_d}{\PR\left(0\hspace{-1mm}\in\hspace{-1mm}\Conv\left(P_{\m U_0^\perp}(\m U_1),...,P_{\m U_0^\perp}(\m U_d)\right)\right)\E\left[\Delta_{d}(\m U_0,...,\m U_d)\right]}
	\\
	& =\frac{d\Gamma(\frac{d^2}2)\Gamma(\frac{d+1}2)}{\Gamma(\frac{d^2+1}2)\Gamma(\frac{d}2)}\geq 1.
	\end{align*}
\begin{proof}
	We start by writing explicitly the expectations hidden behind the constants \(C_d\) and \(C_{d-1}\) in statement (\ref{eq_prop_recursive_formula_on_C_d}). Let \(\m U_0,...,\m U_d\) be i.i.d. uniform random variables on the \((d-1)\)-dimensional unit sphere \(\Sp_{\R^d}\) and \(\m V_0,...,\m V_{d-1}\) be i.i.d. uniform random variables on the \((d-2)\)-dimensional unit sphere \(\Sp_{\R^{d-1}}\). Our goal is to prove that \begin{align}\label{eq_prop_avg_volume_simplex_recursive_formula}
		& \E\left[\Delta_{d}(\m U_0,...,\m U_d) \1_{0\in\Conv(P_{\m U_0^\perp}(\m U_1),...,P_{\m U_0^\perp}(\m U_d))}\right] \nonumber
		\\
		& = \frac{1}{2(d-1)} \left( \frac{B\left(\frac{d}{2},\frac{1}{2}\right)}{B\left(\frac{d-1}{2},\frac{1}{2}\right)} \right)^{d-1} \E\left[\Delta_{d-1}(\m V_0,...,\m V_{d-1}) \1_{0\in\Conv(P_{\m V_0^\perp}(\m V_1),...,P_{\m V_0^\perp}(\m V_{d-1}))}\right].
	\end{align}
	\noindent\textbf{Step 1. Represent \(\m U_1,...,\m U_d\) with cylindrical coordinates.}
	\\~\\
	As in Figure \ref{fig_cylindrical_coordinates}, we express \(\m U_1,...,\m U_d\) in cylindrical coordinates with respect to the axis of \(\m U_0\):
	\begin{equation}\label{eq_cylindrical_coordinates}
		\m Z_{i-1} = \sqrt{1-\|P_{\m U_0^\perp}(\m U_i)\|^2} \text{ and } \m V_{i-1} = \frac{P_{\m U_0^\perp}(\m U_i)}{\|P_{\m U_0^\perp}(\m U_i)\|} \text{, } i\in\lbrace1,...,d\rbrace,
	\end{equation}
	where \(\m V_{i-1}\in \m U_0^\perp\) is the unit vector on \(\m U_0^\perp\) pointing from \(0\) towards the projection of \(\m U_i\) onto \(\m U_0^\perp\) and \(\m Z_{i-1}\) is the oriented height of \(\m U_i\) over \(\m U_0\). We obtain this way \(2d\) independent random variables: \(\m V_0,...,\m V_{d-1}\) i.i.d. random variables uniformly distributed on \(\Sp_{\m U_0^\perp}\),the \((d-2)-\)dimensional unit sphere laying on \(\m U_0^\perp\), and \(\m Z_0,...,\m Z_{d-1}\) i.i.d. random variables of probability density function \(\m  z \mapsto \frac{\1_{\m z\in[-1,1]}}{B\left(\frac{d-1}{2},\frac{1}{2}\right)} (1-\m z^2)^{\frac{d-3}{2}}\). In these cylindrical coordinates,
	\begin{equation*}
		\m U_i = \sqrt{1-\m Z_{i-1}^2} \m V_{i-1} - \m Z_{i-1}\m U_0,\, i\in\lbrace1,...,d\rbrace.
	\end{equation*}
	
	\begin{figure}[h!]
		\centering
		\includegraphics[scale=1]{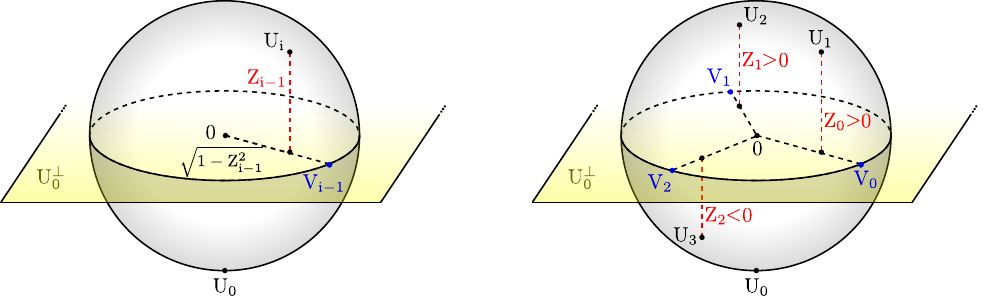}
		\caption{Representation of the cylindrical change of coordinates with respect to the axis spanned by \(\Urm_0\).}
		\label{fig_cylindrical_coordinates}
	\end{figure}
	\noindent\textbf{Step 2. Express the argument of the left-hand side of  (\ref{eq_prop_avg_volume_simplex_recursive_formula}) in the cylindrical coordinates and simplify it by means of operations on the determinant and geometrical considerations.}
	\\~\\
	With the cylindrical coordinates introduced in \text{Step 1}, the indicator function of the argument of the left-hand side of (\ref{eq_prop_avg_volume_simplex_recursive_formula}) takes the form
	\begin{equation}\label{eq_Cd_first_steps_on_pointy_condition_in_cylindrical_coord}
		 \1_{0\in\Conv(P_{\m U_0^\perp}(\m U_1),...,P_{\m U_0^\perp}(\m U_d))} = \1_{0\in\Conv(\sqrt{1-\m Z_{0}^2}\m V_0,...,\sqrt{1-\m Z_{d-1}^2}\m V_{d-1})} = \1_{0\in\Conv(\m V_0,...,\m V_{d-1})}
	\end{equation}
	where the last equality holds because the simplex \(\Conv(\sqrt{1-\m Z_{0}^2}\m V_0,...,\sqrt{1-\m Z_{d-1}^2}\m V_{d-1})\) contains the origin precisely when \(\Conv(\m V_0,...,\m V_{d-1})\) does. For what concerns the random volume on the left-hand side of (\ref{eq_prop_avg_volume_simplex_recursive_formula}), using the definition of the volume of a simplex with vertices at \(\m x_0,...,\m x_d\in\R^d\),
	\begin{equation*}
		\Delta_d(\m x_0,...,\m x_d) = \frac{1}{d!} \left| \text{det} \left(\m x_1-\m x_0,...,\m x_d-\m x_0 \right)\right|
	\end{equation*}
	and by elementary operations on the determinant of a matrix, we have the first four equalities of the following computation,
	\begin{align}\label{eq_Cd_first_steps_on_volume_in_cylindrical_coord}
		\Delta_{d}(\m U_0,...,\m U_d) \nonumber
		& = 
		\Delta_{d}\left(\m U_0,\sqrt{1-\m Z_{0}^2} \m V_{0} - \m Z_{0}\m U_0,...,\sqrt{1-\m Z_{d-1}^2} \m V_{d-1} - \m Z_{d-1}\m U_0\right) \nonumber
		\\
		& = \frac{1}{d!} \left| \text{det}\left( \sqrt{1-\m Z_{0}^2} \m V_{0} - (1+\m Z_{0})\m U_0,...,\sqrt{1-\m Z_{d-1}^2} \m V_{d-1} - (1+\m Z_{d-1})\m U_0 \right) \right| \nonumber
		\\
		& = (1+\m Z_0)\cdots(1+\m Z_{d-1}) \frac{1}{d!}  \left| \text{det}\left( \sqrt{\frac{1-\m Z_{0}}{1+\m Z_0}} \m V_{0} - \m U_0,...,\sqrt{\frac{1-\m Z_{d-1}}{1+\m Z_{d-1}}} \m V_{d-1} - \m U_0 \right) \right| \nonumber
		\\
		& = (1+\m Z_0)\cdots(1+\m Z_{d-1}) \Delta_d\left( \m U_0,\sqrt{\frac{1-\m Z_{0}}{1+\m Z_0}} \m V_{0},...,\sqrt{\frac{1-\m Z_{d-1}}{1+\m Z_{d-1}}} \m V_{d-1} \right) \nonumber
		\\
		& = (1+\m Z_0)\cdots(1+\m Z_{d-1}) \frac{\Delta_{d-1}(\sqrt{\frac{1-\m Z_{0}}{1+\m Z_0}} \m V_{0},...,\sqrt{\frac{1-\m Z_{d-1}}{1+\m Z_{d-1}}}\m V_{d-1})}{d}.
	\end{align}
	\noindent The last equality of (\ref{eq_Cd_first_steps_on_volume_in_cylindrical_coord}) holds because the simplex \(\Conv\left( \m U_0,\sqrt{\frac{1-\m Z_{0}}{1+\m Z_0}} \m V_{0},...,\sqrt{\frac{1-\m Z_{d-1}}{1+\m Z_{d-1}}} \m V_{d-1} \right)\) is a pyramid with \((d-1)\)-dimensional base \(\Conv\left(\sqrt{\frac{1-\m Z_{0}}{1+\m Z_0}} \m V_{0},...,\sqrt{\frac{1-\m Z_{d-1}}{1+\m Z_{d-1}}} \m V_{d-1} \right) \subseteq \m U_0^\perp\) and height \(1\), see Figure \ref{fig_simplex_with_base_on_U_0_perp}.
	\begin{figure}[h!]
		\centering
		\includegraphics[scale=1]{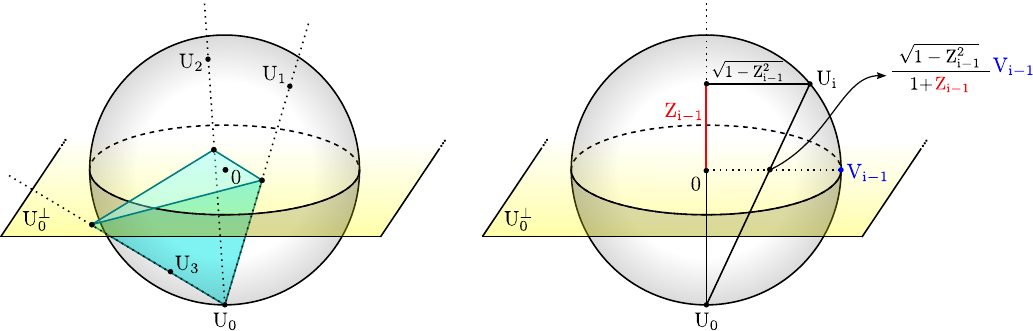}
		\caption{On the left-hand side, the \(d\)-dimensional simplex \(\Conv\left(\Urm_0,\sqrt{\frac{1-\m Z_{0}}{1+\m Z_0}} \Vrm_{0},...,\sqrt{\frac{1-\m Z_{d-1}}{1+\m Z_{d-1}}} \Vrm_{d-1}\right)\).
		On the right-hand side, geometrical details showing that the point \(\sqrt{\frac{1-\m Z_{i-1}}{1+\m Z_{i-1}}} \Vrm_{i-1}\) is found at the intersection of the plane \(\Urm_0^\perp\) with the line passing by \(\Urm_0\) and \(\Urm_i\).}
		\label{fig_simplex_with_base_on_U_0_perp}
	\end{figure}
	Under the condition \(0\in\Conv(\m V_0,...,\m V_{d-1})\), the simplex \(\Conv\left(\sqrt{\frac{1-\m Z_{0}}{1+\m Z_0}} \m V_{0},...,\sqrt{\frac{1-\m Z_{d-1}}{1+\m Z_{d-1}}} \m V_{d-1}\right)\) contains \(0\) as well and so, it can be partitioned into \(d\) subsimplices pinned at \(0\), see Figure \ref{fig_conv_V_1_V_d-1}. Therefore, we can express its volume as the sum
	\begin{align}\label{eq_02}
	    & \Delta_{d-1}\left(\sqrt{\frac{1-\m Z_{0}}{1+\m Z_0}} \m V_{0},...,\sqrt{\frac{1-\m Z_{d-1}}{1+\m Z_{d-1}}} \m V_{d-1}\right) \nonumber
	    \\
	    & = \sum_{i=1}^{d-1} \Delta_{d-1} \left(0,\sqrt{\frac{1-\m Z_{0}}{1+\m Z_0}} \m V_{0},...,\widehat{\sqrt{\frac{1-\m Z_{i}}{1+\m Z_i}} \m V_{i}},...,\sqrt{\frac{1-\m Z_{d-1}}{1+\m Z_{d-1}}} \m V_{d-1}\right) \nonumber
	    \\
	    & = \sum_{i=1}^{d-1}  \left(\prod_{j\in\lbrace0,...,d-1\rbrace\setminus\lbrace i \rbrace} \sqrt{\frac{1-\m Z_j}{1+\m Z_j}}\right) \Delta_{d-1} \left(0, \m V_{0},...,\widehat{\m V_{i}},...,\m V_{d-1}\right)
	\end{align}
	where the last equality holds by the linearity of the determinant.
	\begin{figure}[h!]
		\centering
		\includegraphics[scale=1]{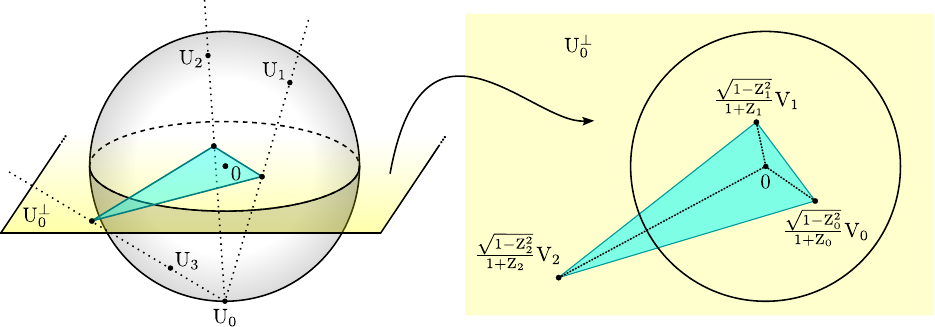}
		\caption{On the left-hand side, the \((d-1)\)-dimensional simplex \(\Conv\left(\sqrt{\frac{1-\m Z_{0}}{1+\m Z_0}} \Vrm_{0},...,\sqrt{\frac{1-\m Z_{d-1}}{1+\m Z_{d-1}}} \Vrm_{d-1}\right)\) and on the right-hand side, its partition in \(d\) subsimplices pinned at \(0\).}
		\label{fig_conv_V_1_V_d-1}
	\end{figure}
	We now multiply (\ref{eq_Cd_first_steps_on_volume_in_cylindrical_coord}) and (\ref{eq_Cd_first_steps_on_pointy_condition_in_cylindrical_coord}) and use the identity \eqref{eq_02} to get
	\begin{align}\label{eq_Cd_end_of_step2}
		& \Delta_{d}(\m U_0,...,\m U_d) \1_{0\in\Conv(P_{\m U_0^\perp}(\m U_1),...,P_{\m U_0^\perp}(\m U_d))} \nonumber
		\\
		& = \frac{1}{d} (1+\m Z_0)\cdots(1+\m Z_{d-1}) \Delta_{d-1}\left(\sqrt{\frac{1-\m Z_{0}}{1+\m Z_0}} \m V_{0},...,\sqrt{\frac{1-\m Z_{d-1}}{1+\m Z_{d-1}}} \m V_{d-1}\right) \1_{0\in\Conv(\m V_0,...,\m V_{d-1})} \nonumber
		\\
		& = \frac{1}{d} \sum\limits_{i=0}^{d-1} (1+\m Z_0)\cdots(1+\m Z_{d-1}) \left(\prod\limits_{j\in\lbrace0,...,d-1\rbrace\setminus\lbrace i\rbrace} \sqrt{\frac{1-\m Z_{j}}{1+\m Z_{j}}} \right) \Delta_{d-1}(0,\m V_0,...,\widehat{\m V_i}...,\m V_{d-1}) \1_{0\in\Conv(\m V_0,...,\m V_{d-1})} \nonumber
		\\
		& = \frac{1}{d} \sum\limits_{i=0}^{d-1} (1+\m Z_i) \left( \prod_{j\in\lbrace0,...,d-1\rbrace\setminus\lbrace i\rbrace} \sqrt{1-\m Z_j^2}\right) \Delta_{d-1}(0,\m V_0,...,\widehat{\m V_i}...,\m V_{d-1}) \1_{0\in\Conv(\m V_0,...,\m V_{d-1})}.
	\end{align}

	\noindent \textbf{Step 3. Analysis of the expectation of (\ref{eq_Cd_end_of_step2}).}
	\\~\\
	We now take the expectation of the identity (\ref{eq_Cd_end_of_step2}) and use the linearity of the expectation and the fact that the \(d\) summands in the right hand side of (\ref{eq_Cd_end_of_step2}) are identically distributed to write
	\begin{align}\label{eq_Cd_extraction_factors_Z}
		C_d & = \E\left[\Delta_{d}(\m U_0,...,\m U_d) \1_{0\in\Conv(P_{\m U_0^\perp}(\m U_1),...,P_{\m U_0^\perp}(\m U_d))}\right] \nonumber
		\\
		& = \frac{1}{d} \sum_{i=0}^{d-1} \E\left[ (1+\m Z_i) \left( \prod_{j\in\lbrace0,...,d-1\rbrace\setminus\lbrace i\rbrace} \sqrt{1-\m Z_j^2}\right) \Delta_{d-1}(0,\m V_0,...,\widehat{\m V_i}...,\m V_{d-1}) \1_{0\in\Conv(\m V_0,...,\m V_{d-1})}\right] \nonumber
		\\
		& = \E\left[ (1+\m Z_{d-1}) \left(\prod\limits_{j\in\lbrace0,...,d-2\rbrace} \sqrt{\frac{1-\m Z_{j}}{1+\m Z_{j}}} \right) \Delta_{d-1}\left(0,\m V_{0},...,\m V_{d-2}\right) \1_{0\in\Conv(\m V_0,...,\m V_{d-1})} \right] \nonumber
		\\
		& = \E\left[ 1+\m Z_{d-1} \right] \left(\prod\limits_{j\in\lbrace0,..., d-2\rbrace}  \E\left[ \sqrt{1-\m Z_j^2}\right] \right) \E\left[ \Delta_{d-1}\left(0,\m V_{0},...,\m V_{d-2}\right) \1_{0\in\Conv(\m V_0,...,\m V_{d-1})} \right]
	\end{align}
	where the last equality holds because the variables \(\m Z_0,...,\m Z_{d-1},\m V_0,...,\m V_{d-1}\) are independent. As the law of \(\m Z_0,...,\m Z_{d-1}\) is known, the expectations relative to these variables can be easily computed.
	
	From this point until the end of Step 3, we will find the relation between the average volume factor in the right-hand side of (\ref{eq_Cd_extraction_factors_Z}), that is \(\E\left[ \Delta_{d-1}\left(0,\m V_{0},...,\m V_{d-2}\right) \1_{0\in\Conv(\m V_0,...,\m V_{d-1})} \right]\), and that of the right-hand side of equation (\ref{eq_prop_avg_volume_simplex_recursive_formula}).
		We start by remarking that conditionally on \(\m V_0,...,\m V_{d-2}\), the condition \(0\in\Conv(\m V_0,...,\m V_{d-1})\) is satisfied whenever \(\m V_{d-1}\) falls in a region diametrically opposed to \(\m V_0,...,\m V_{d-2}\). More precisely, we have the equality of events
	\begin{equation*}
		\{0\in\Conv(\m V_0,...,\m V_{d-1})\} = \{\m V_{d-1}\in \text{SpConv}(-\m V_0,...,-\m V_{d-2})\}
	\end{equation*}
	where \(\text{SpConv}(\cdot)\) indicates the spherical convex hull of the specified points. Moreover, by symmetry, the \(d\)-tuples \((\m V_0,...,\m V_{d-1})\) and \((-\m V_0,\m V_1,...,\m V_{d-1})\) have the same law. Using these two considerations we find
	\begin{align}\label{eq_Cd_sum_of_spherical_indicators}
		& \E\left[ \Delta_{d-1}\left(0,\m V_{0},...,\m V_{d-2}\right) \1_{0\in\Conv(\m V_0,...,\m V_{d-1})} \right] \nonumber
		\\
		& =
		\E\left[ \Delta_{d-1}\left(0,\m V_{0},...,\m V_{d-2}\right) \1_{\m V_{d-1}\in\text{SpConv}(-\m V_0,...,-\m V_{d-2})} \right] \nonumber
		\\
		& = \frac{1}{2} \E\left[ \Delta_{d-1}\left(0,\m V_{0},...,\m V_{d-2}\right) \1_{\m V_{d-1}\in\text{SpConv}(-\m V_0,...,-\m V_{d-2})}\right] \nonumber
		\\
		& \hspace{3.8mm} + \frac{1}{2} \E\left[ \Delta_{d-1}\left(0,-\m V_{0},\m V_1,...,\m V_{d-2}\right) \1_{\m V_{d-1}\in\text{SpConv}(+\m V_0,-\m V_1,...,-\m V_{d-2})} \right] \nonumber
		\\
		& = \frac{1}{2} \E\left[ \Delta_{d-1}\left(0,\m V_{0},...,\m V_{d-2}\right) \left( \1_{\m V_{d-1}\in\text{SpConv}(-\m V_0,...,-\m V_{d-2})} + \1_{\m V_{d-1}\in\text{SpConv}(+\m V_0,-\m V_1,...,-\m V_{d-2})} \right)  \right]
	\end{align}
	where the last equality holds because the volume of a simplex pinned at \(0\) is by definition invariant by symmetry with respect to the origin of any of its vertices. Hence, in particular,
	\begin{equation*}
		\Delta_{d-1}(0,-\m V_0,\m V_1,...,\m V_{d-2}) = \frac{1}{(d-1)!} \left|\text{det}\left(\pm \m V_0,\m V_1,...,\m V_{d-2}\right)\right| = \Delta_{d-1}(0,\m V_0,...,\m V_{d-2}).
	\end{equation*}
	Concerning the sum of indicator functions in the last line of (\ref{eq_Cd_sum_of_spherical_indicators}), as in Figure \ref{fig_SpC_u_SpC_3d}, we remark that the sets \(\text{SpConv}(-\m V_0,...,-\m V_{d-2})\) and \(\text{SpConv}(+\m V_0,-\m V_1,...,-\m V_{d-2})\) are disjoint and that their union is the pre-image under \(P_{\m V_0^\perp}\) of \(\{x\in \m V_0^\perp\cap\Ball(0,1) : 0 \in \Conv(x,P_{\m V_0^\perp}(\m V_1),...,P_{\m V_0^\perp}(\m V_{d-2}))\}\), see also Figure \ref{fig_SpC_u_SpC_proj}. Therefore,
	\begin{align}\label{eq_Cd_resolution_of_sum_of_indicators}
		&
		\hspace{-0.5cm} \1_{\m V_{d-1}\in\text{SpConv}(-\m V_0,...,-\m V_{d-2})} + \1_{\m V_{d-1}\in\text{SpConv}(+\m V_0,-\m V_1,...,-\m V_{d-2})} \nonumber
		\\
		& \hspace{0.5cm} = \1_{\m V_{d-1}\in\text{SpConv}(-\m V_0,...,-\m V_{d-2})\cup\text{SpConv}(+\m V_0,-\m V_1,...,-\m V_{d-2})} \nonumber
		\\
		& \hspace{0.5cm} = \1_{P_{\m V_0^\perp}(\m V_{d-1}) \in \{x\in \m V_0^\perp\cap\Ball(0,1) : 0 \in \Conv(x,P_{\m V_0^\perp}(\m V_1),...,P_{\m V_0^\perp}(\m V_{d-2}))\}} \nonumber
		\\
		& \hspace{0.5cm} = \1_{0 \in \Conv(P_{\m V_0^\perp}(\m V_1),...,P_{\m V_0^\perp}(\m V_{d-1})}.
	\end{align}
	\begin{figure}[h!]
		\centering
		\includegraphics[scale=1]{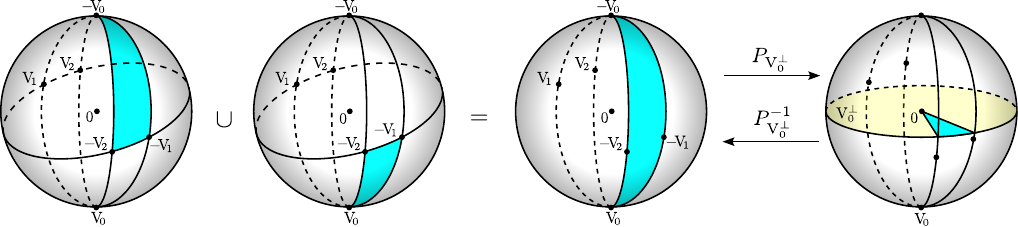}
		\caption{Illustration of the disjoint sets \(\text{SpConv}(-\m V_0,\cdots,-\m V_{d-2})\) (most left) and \(\text{SpConv}(+\m V_0,-\m V_1,...,-\m V_{d-2})\) (second from the left), their union (second from the right), which is the preimage under \(P_{\Vrm_0^\perp}\)  of its projection onto \(\Vrm_{0}^\perp\) (most right).}
		\label{fig_SpC_u_SpC_3d}
	\end{figure}
	\begin{figure}[h!]
		\centering
		\includegraphics[scale=0.9]{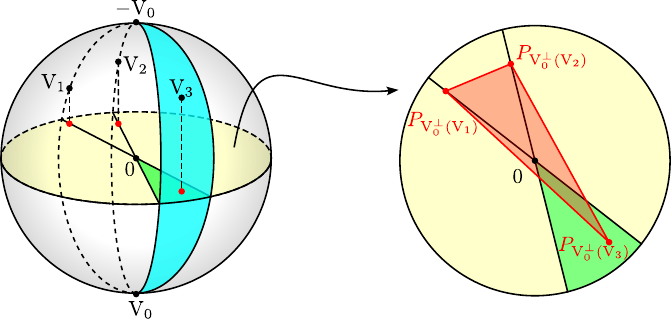}
		\caption{In cyan, the set \(\text{SpConv}(-\Vrm_0,\cdots,-\Vrm_{d-2})\cup\text{SpConv}(+\Vrm_0,-\Vrm_1,...,-\Vrm_{d-2})\) and in green, its orthogonal projection onto \(\Vrm_{0}^\perp\), which is equal to the set \(\{x\in V_0^\perp\cap\Ball(0,1) : 0 \in \Conv(x,P_{V_0^\perp}(V_1),...,P_{V_0^\perp}(V_{d-2}))\}\).}
		\label{fig_SpC_u_SpC_proj}
	\end{figure}
	
	\noindent Inserting (\ref{eq_Cd_resolution_of_sum_of_indicators}) into (\ref{eq_Cd_sum_of_spherical_indicators}), we have the first of the following two equalities:
	\begin{align}\label{eq_Cd_volume_subsimplex_under_pointy}
		 \E\left[ \Delta_{d-1}\left(0,\m V_{0},...,\m V_{d-2}\right) \1_{0\in\Conv(\m V_0,...,\m V_{d-1})} \right] 
		& = \frac{1}{2} \E\left[ \Delta_{d-1}\left(0,\m V_{0},...,\m V_{d-2}\right) \1_{0 \in \Conv(P_{\m V_0^\perp}(\m V_1),...,P_{\m V_0^\perp}(\m V_{d-1})} \right] \nonumber
		\\
		& = \frac{C_{d-1}}{2(d-1)}.
	\end{align}
	In the rest of Step 3, we want to justify (\ref{eq_Cd_volume_subsimplex_under_pointy}), i.e. we aim to prove that 
	\begin{align}\label{eq_Cd_de_construction_simplex_under_pointy}
	     & \E\left[ \Delta_{d-1}\left(0,\m V_{0},...,\m V_{d-2}\right) \1_{0 \in \Conv(P_{\m V_0^\perp}(\m V_1),...,P_{\m V_0^\perp}(\m V_{d-1})} \right] \nonumber\\
	     & = \frac{C_{d-1}}{d-1} 
	     = \frac{\E\left[ \Delta_{d-1}\left(\m V_{0},...,\m V_{d-1}\right) \1_{0 \in \Conv(P_{\m V_0^\perp}(\m V_1),...,P_{\m V_0^\perp}(\m V_{d-1})} \right]}{d-1}.
	\end{align}
	Under the pointy condition \(0 \in \Conv(P_{\m V_0^\perp}(\m V_1),...,P_{\m V_0^\perp}(\m V_{d-1})\), the volume of the simplex \(\Conv\left(\m V_{0},...,\m V_{d-1}\right)\) on the right-hand side of \eqref{eq_Cd_de_construction_simplex_under_pointy} can be expressed as a signed sum of volumes of simplices, as in Figure \ref{fig_signed_sum}.
	\begin{figure}[h!]
		\centering
		\includegraphics[scale=0.9]{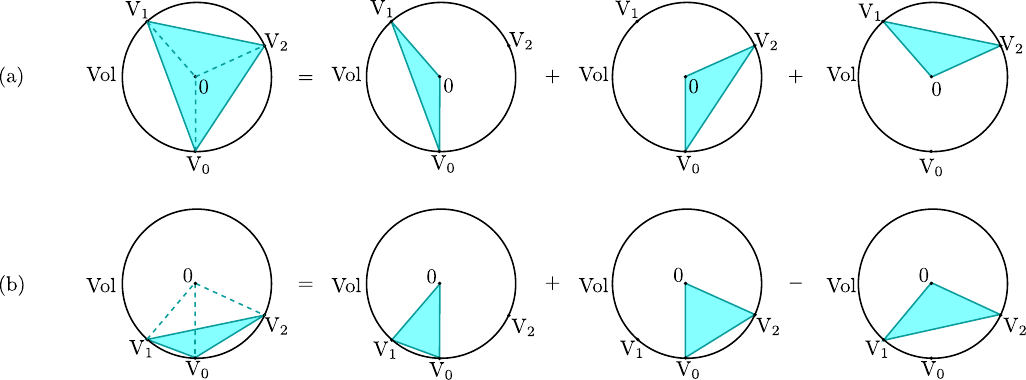}
		\caption{Representation of equation (\ref{eq_signed_sum}) relative to \(\Vrm_0,...,\Vrm_{d-1}\in\Sp_{\R^{d-1}}\) for \(d-1=2\). (a) represents the case \(\sign(\m V_0, \m V_1,\m V_2)=+1\) and (b) the case \(\sign(\m V_0, \m V_1,\m V_2)=-1\).}
		\label{fig_signed_sum}
	\end{figure}
	
	\noindent Explicitly, we have
	\begin{align}\label{eq_signed_sum}
		& \Delta_{d-1}\left(\m V_0,...,\m V_{d-1}\right) 1_{0 \in \Conv(P_{\m V_0^\perp}(\m V_1),...,P_{\m V_0^\perp}(\m V_{d-1}))} \nonumber
		\\
		& = \sum\limits_{i=1}^{d-1} \Delta_{d-1}\left(0,\m V_0,...,\widehat{\m V_i},...,\m V_{d-1}\right) \1_{0 \in \Conv(P_{\m V_0^\perp}(\m V_1),...,P_{\m V_0^\perp}(\m V_{d-1}))} \nonumber
		\\
		& \hspace{3.8mm} + \text{sign}(\m V_0,...,\m V_{d-1}) \Delta_{d-1}\left(0,\m V_1,...,\m V_{d-1}) \right) 1_{0 \in \Conv(P_{\m V_0^\perp}(\m V_1),...,P_{\m V_0^\perp}(\m V_{d-1}))}
	\end{align}
	where
	\begin{align*}
		\text{sign}(\m V_0, ... ,\m V_{d-1}) =
		\begin{cases}
			+1 & \text{ if } 0 \in \Conv\left(\m V_0,...,\m V_{d-1}\right)\text{,}
			\\
			-1 & \text{ else.}
		\end{cases}
	\end{align*}
	As \(\m V_0,...,\m V_{d-1}\) are i.i.d. uniform random variables, the summands in the first line of (\ref{eq_signed_sum}) are identically distributed and the signed volume of the second line of (\ref{eq_signed_sum}) has zero expectation, i.e., 
	\begin{align*}
		& \E\left[\Delta_{d-1}\left(\m V_0,...,\m V_{d-1}\right) \1_{0 \in \Conv(P_{\m V_0^\perp}(\m V_1),...,P_{\m V_0^\perp}(\m V_{d-1}))}\right] \nonumber
		\\
		& = \sum_{i=1}^{d-1} \E\left[\Delta_{d-1}\left(0,\m V_0,...,\widehat{\m V_i},...,\m V_{d-1}\right) \1_{0 \in \Conv(P_{\m V_0^\perp}(\m V_1),...,P_{\m V_0^\perp}(\m V_{d-1}))}\right] \nonumber
		\\
		& \hspace{3.8mm} + \E\left[\text{sign}(\m V_0,...,\m V_{d-1}) \Delta_{d-1}(0,\m V_1,...,\m V_{d-1}) \1_{0 \in \Conv(P_{\m V_0^\perp}(\m V_1),...,P_{\m V_0^\perp}(\m V_{d-1}))}\right] \nonumber
		\\
		& = (d-1) \E\left[\Delta_{d-1}\left(0,\m V_0,...,\m V_{d-2}\right) \1_{0 \in \Conv(P_{\m V_0^\perp}(\m V_1),...,P_{\m V_0^\perp}(\m V_{d-1}))}\right].
	\end{align*}
	This justifies \eqref{eq_Cd_de_construction_simplex_under_pointy} and subsequently \eqref{eq_Cd_volume_subsimplex_under_pointy}.
    \\~\\
	\noindent\textbf{Step 4. Conclusions.}
	\\~\\
	We merge equations (\ref{eq_Cd_extraction_factors_Z}) and (\ref{eq_Cd_volume_subsimplex_under_pointy}) to find 
	\begin{align}\label{eq_recursive_relation_with_Z_expectations}
	    C_d
		& =\E\left[1+\m Z_{d-1}\right] \left(\prod\limits_{j\in\lbrace0,...,d-2\rbrace} \E\left[\sqrt{1-\m Z_j^2}\right]\right) 
		\frac{C_{d-1}}{2(d-1)}
	\end{align}
	The first statement of Proposition \ref{prop_recursive_relation_on_vol_full_simplex}, i.e. the identity \eqref{eq_prop_recursive_formula_on_C_d} or equivalently (\ref{eq_prop_avg_volume_simplex_recursive_formula}), is directly obtained from (\ref{eq_recursive_relation_with_Z_expectations}) by evaluating the expectations relative to the variables \(\m Z_0,...,\m Z_{d-1}\), which are i.i.d. with density \(\m z\mapsto\frac{\1_{\m z\in[-1,1]}}{B\left(\frac{d-1}{2},\frac{1}{2}\right)}(1-\m z^2)^{\frac{d-3}{2}}\). Explicitly,
	\begin{align*}
		\E\left[\sqrt{1-\m Z_j^2}\right] & = \frac{B\left(\frac{d}{2},\frac{1}{2}\right)}{B\left(\frac{d-1}{2},\frac{1}{2}\right)} \text{ for } j\in\lbrace1,...,d-2\rbrace,
		\\
		\E\left[1+\m Z_{d-1}\right] & = 1.
	\end{align*}
	
	To obtain the second statement of Proposition \ref{prop_recursive_relation_on_vol_full_simplex}, we start by applying \((d-1)\) times the recursive relation (\ref{eq_prop_avg_volume_simplex_recursive_formula}) and simplifying the Beta coefficients by telescoping product, i.e.
	\begin{align}\label{eq_Cd_unraveled_recursion}
		& \E\left[\Delta_{d}(\m U_0,...,\m U_{d})\1_{0\in\Conv(P_{\m U_0^\perp}(\m U_1),...,P_{\m U_0^\perp}(\m U_d))}\right] \nonumber
		\\
		& = \prod\limits_{k=2}^{d-1} \frac{\left(B\left(\frac{d}{2},\frac{1}{2}\right)\right)^k}{2k\left(B\left(\frac{d-1}{2},\frac{1}{2}\right)\right)^k} \E\left[\Delta_2(\m W_0,\m W_1,\m W_2)\1_{0\in\Conv(P_{\m W_0^\perp}(\m W_1),P_{\m W_0^\perp}(\m W_2))}\right]\nonumber 
		\\
		& = \frac{\sqrt{\pi}\Gamma\left(\frac{d}{2}\right)^d}{2^{d-1}(d-1)!\Gamma\left(\frac{d+1}{2}\right)^{d-1}}
		\E\left[\Delta_2(\m W_0,\m W_1,\m W_2)\1_{0\in\Conv(P_{\m W_0^\perp}(\m W_1),P_{\m W_0^\perp}(\m W_2))}\right] 
	\end{align}
	where \(\m W_0,\m W_1,\m W_2\) are i.i.d. uniform random variables on the \(1\)-dimensional sphere \(\Sp_{\R^2}\). To conclude, we only need to obtain the expectation of the last line of (\ref{eq_Cd_unraveled_recursion}). For this, we apply (\ref{eq_Cd_de_construction_simplex_under_pointy}),
	\begin{equation}\label{eq_Cd_reduction_to_subsimplex_base_case_computation}
		\E\left[\Delta_2(\m W_0,\m W_1,\m W_2)\1_{0\in\Conv(P_{\m W_0^\perp}(\m W_1),P_{\m W_0^\perp}(\m W_2))}\right] = 2 \E\left[\Delta_2(0,\m W_0,\m W_1)\1_{0\in\Conv(P_{\m W_0^\perp}(\m W_1),P_{\m W_0^\perp}(\m W_2))}\right]
	\end{equation}
	and then compute the expectation in the right-hand side of \eqref{eq_Cd_reduction_to_subsimplex_base_case_computation}. Using the rotation invariance, we fix \(\m W_0 = (0,-1)^T\). We then represent the variables \(\m W_1,\m W_2\) with the standard polar coordinate system \(\m W_1=(\cos(\theta_1),\sin(\theta_1))^T,\m W_2=(\cos(\theta_2),\sin(\theta_2))^T\). In these coordinates, we have \(\Delta_2(0,\m W_0,\m W_1)=\frac{\cos(\theta_1)}{2}\), as in Figure \ref{fig_polar_coordinates}.
	\begin{figure}[h!]
		\centering
		\includegraphics[scale=1]{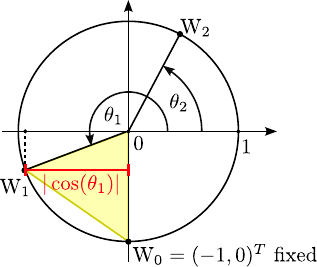}
		\caption{Standard polar change of variables for the points \(\m W_1,\m W_2\).} 
		\label{fig_polar_coordinates}
	\end{figure}
	
	\noindent Moreover, the condition \(0\in\Conv(P_{\m W_0^\perp}(\m W_1,\m W_2))\) is satisfied whenever \(\m W_1\) belongs to the right half-plane and \(\m W_2\) to the left half-plane or vice versa, and when we express it into a condition on the angles \(\theta_1,\theta_2\) we obtain
	\begin{align*}
		& \E \left[ \Delta_2(0,\m W_0,\m W_1) \1_{0\in\Conv(P_{\m W_0^\perp}(\m W_1),P_{\m W_0^\perp}(\m W_2))} \right] 
		\\
		& = \int\limits_{[0,2\pi]^2} \frac{d\theta_1}{2\pi} \frac{d\theta_2}{2\pi} \, \frac{|\cos(\theta_1)|}{2} \1_{(\theta_1,\theta_2)\in\left[-\frac{\pi}{2},\frac{\pi}{2}\right]\times\left[\frac{\pi}{2},\frac{3\pi}{2}\right]\cup\left[\frac{\pi}{2},\frac{3\pi}{2}\right]\times\left[-\frac{\pi}{2},\frac{\pi}{2}\right]} = \frac{1}{2\pi}.
	\end{align*} 
	By inserting this explicit value into (\ref{eq_Cd_reduction_to_subsimplex_base_case_computation}) and (\ref{eq_Cd_unraveled_recursion}), we obtain
	\begin{equation}\label{eq_explicit_value_E_Delta_d}
		\E\left[\Delta_{d}(\m U_0,...,\m U_d) \1_{0\in\Conv(P_{\m U_0^\perp}(\m U_1),...,P_{\m U_0^\perp}(\m U_d))}\right] = 
		\frac{1}{2^{d-1}\sqrt{\pi}(d-1)!}\frac{\Gamma\left(\frac{d}{2}\right)^{d}}{\Gamma\left(\frac{d+1}{2}\right)^{d-1}}.
	\end{equation}
	Thanks to (\ref{eq_C_d_as_expectation}), we simply need to multiply the right-hand side of  (\ref{eq_explicit_value_E_Delta_d}) by \((d\kappa_d)^d\), where \(\kappa_d\) is given at \eqref{eq:defkappad}, to obtain the required value of the constant \(C_d\) given at (\ref{eq_prop_avg_volume_simplex_explicit_value}).
	\end{proof}

\section{Extremal index: proof of Theorem \ref{thm_extremal_index}}\label{sec_extrremal_index}
	In this section, we prove Theorem \ref{thm_extremal_index}, i.e. we connect Theorem \ref{thm_extremal_index} to a former result regarding the quantity
	$\displaystyle \max_{x\in \Phi\cap \rho^{\frac1{d}}[0,1]^d}{\mathcal D}(x)$
	where $\rho>0$ and
	$${\mathcal D}(x):=\min\{r\ge 0:{\mathcal C}(x,\Phi)\subset {\mathcal B}_r(x)\}.$$
In \cite[Theorem 1 (2c)]{CalkaChen14}, the following convergence when $\rho\to\infty$ is derived: for any $t\in \R$,
	\begin{equation}\label{eq:conv_extreme}
	\lim_{\rho\to\infty}\PR\left(\kappa_d \left(\max_{x\in \Phi\cap \rho^{\frac1{d}}[0,1]^d}{\mathcal D}(x)\right)^d\le \log(\alpha_1\rho(\log\rho)^{d-1})+t\right)=e^{-e^{-t}}
	\end{equation}
	where 
	\begin{equation}\label{eq:value_alpha1}
	\alpha_1=\frac1{d!}\left(\frac{\pi^{\frac12}\Gamma(\frac{d}2+1)}{\Gamma(\frac{d+1}2)}\right)^{d-1}.
	\end{equation}
	Besides, Theorem \ref{thm_main_result} implies that for any $t\in \R$,
	\begin{equation}\label{eq:condition1.6CR}
	\lim_{\rho\to\infty}\rho\PR\left(\kappa_d {\mathcal D}^d\ge \log(\alpha_1'\rho(\log\rho)^{d-1})+t\right)=e^{-t}
	\end{equation}
	where 
	\begin{equation}\label{eq:value_alpha1prime}
	\alpha_1':=d^d\kappa_dC_d=d^d\frac{\pi^{\frac{d}{2}}}{\Gamma(\frac{d}{2}+1)}\frac{1}{2^{d-1}\sqrt{\pi}(d-1)!} \frac{\Gamma\left(\frac{d}{2}\right)^d}{\Gamma\left(\frac{d+1}{2}\right)^{d-1}}
	\end{equation}
	the second equality being due to \eqref{eq_prop_avg_volume_simplex_explicit_value}. In view of \eqref{eq:condition1.6CR}, we introduce the extremal index $\theta\in (0,1]$ with respect to the maximum of the sequence $\{{\mathcal D}(x), x\in \Phi\}$, in the spirit of \cite{ChenavierRobert}, i.e. $\theta$ is the positive number such that
	$$\lim_{\rho\to\infty}\PR\left(\kappa_d\left(\max_{x\in \Phi\cap \rho^{\frac1{d}}[0,1]^d}{\mathcal D}(x)\right)^d\le 
	\log(\alpha_1'\rho(\log\rho)^{d-1})+t\right)=e^{-\theta e^{-t}}.$$
	Applying \eqref{eq:conv_extreme} to $t=t'+\log(\alpha_1'/\alpha_1)$ for some $t'\in\R$, we observe that
	$$\lim_{\rho\to\infty}\PR\left(\kappa_d\left(\max_{x\in \Phi\cap \rho^{\frac1{d}}[0,1]^d}{\mathcal D}(x)\right)^d\le 
	\log(\alpha_1'\rho(\log\rho)^{d-1})+t'\right)=e^{-\frac{\alpha_1}{\alpha_1'}e^{-t'}}.$$
	In conclusion, we obtain, thanks to \eqref{eq:value_alpha1} and \eqref{eq:value_alpha1prime} that 
	$$\theta=\frac{\alpha_1}{\alpha_1'}=\frac1{2d}.$$
	
	\section{Extension to a parametric model}\label{sec_alpha_case}
	
	Let \(\alpha>-d\) and \(\Phi_\alpha\) be a Poisson point process on \(\R^d\) of isotropic intensity measure \(m_\alpha\),
	\begin{equation*}
		m_\alpha(d\m x) = \|\m x\|^{\alpha} d\m x
	\end{equation*}
	where we recall that \(d\m x\) indicates the Lebesgue measure on \(\R^d\). Let \(\mathcal{C}_{0}(\Phi_\alpha)\) be the Voronoi cell associated to \(0\) arising from the point process \(\Phi_\alpha\cup\{0\}\) and \(\mathcal{D}_\alpha\) the maximal distance from its boundary to \(0\),
	\begin{equation*}
		\mathcal{D}_\alpha = \min\left\{\m r\geq0 : \mathcal{C}_{0}(\Phi_\alpha) \subseteq \Ball_{\m r}(0)\right\}.
	\end{equation*}
	
	Two cases are of particular importance. When \(\alpha=0\), \(\mathcal{D}_\alpha=\mathcal{D}\) is the maximal vertex-nucleus distance of the typical cell of the homogeneous Poisson-Voronoi tessellation studied so far. When \(\alpha=-(d-1)\),  the cell \(\mathcal{C}_{0}(\Phi_\alpha)\) is up to a scaling factor $2$ the Crofton cell (or zero cell) of the isotropic and stationary Poisson hyperplane tessellation with unit intensity, see \cite[Definition 6.1.1]{HugSchneider}. When \(d=2\), a double bound of the tail probability of \(\mathcal{D}_0=\Rcirc\) and \(\mathcal{D}_{d-1}\) has been obtained by Calka \cite{Calka2002}, and this work motivates our generalization to the study of \(\mathcal{D}_\alpha\) for all \(\alpha\) such that the cell \(\mathcal{C}_{0}(\Phi_\alpha)\) is not degenerate, that is for \(\alpha>-d\). The main result of Section \ref{sec_alpha_case} is Theorem \ref{thm_tail_alpha_case} which extends Theorem \ref{thm_main_result} to any \(\alpha>-d\).
	
	\begin{theorem}\label{thm_tail_alpha_case}
		When \(t\rightarrow\infty\),
		\begin{align}\label{eq:equiv gen alpha}
			\PR\left(\mathcal{D}_\alpha \geq t\right) = & \frac{2^{d^2+d(\alpha-2)-(\alpha-2)}\pi^{\frac{d(d-1)}{2}}}{(d-1)!} \left(\frac{\Gamma\left(\frac{d+\alpha}{2}\right)}{\Gamma\left(d+\frac{\alpha}{2}\right)}\right)^{d-1} t^{d^2+d(\alpha-1)-\alpha} e^{-K_{d,\alpha}t^{d+\alpha}} \nonumber
			\\
			& \hspace{8cm} + \mathcal{O}\left(t^{d^2+d(\alpha-2)-2\alpha} e^{-K_{d,\alpha}t^{d+\alpha}}\right)
		\end{align}
		where
		\begin{equation*}
			K_{d,\alpha} = \frac{\sigma_{d-1} 2^{d+\alpha-1}}{d+\alpha} B\left(\frac{d-1}{2},\frac{d+\alpha+1}{2}\right).
		\end{equation*}
	\end{theorem}
	The rest of this section sketches a proof of Theorem \ref{thm_tail_alpha_case}, whose strategy is analogous to what was done in previous sections for the case \(\alpha=0\). We start by defining the set of pointy vertices of \(\mathcal{C}_0(\Phi_\alpha)\):
	\begin{equation*}
		\mathcal{V}^{\geq t}_{\text{max}}(\Phi_\alpha) = \{\text{all pointy vertices of } \mathcal{C}_{0}(\Phi_\alpha) \text{ at distance} \geq t \text{ from } 0\}.
	\end{equation*}
	For sake of consistency, we recall that by $\mathcal{V}^{\ge t}_{\text{max}}$ without argument, we mean the variable $\mathcal{V}^{\ge t}_{\text{max}}(\Phi_0)$ which has been considered in the previous sections. Then, the exact same reasoning as in 
	Section \ref{sec_proof_of_theorem} holds in the \(\alpha\)-case as well and justifies the double bound
	\begin{equation*}
		\E\left[\Card\left(\VOMR(\Phi_\alpha)\right)\right] - \E\left[\Card\left((\VOMR(\Phi_\alpha))^2_{\neq}\right)\right] \leq \PR\left(\mathcal{D}_\alpha\geq t\right) \leq \E\left[\Card\left(\VOMR(\Phi_\alpha)\right)\right].
	\end{equation*}
	Therefore, the probability tail \(\PR\left(\Rcirc_\alpha\geq t\right)\) shares the same asymptotics as \(\E\left[\Card\left(\VOMR(\Phi_\alpha)\right)\right]\) up to an error of growth order \(\E\left[\Card\left((\VOMR(\Phi_\alpha))^2_{\neq}\right)\right]\). We give a detailed description of the computation of \(\E\left[\Card\left(\VOMR(\Phi_\alpha)\right)\right]\) in Sections \ref{subsec_exp_number_pointy_alpha} and \ref{subsec_C_d_alpha}. In Section \ref{subsec_expectation_couples_alpha_case}, we discuss the strategy for proving the negligibility of \(\E\left[\Card\left((\VOMR(\Phi_\alpha))^2_{\neq}\right)\right]\) with respect to \(\E\left[\Card\left(\VOMR(\Phi_\alpha)\right)\right]\), without going into details.
	
	\subsection{An integral expression for the expected number of pointy vertices}\label{subsec_exp_number_pointy_alpha}
	Following the same strategy as in the proof of Proposition \ref{thm_E_nr_locmaxvert}, we obtain an explicit integral expression for \(\E\left[\Card\left(\mathcal{V}^{\geq t}_{\blacktriangle}(\Phi_\alpha)\right)\right]\) by consecutive applications of the Mecke formula and the spherical Blaschke-Petkantschin change of variables combined with \eqref{eq_coord_chg_BP}, i.e.
	\begin{align}\label{eq_expectation_nr_pointy_alpha}
		& \E\left[\Card\left(\mathcal{V}^{\geq t}_{\blacktriangle}(\Phi_\alpha)\right)\right] \nonumber
		\\
		& = \E\left[  \frac{1}{d!} \sum\limits_{(\m X_1,...,\m X_d)\in(\Phi_\alpha)^2_{\neq}} \1_{\Center(\text{Ball}(0,\m X_1,...,\m X_d))\in\mathcal{V}^{\geq t}_{\blacktriangle}(\Phi_\alpha)} \right] \nonumber
		\\
		& = \frac{1}{d!} \int\limits_{\R^d} m_\alpha(d\m x_1)... \int\limits_{\R^d}m_\alpha(d\m x_d) \E\left[ \1_{\underbrace{ \text{\scalebox{.7}{Center}}(\text{\scalebox{.7}{Ball}}(0,\m x_1,...,\m x_d))}_{c}\in\mathcal{V}^{\geq t}_{\blacktriangle}(\Phi_\alpha\cup(0,\m x_1,...,\m x_d))} \right] \nonumber
		\\
		& = \frac{1}{d!} \int\limits_{\R^d} d\m x_1... \int\limits_{\R^d}d\m x_d\,\|\m x_1\|^{\alpha}\cdots\|\m x_d\|^{\alpha} \E\left[ \1_{c\in \mathcal{V}^{\geq t}_{\blacktriangle}(\Phi_\alpha\cup(0,\m x_1,...,\m x_d))} \right] \nonumber
		\\
		& = \int\limits_{\R_+} d\m r \int\limits_{(\Sp_{\R^d})^{d+1}} d\m u_0\cdots d\m u_d \, \|\m r(\m u_1-\m u_0)\|^{\alpha} \cdots \|\m r(\m u_d-\m u_0)\|^{\alpha} \m r^{d^2-1} \Delta_{d}(\m u_0,...,\m u_d) 
		\E\left[ \1_{c\in\mathcal{V}^{\geq t}_{\blacktriangle}(\Phi_\alpha\cup(0,\m x_1,...,\m x_d))} \right].
	\end{align}
	In the same way as at \eqref{eq_set_decomposition_of_VOMR_after_MS}, we obtain
	\begin{equation}\label{eq_three_conditions_alpha}
		\E\left[ \1_{c\in \mathcal{V}^{\geq t}_{\blacktriangle}(\Phi_\alpha\cup(0,\m x_1,...,\m x_d))} \right] = \1_{\m r\geq t} \1_{0\in\Conv(P_{\m u_0^\perp}(\m u_1),...,P_{\m u_0^\perp}(\m u_d))} e^{-m_\alpha(\Ball_{\m r}(-\m r\m u_0))}
	\end{equation}
	where \(\Ball_{\m r}(-\m r\m u_0)\) indicates the ball adjacent to \(0\) of radius \(\m r\) and center \(-\m r\m u_0\). Its \(m_\alpha\)-measure is
	\begin{equation*}
		m_\alpha(\Ball_r(-\m r\m u_0)) = K_{d,\alpha} \m r^{d+\alpha},
	\end{equation*}
	where \(K_{d,\alpha}\) is the \(m_\alpha\)-measure of a unit ball adjacent to the origin, i.e.
	\begin{equation}\label{eq_m_alpha_volume_unit_ball}
		K_{d,\alpha} = m_\alpha(\Ball_1(-\m u_0)) = \frac{(d-1)\kappa_{d-1} 2^{d+\alpha-1}}{d+\alpha} B\left(\frac{d-1}{2},\frac{d+\alpha+1}{2}\right).
	\end{equation}
	Inserting (\ref{eq_three_conditions_alpha}) into (\ref{eq_expectation_nr_pointy_alpha}) and using Fubini's theorem, the integral splits into two factors:
	\begin{align}\label{eq_expectation_pointy_after_Fubini_alpha}
			& \E\left[\Card\left(\mathcal{V}^{\geq t}_{\blacktriangle}(\Phi_\alpha)\right)\right] \nonumber
			\\
			&= \int\limits_{t}^\infty d\m r \int\limits_{(\Sp_{\R^d})^{d+1}} \prod_{i=1}^dd\m u_i \, \|\m r(\m u_i-\m u_0)\|^{\alpha}{\m r}^{d^2-1} \Delta_{d}(\m u_0,...,\m u_d) \1_{0\in\Conv(P_{\m u_0^\perp}(\m u_1),...,P_{\m u_0^\perp}(\m u_d))} e^{-K_{d,\alpha} \m r^{d+\alpha}} \nonumber
		\\
		& = \int\limits_{t}^\infty d\m r \, \m r^{d^2+d\alpha-1} e^{-K_{d,\alpha} \m r^{d+\alpha}}\hspace*{-.5cm}\int\limits_{(\Sp_{\R^d})^{d+1}} \prod_{i=1}^dd\m u_i \|\m u_i-\m u_0\|^{\alpha}\Delta_{d}(\m u_0,...,\m u_d) \1_{0\in\Conv(P_{\m u_0^\perp}(\m u_1),...,P_{\m u_0^\perp}(\m u_d))}.
	\end{align}
	The integral with respect to \(\m u_0,...,\m u_d\) in the right-hand side of (\ref{eq_expectation_pointy_after_Fubini_alpha}) is, up to renormalization, the expectation
	\begin{equation}\label{eq_C_d_alpha_implicit}
		C_{d,\alpha} := \E\left[ \|\m U_1-\m U_0\|^{\alpha} \cdots \|\m U_d-\m U_0\|^{\alpha} \Delta_{d}(\m U_0,...,\m U_d) \1_{0\in\Conv(P_{\m U_0^\perp}(\m U_1),...,P_{\m U_0^\perp}(\m U_d))} \right]
	\end{equation}
	where \(\m U_0,...,\m U_d\) are i.i.d. uniform random variables on the \((d-1)\)-dimensional sphere \(\Sp_{\R^d}\). Moreover, the integral with respect to \(r\) in the right-hand side of (\ref{eq_expectation_pointy_after_Fubini_alpha}) can be computed by multiple integration by parts. Consequently, we obtain
	\begin{equation}\label{eq:equiv interm alpha}
		\E\left[\Card\left(\mathcal{V}^{\geq t}_{\blacktriangle}(\Phi_\alpha)\right)\right] \underset{t\rightarrow\infty}{\sim} \frac{(d\kappa_d)^{d+1}C_{d,\alpha}}{K_{d,\alpha}(d+\alpha)} t^{d^2+d(\alpha-1)-\alpha} e^{-K_{d,\alpha}t^{d+\alpha}}
	\end{equation}
	where we recall that \(\kappa_d\) is the volume of the $d$-dimensional unit ball given in \eqref{eq:defkappad}.
	
	\subsection{Computation of the constant \(C_{d,\alpha}\)}\label{subsec_C_d_alpha}
	
	We now proceed to compute the constant \(C_{d,\alpha}\), that is, the expectation in (\ref{eq_C_d_alpha_implicit}). Similarly as in the proof of Proposition \ref{prop_recursive_relation_on_vol_full_simplex}, we apply to \(\m U_0,...,\m U_d\) the cylindrical change of coordinates of equation \eqref{eq_cylindrical_coordinates}, see Figure \ref{fig_cylindrical_coordinates}. Using \eqref{eq_Cd_first_steps_on_volume_in_cylindrical_coord} and \eqref{eq_Cd_extraction_factors_Z} for the treatment of the factor $\Delta_{d}(\m U_0,...,\m U_d) \1_{0\in\Conv(P_{\m U_0^\perp}(\m U_1),...,P_{\m U_0^\perp}(\m U_d))}$ as well as the equality
	due to Pythagora's theorem $\|\m U_i-\m U_0\|^2=2(1+\m Z_{i-1})$ for $1\le i\le d$, we obtain 
	\begin{align}\label{eq_C_d_alpha_splitted_integrals}
		C_{d,\alpha}
		& = \E\left[ 2^{\frac{d\alpha}{2}}(1+\m Z_0)^{\frac{\alpha}{2}+1}\cdots(1+\m Z_{d-1})^{\frac{\alpha}{2}+1} \prod\limits_{j=0}^{d-2}\sqrt{\frac{1-\m Z_j}{1+\m Z_j}}  \right]
		\E\left[ \Delta_{d-1}(0,\m V_0,...,\m V_{d-2})\1_{0\in\Conv(\m V_0,...,\m V_{d-1})} \right]
	\end{align}
	where \(\m V_0,...,\m V_{d-1}\) are i.i.d. random variables uniformly distributed on the \((d-2)\)-dimensional sphere \(\Sp_{\R^{d-1}}\) and \(\m Z_0,...,\m Z_{d-1}\) are i.i.d. random variables with density \(\m z\mapsto\frac{1}{B(\frac{d-1}{2},\frac{1}{2})}(1-\m z^2)^{\frac{d-3}{2}}1_{\m z\in[-1,1]}\). The expectation relative to \(\m Z_0,...,\m Z_{d-1}\) can be made explicit, i.e. 
	\begin{equation}\label{eq_extraction_zeta_expectations_alpha}
		\E\left[ 2^{\frac{d\alpha}{2}}(1+\m Z_0)^{\frac{\alpha}{2}+1}\cdots(1+\m Z_{d-1})^{\frac{\alpha}{2}+1} \prod\limits_{j=0}^{d-2}\sqrt{\frac{1-\m Z_j}{1+\m Z_j}}  \right] = 2^{d^2+d(\alpha-1)} \frac{B\left(\frac{d+\alpha}{2},\frac{d}{2}\right)^{d-1}B\left(\frac{d+\alpha+1}{2},\frac{d-1}{2}\right)}{B\left(\frac{d-1}{2},\frac{1}{2}\right)^{d}}.
	\end{equation}
	Moreover, using (\ref{eq_Cd_volume_subsimplex_under_pointy}), we obtain that the expectation with respect to the variables \(\m V_0,...,\m V_{d-1}\) 
	satisfies
	\begin{equation}\label{eq_random_volume_alpha_case}
		\E\left[ \Delta_{d-1}(0,\m V_0,...,\m V_{d-2})\1_{0\in\Conv(\m V_0,...,\m V_{d-1})} \right] = \frac{C_{d-1}}{2(d-1)},
	\end{equation}
	where we recall the value of $C_{d-1}$ at \eqref{eq:value of Cd}. 
	
	We now insert (\ref{eq_extraction_zeta_expectations_alpha}) and (\ref{eq_random_volume_alpha_case}) into (\ref{eq_C_d_alpha_splitted_integrals}) and simplify the resulting expression by means of the identity 
	\begin{equation*}
		B\left(\frac{d+\alpha+1}{2},\frac{d-1}{2}\right) = \frac{B\left(\frac{d+\alpha}{2},\frac{d}{2}\right)B\left(\frac{d-1}{2},\frac{1}{2}\right)}{B\left(\frac{d+\alpha}{2},\frac{1}{2}\right)}
	\end{equation*}
	to deduce that 
	\begin{equation*}
		C_{d,\alpha}  = 2^{d^2+d(\alpha-1)} \frac{B\left(\frac{d+\alpha}{2},\frac{d}{2}\right)^{d-1}B\left(\frac{d+\alpha+1}{2},\frac{d-1}{2}\right)}{B\left(\frac{d-1}{2},\frac{1}{2}\right)^{d}} \frac{1}{2(d-1)}C_{d-1}
		= \frac{2^{d^2+d(\alpha-2)+1}}{\pi^{\frac{d}{2}}(d-1)!} \Gamma\left(\frac{d}{2}\right) \frac{B\left( \frac{d+\alpha}{2},\frac{d}{2} \right)^d}{B\left(\frac{d+\alpha}{2},\frac{1}{2}\right)}
	\end{equation*}
	where the second equality is obtained by using (\ref{eq:value of Cd}).
	Simplifying the constant in the right-hand side of \eqref{eq:equiv interm alpha}, we obtain that $\E\left[\Card\left(\mathcal{V}^{\geq t}_{\text{max}}(\Phi_\alpha)\right)\right]$ has same asymptotics as in the right-hand side of \eqref{eq:equiv gen alpha}, i.e. 
	\begin{align*}
		\E\left[\Card\left(\mathcal{V}^{\geq t}_{\text{max}}(\Phi_\alpha)\right)\right] \underset{t\rightarrow\infty}{\sim} &
		\frac{2^{d^2+d(\alpha-2)-(\alpha-2)} \pi^{\frac{d(d-1)}{2}}}{(d-1)!} \left( \frac{\Gamma\left(\frac{d+\alpha}{2}\right)}{\Gamma\left(d+\frac{\alpha}{2}\right)} \right)^{d-1} t^{d^2+d(\alpha-1)-\alpha} e^{-K_{d,\alpha}t^{d+\alpha}}.
	\end{align*}
	
	\subsection{Asymptotic bound for \(\E\left[\Card\left((\VOMR(\Phi_\alpha))^2_{\neq}\right)\right]\)}\label{subsec_expectation_couples_alpha_case}
	We compute the bound on \(\E \left[ \Card \left((\VOMR(\Phi_\alpha))^2_{\neq}\right) \right] \) with the same strategy used in Section \ref{sec_expected_nr_pairs} for the case \(\alpha=0\). The main difference comes from the fact that the measure $\mu_\alpha$ is no longer translation-invariant. Our main difficulty then consists in finding a proper lower bound for the measure of the union of two balls, as discussed in Lemma \ref{lemma_bound_on_volume_alpha}.

    Using the same notation as in Section \ref{sec_expected_nr_pairs}, we observe that the analogue of \eqref{eq_expectation_pairs_after_Nikitenko} for general $\alpha$ is
	\begin{align}\label{eq_after_Nikitenko_alpha}
		& \E \left[ \Card \left((\VOMR(\Phi_\alpha))^2_{\neq}\right) \right] \nonumber
		\\
		& = \sum\limits_{k=0}^{d-1} \frac{1}{k!((d-k)!)^2} \int\limits_{\mathcal{L}_k^d}d\m Q\int\limits_{\Sp_{\m Q}^{k+1}}d\urm_0d\vbf\int\limits_{\R_{+}}d\rho_0 \int\limits_{\Sp_{\m Q^\perp}}d\m u\int\limits_{\R_{+}}d\m r\int\limits_{\Sp_{\R^d}^{d-k}}d\wbf \int\limits_{\Sp_{\m Q^\perp}}d\m u'\int\limits_{\R_{+}}d\m r'\int\limits_{\Sp_{\R^d}^{d-k}}d\wbf 
		\nonumber
		\\
		& \hspace{5mm} \rho_0^{kd-1} \left(k!\Delta_{k}\left(-\urm_0,\vbf\right)\right)^{d-k+1} \1_{0\leq\rho_0\leq \m r,\m r'} 
		\nonumber
		\\
		& \hspace{5mm} \m r^{(d-k)(d-1)+1} \sqrt{\m r^2-\rho_0^2}^{d-k-2}(d-k)!\Delta_{d-k}\left(-\frac{\sqrt{\m r^2-\rho_0^2}}{\m r}\m u,P_{\m Q^\perp}(\wbf)\right) \nonumber
		\\
		& \hspace{5mm} (\m r')^{(d-k)(d-1)+1} \sqrt{(\m r')^2-\rho_0^2}^{d-k-2}(d-k)!\Delta_{d-k}\left(-\frac{\sqrt{(\m r')^2-\rho_0^2}}{\m r'}\m u',P_{\m Q^\perp}(\wbf')\right) 
		\nonumber\\
		& \hspace{5mm} \|\m x_1\|^\alpha \cdots \|\m x_k\|^\alpha \|\m y_1\|^\alpha \cdots \|\m y_{d-k}\|^\alpha \|\m y'_1\|^\alpha \cdots \|\m y'_{d-k}\|^\alpha e^{-m_\alpha(\Ball\cup\Ball')} \1_{(c,c')\in\left(\VOMR(\lbrace0\rbrace\cup\lbrace\x,\y,\y'\rbrace)\right)^2_{\neq}}.
	\end{align}
	We now wish to find an upper bound for the integrand of (\ref{eq_after_Nikitenko_alpha}). We proceed as in Step 3 of Section \ref{sec_expected_nr_pairs} by using the notation \(\m x_i=\rho_0(\m u_0+\m v_i)\), \(\m y_i=\rho_0\m u_0+\sqrt{\m r^2-\rho_0^2}\m u + \m r\m w_i\) and \(\m y'_i=\rho_0\m u_0+\sqrt{(\m r')^2-\rho_0^2}\m u' + \m r'\m w'_i\), see Figure \ref{fig_BP_Nikitenko}. 
	The main novelty consists in bounding from below the \(m_\alpha\)-measure of the union of two balls centered at pointy vertices and of respective radii $\m r$ and $\m r'$ under the condition that $\m r'\ge \m r$.
	This is done in Lemma \ref{lemma_bound_on_volume_alpha} which is the analogue to Lemma \ref{prop_bound_for_couple_of_vertices}.
	\begin{lemma}\label{lemma_bound_on_volume_alpha}
		Let \(\Ball\) (resp. \(\Ball'\)) be a \(d\)-dimensional ball of center \(c\) (resp. \(c'\)), radius \(\m r\) (resp. \(\m r'\)) whose boundary contains the origin \(0\) and the \(d\) points \(\m x_1,...,\m x_k,\m y_1,...,\m y_{d-k}\) (resp. \(\m x_1,...,\m x_k,\m y'_1,...,\m y'_{d-k}\)). Suppose \(\m r'\geq \m r\). Then there exists a positive constant \(L\) depending only on \(d \text{ and }\alpha\) such that 
		\begin{equation}
			e^{-m_\alpha(\Ball\cup\Ball')} \1_{(c,c')\in\left(\VOMR(\lbrace0\rbrace\cup\lbrace\x,\y,\y'\rbrace)\right)^2_{\neq}} \leq e^{-K_{d,\alpha}\m r^{d+\alpha}-L\m r^{d+\alpha-1}\delta} 1_{\delta^2\geq (\m r')^2-\m r^2}
		\end{equation}
		where \(K_{d,\alpha}\) is given by (\ref{eq_m_alpha_volume_unit_ball}).
	\end{lemma}
	\begin{proof}
		As detailed in the proof of Lemma \ref{prop_bound_for_couple_of_vertices}, for the condition \((c,c')\in\left(\VOMR(\lbrace0\rbrace\cup\lbrace\x,\y,\y'\rbrace)\right)^2_{\neq}\) to be satisfied, we need \(\partial\Ball\) to exceed \(\Ball'\) by at least a half-sphere. This leads to the bound (see Figure \ref{fig_vol_bounds_condition_on_delta})
		\begin{equation*}
			\1_{(c,c')\in\left(\VOMR(\lbrace0\rbrace\cup\lbrace\x,\y,\y'\rbrace)\right)^2_{\neq}} \leq 1_{\delta^2\geq (\m r')^2-\m r^2}.
		\end{equation*}
		We now construct a lower bound for \(m_\alpha(\Ball\cup\Ball')\) under the conditions \(\delta^2\geq (\m r')^2-\m r^2\) and \(\m r'\geq \m r\). As in Figure \ref{fig_vol_bounds_alpha}, we partition \(\Ball\cup\Ball'\) into the four disjoint sets \(\Ball_{(\frac{1}{2})}\), \(\Ball'_{(\frac{1}{2})}\), \(\mathcal{A}\) and \(\mathcal{A}'\).
		Then, we bound the \(m_\alpha\)-measure of each of these four sets separately. The measures \(m_\alpha(\Ball_{(\frac{1}{2})})\) and \(m_\alpha(\Ball'_{(\frac{1}{2})})\) are invariant by rotation around the origin of \(\Ball_{(\frac{1}{2})}\) and \(\Ball'_{(\frac{1}{2})}\) respectively, hence 
		\begin{equation*}
		m_\alpha(\Ball_{(\frac{1}{2})})=\frac{1}{2}m_\alpha(\Ball)=\frac{K_{d,\alpha}\m r^{d+\alpha}}{2} \text{ and } m_\alpha(\Ball'_{(\frac{1}{2})})=\frac{K_{d,\alpha}(\m r')^{d+\alpha}}{2}.
		\end{equation*}
		
		\begin{figure}[h!]
			\centering	\includegraphics[scale=0.8]{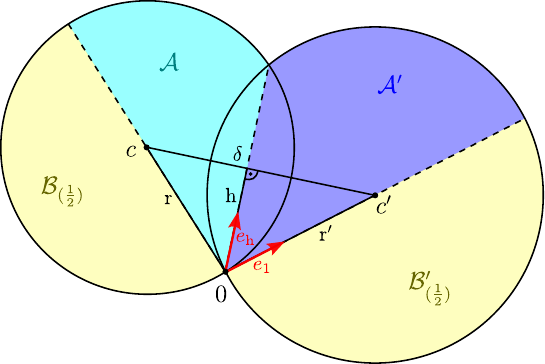}
			\caption{Partition of the set \(\Ball\cup\Ball'\) into the two half balls \(\Ball_{(\frac{1}{2})},\Ball'_{(\frac{1}{2})}\) and the truncated half balls \(\mathcal{A}\) and \(\mathcal{A}'\). The interface between \(\mathcal{A}\) and \(\mathcal{A}'\) is the disk of radius \(h\) with border \(\partial\Ball\cap\partial\Ball'\). In red, the unit vectors \(e_1\) and \(e_{\m h}\) pointing respectively towards the center of \(\Ball\) and the center of \(\partial\Ball\cap\partial\Ball'\).}	\label{fig_vol_bounds_alpha}
		\end{figure}
		
		Moreover, for our purposes we do not need to consider the set \(\mathcal{A}\), so we simply bound \(m_\alpha(\mathcal{A})\geq0\). Finally, we construct the bound for \(m_\alpha(\mathcal{A}')\). We describe \(\mathcal{A}'\) as a half-ball of radius \(\m r'\) with a spherical cap missing. This spherical cap is pinned at \(0\) and has radius
		\begin{equation}\label{eq_value_of_h_alpha_case}
			\m h = \frac{\sqrt{2(\m r^2(\m r')^2+\delta^2(\m r')^2+\delta^2\m r^2)-(\delta^4+\m r^4+(\m r')^4)}}{2\delta} \in [0,\m r]
		\end{equation}
		where the domain of \(\m h\) is a consequence of the conditions \(\m r\leq \m r'\) and \(\delta^2\geq(\m r')^2-\m r^2\), see Figure \ref{fig_vol_bounds_condition_on_delta}. Let us write an explicit integral expression for \(m_\alpha(\mathcal{A'})\). As in Figure \ref{fig_integral_m_alpha_A}, we choose an orthogonal coordinate system  of \(\R^d\)  \((e_1,...,e_d)\) such that the disk of radius \(\m r'\) supporting the base of \(\mathcal{A}'\) has center \(\|c'\|e_1\) and is included in the plane \(e_2^\perp\) and that the disk of radius \(\m h\) supporting the cap is included in the hyperplane \(\text{Vect}_{\R}\left(e_h,e_3,...,e_d\right)\), where \(e_{\m h}=\frac{\m h}{\m r'}e_1+\sqrt{1-\left(\frac{\m h}{\m r'}\right)^2}e_2\). Denoting by \(\m z=(\m z_1,...,\m z_d)\) the variable of integration in this coordinate system and by \(d\m z\) the Lebesgue measure on \(\R^d\), we have \(m_\alpha(\mathcal{A'})=\int_{\mathcal{A}'}d\m z \|\m z\|^\alpha\). We want to perform this integration on the sets \(D_\varphi\), i.e. the disks of radius \(r'\cos(\varphi)\) determined by the intersection of \(\mathcal{A}'\) with the hyperplane perpendicular to \(-\cos(\varphi)e_1+\sin(\varphi)e_2\) for \(\varphi \in [0,\arccos(\frac{\m h}{\m r'})]\), see Figure \ref{fig_integral_m_alpha_A}. 
		\begin{figure}[h!]
			\centering
			\includegraphics[scale=0.8]{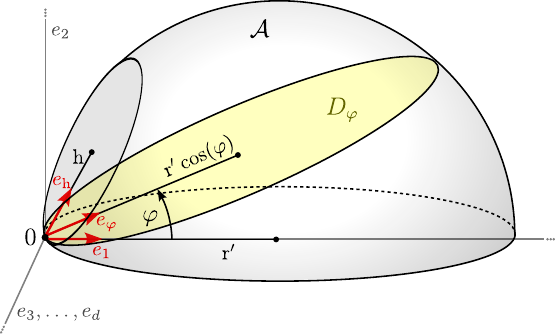}
			\caption{Coordinate system used to obtain the integral expression for \(m_\alpha(\mathcal{A}')\) of equation (\ref{eq_integral_expression_m_alpha_A'}).}
			\label{fig_integral_m_alpha_A}
		\end{figure}
		Therefore, we introduce a polar change of variables on the \((\m z_1,\m z_2)\), i.e. we introduce variables \(\varphi=\arctan\left(\frac{\m z_2}{\m z_1}\right)\) and \((\m w_1,...,\m w_{d-1})=(\sqrt{\m z_1^2+\m z_2^2},\m z_3,...,\m z_d)\), so that for fixed \(\varphi\), \(\m w = (\m w_1,...,\m w_{d-1})\) represent the local coordinates on \(D_\varphi\) with respect to the coordinate system \((\tilde{e}=\cos(\varphi)e_1+\sin(\varphi)e_2,e_3,...,e_d)\), see Figure  \ref{fig_integral_m_alpha_A}. The Jacobian for this change of variables is
		\begin{equation*}
			d\m z_1 ... d\m z_d = d\varphi d\m w_1 ... d\m w_{d-1} \m w_1.
		\end{equation*}
		The domain of integration of \(\varphi\) is the interval \(\left[0,\arccos\left(\frac{\m h}{\m r'}\right)\right]\). Moreover \(\|\m z\|=\|\m w\| = \sqrt{\m w_1^2+...\m w_{d-1}^2}\). Therefore,
		\begin{equation}\label{eq_integral_expression_m_alpha_A'}
			m_\alpha(\mathcal{A}') = \int_{\mathcal{A}'}d\m z \|\m z\|^\alpha = \int\limits_{0}^{\arccos\left(\frac{\m h}{\m r'}\right)} d\varphi \, \int\limits_{D_\varphi} d\m w \, \|\m w\|^\alpha \m w_1.
		\end{equation}
		We introduce the normalized variables \(\Tilde{\m w}_i = \frac{\m w_i}{\m r'\cos(\varphi)}\) for \(i=1,...,d-1\): this choice makes the domain of integration of \(\Tilde{\m w}=(\Tilde{\m w}_1,..., \Tilde{\m w}_{d-1})\) independent from \(\varphi\). Indeed, after this change of variables,
		\begin{align}\label{eq_bound_on_m_A'}
		    m_\alpha(\mathcal{A}') & = \int\limits_{0}^{\arccos\left(\frac{\m h}{\m r'}\right)} d\varphi (\m r'\cos(\varphi))^{d+\alpha}\, \int\limits_{\Ball^{(d-1)}_1(\Tilde{e})} d\Tilde{\m w} \, \|\Tilde{\m w}\|^\alpha \Tilde{\m w}_1 \nonumber
		    \\
		    & \geq C (\m r')^{d+\alpha} \int\limits_{0}^{\arccos\left(\frac{\m h}{\m r'}\right)} d\varphi (\cos(\varphi))^{d+\alpha} \nonumber
		    \\
		    & \geq C (\m r')^{d+\alpha} \sqrt{1-\frac{\m h}{\m r'}}
		\end{align}
		where \(C\) denotes again a generic positive constant that can change value from line to line. Concerning the square root factor in the right-hand side of \eqref{eq_bound_on_m_A'}, we use \(0\leq\frac{\m h}{\m r'}\leq\frac{\m r}{\m r'}\leq 1\) to remark that
		\begin{equation}\label{eq_94}
		    \sqrt{1-\frac{\m h}{\m r'}} 
			=
			\frac{1}{\sqrt{1+\frac{\m h}{\m r'}}}
			\sqrt{1-\left(\frac{\m h}{\m r'}\right)^2} 
			\geq \frac{1}{\sqrt{2}} \sqrt{1-\left(\frac{\m h}{\m r'}\right)^2}.
		\end{equation}
		We use the identity (\ref{eq_value_of_h_alpha_case}) to express the variable \(\m h\) in the square root factor in the right-hand side of \eqref{eq_94} with respect to \(\delta\) and obtain
		\begin{align}\label{eq_bound_for_root}
			\sqrt{1-\left(\frac{\m h}{\m r'}\right)^2} &= \sqrt{\frac{(\m r')^4+\m r^4+\delta^4-2(\m r')^2\m r^2+2(\m r')^2\delta^2-2\m r^2\delta^2}{4(\m r')^2\delta^2}} \nonumber \\
			& = \sqrt{\frac{\left((\m r')^2+\delta^2-\m r^2\right)^2}{4\delta^2(\m r')^2}} = \frac{\m r'}{2\delta} \left(1+\left(\frac{\delta}{\m r'}\right)^2-\left(\frac{\m r}{\m r'}\right)^2\right)\geq\frac{\delta}{2\m r'}
		\end{align}
		where the last inequality holds because \(\m r' \geq \m r\). Merging \eqref{eq_bound_for_root} and \eqref{eq_94} and inserting the resulting bound into \eqref{eq_bound_on_m_A'} shows that there is a positive constant \(L\) such that \(m_\alpha(\mathcal{A}')\geq L (\m r')^{d+\alpha-1}\delta\). Finally, combining the bounds for the \(m_\alpha\)-measures of \(\Ball_{(\frac{1}{2})}\), \(\Ball'_{(\frac{1}{2})}\), \(\mathcal{A}\) and \(\mathcal{A}'\) presented so far and using \(\m r'\geq \m r\) we conclude
		\begin{align*}
		    m_\alpha(\Ball\cup\Ball') & = m_\alpha(\Ball_{(\frac{1}{2})}) + m_\alpha(\Ball'_{(\frac{1}{2})}) + m_\alpha(\mathcal{A}) + m_\alpha(\mathcal{A}') \\
		    & \geq \frac{K_{d,\alpha}\m r^{d+\alpha}}{2} + \frac{K_{d,\alpha}(\m r')^{d+\alpha}}{2} +0 + L (\m r')^{d+\alpha-1}\delta \geq K_{d,\alpha}\m r^{d+\alpha} + L \m r^{d+\alpha-1}\delta.
		\end{align*}
	\end{proof}
	After fixing \(\m r'\geq \m r\), we need to bound the norms of $\m x_i$, $\m y_j$ and $\m y_j'$ in \eqref{eq_after_Nikitenko_alpha}, i.e.
	\begin{align}\label{eq_bound_on_norm_x_y_y}
		\|\m x_i\| \leq 2 \rho_0 \text{ for } i=1,...,k , \|\m y_j\|\leq 2 \m r \text{ and }\|\m y_j'\|\leq 2 \m r' \text{ for } j=1,...,d-k.
	\end{align}
	Inserting (\ref{eq_bound_on_norm_x_y_y})  and the bound of Lemma \ref{lemma_bound_on_volume_alpha} into (\ref{eq_after_Nikitenko_alpha}), we proceed as in the end of Step 3 of Section \ref{sec_expected_nr_pairs}. That is, we bound the volume of the simplices by a constant, then we integrate the variables \(\m u_0,\vbf,\wbf,\wbf'\) and finally we fix wlog \(\m Q=\m Q_0=\text{Vect}_{\R}(e_1,...,e_k)\) and \(\m u'=e_d\) to obtain
	\begin{align}\label{eq_bound_with_only_u_rho_r_r'_alpha}
		\E \left[ \Card \left((\VOMR)^2_{\neq}\right) \right] =  C \sum\limits_{k=0}^{d-1} \int\limits_{t}^{\infty} d\m r \int\limits_{\m r}^{\infty} d\m r' \int\limits_{0}^{\m r} d\rho_0 \int\limits_{\Sp_{\m Q_0^\perp}} d\m u \, & \rho_0^{k(d+\alpha)-1} (\m r \m r')^{(d-k)(d+\alpha-1)+1} \left((\m r^2-\rho_0^2)((\m r')^2-\rho_0^2)\right)^{\frac{d-k-2}{2}} \nonumber
		\\
		& \hspace{3cm}   e^{-K_{\alpha,d}\m r^{d+\alpha}-L\m r^{d+\alpha-1}\delta} 1_{\delta^2\geq (\m r')^2-\m r^2}
	\end{align}
	where, as before, \(C\) denotes a generic positive constant and \(\delta=\|c-c'\|\) is given by (\ref{eq_definition_delta_before_step_4}).
	The analytical treatment of (\ref{eq_bound_with_only_u_rho_r_r'_alpha}) is precisely the same as in Steps 4-6 of Section \ref{sec_expected_nr_pairs} and leads to the bound
	\begin{equation*}
		\E\left[\Card\left((\VOMR(\Phi_\alpha))^2_{\neq}\right)\right] = \mathcal{O}\left( t^{d^2+d(\alpha-2)+\alpha-2} e^{K_{d,\alpha}t^{d+\alpha}}\right) = o\left(\E\left[\Card\left(\VOMR(\Phi_\alpha)\right)\right]\right) \text{ as } t\rightarrow\infty.
	\end{equation*}

        ~\\
        Pierre Calka, Univ Rouen Normandie, CNRS, Normandie Univ, LMRS UMR 6085, Rouen, F-76000, France;\quad pierre.calka@univ-rouen.fr\\~\\
	Cecilia D'Errico, Laboratoire de Math\'ematiques d'Orsay, Universit\'e Paris-Saclay, B\^atiment 307, F-91405 Orsay, France;\quad cecilia.derrico@universite-paris-saclay.fr\\~\\
        Nathana\"el Enriquez, Laboratoire de Math\'ematiques d'Orsay, Universit\'e Paris-Saclay, B\^atiment 307, F-91405 Orsay, France;\quad nathanael.enriquez@universite-paris-saclay.fr
\end{document}